\documentclass[12pt]{amsart}
\usepackage[utf8]{inputenc}
\usepackage{amssymb}
\usepackage{amsbsy}
\usepackage{amscd}
\usepackage{imakeidx}

\usepackage{mathrsfs}
\usepackage{eucal}
\usepackage{nicefrac}
\usepackage{diagbox}
\usepackage{array}\newcolumntype{C}[1]{>{\centering\arraybackslash}p{#1}}
\usepackage{makecell}

\DeclareSymbolFont{rsfs}{U}{rsfs}{m}{n}
\DeclareSymbolFontAlphabet{\mathrsfs}{rsfs}

\usepackage{a4wide}
\usepackage[boxsize=1.25em,centerboxes]{ytableau}
\usepackage{mathdots}
\usepackage{tikz-cd}
\usepackage{tikz}
\usetikzlibrary{arrows,decorations.pathmorphing,backgrounds,positioning,fit,petri}
\usetikzlibrary{decorations.markings}
\usetikzlibrary{patterns,calc}
\usetikzlibrary{matrix,positioning,decorations.pathreplacing}
\usetikzlibrary{shapes.geometric}
\usetikzlibrary{arrows.meta,decorations.pathreplacing}
\usepackage{genyoungtabtikz}

\usepackage{verbatim}
\usepackage{version}
\usepackage{color}
\definecolor{darkspringgreen}{rgb}{0.09, 0.45, 0.27}
\definecolor{deepjunglegreen}{rgb}{0.0, 0.29, 0.29}
\definecolor{oldgold}{rgb}{0.81, 0.71, 0.23}
\definecolor{pastelgreen}{rgb}{0.47, 0.87, 0.47}
\newenvironment{NB}{
\color{red}{\bf NB}. \footnotesize
}{}
\newenvironment{NB2}{
\color{blue}{\bf NB2}. \footnotesize
}{}
\excludeversion{NB}
\excludeversion{NB2}
\excludeversion{NB3}
\excludeversion{NB4}
\makeatletter

\usepackage{url}
\usepackage{ifmtarg}

\usepackage{xr-hyper}
\usepackage[
pdftex,
backref=page,
bookmarks=true,
colorlinks=true,
citecolor=deepjunglegreen,
filecolor=deepjunglegreen,
debug=true,
unicode,
pdfnewwindow=true]{hyperref}

\usepackage{cleveref}

\newcommand\fixcref[1]{%
  \AddToHook{env/#1/begin}{\crefalias{equation}{#1}}%
  \crefname{#1}{#1}{#1s}%
  \crefformat{#1}{##2~##1##3}%
  \crefrangeformat{#1}{#1s~##3##1##4 to~##5##2##6}%
  \crefmultiformat{#1}{#1~##2##1##3}{ and~##2##1##3}{, ##2##1##3}{ and~##2##1##3}%
}

\fixcref{Theorem}
\fixcref{Lemma}
\fixcref{Proposition}
\fixcref{Colloary}
\fixcref{Definition}
\fixcref{Remark}
\fixcref{Conjecture}
\fixcref{figure}
\fixcref{Example}

\crefname{Theorem}{Theorem}{Theorems}
\crefformat{Theorem}{Theorem~#2#1#3}
\crefname{section}{\S}{\S\S}
\crefformat{section}{\S#2#1#3} 
\crefmultiformat{section}{\S\S#2#1#3}{, #2#1#3}{, #2#1#3}{, #2#1#3}
\crefname{Lemma}{Lemma}{Lemmas}
\crefformat{Lemma}{Lemma~#2#1#3}
\crefname{Proposition}{Proposition}{Propositions}
\crefformat{Proposition}{Proposition~#2#1#3}
\crefname{Corollary}{Corollary}{Corollaries}
\crefformat{Corollary}{Corollary~#2#1#3}
\crefname{Definition}{Definition}{Definitions}
\crefformat{Definition}{Definition~#2#1#3}
\crefname{Remark}{Remark}{Remarks}
\crefformat{Remark}{Remark~#2#1#3}
\crefname{Remarks}{Remark}{Remarks}
\crefformat{Remarks}{Remark~#2#1#3}
\crefname{Conjecture}{Conjecture}{Conjectures}
\crefformat{Conjecture}{Conjecture~#2#1#3}
\crefname{figure}{Figure}{Figure}
\crefformat{figure}{Figure~#2#1#3}
\crefname{Example}{Example}{Examples}
\crefformat{Example}{Example~#2#1#3}
\crefname{Fact}{Fact}{Facts}
\crefformat{Fact}{Fact~#2#1#3}

\crefname{appendix}{Appendix}{Appendices}
\crefformat{appendix}{\S#2#1#3}
\crefformat{enumi}{#2#1#3}

\crefname{equation}{}{}

\crefname{table}{Table}{Table}
\crefformat{table}{Table~#2#1#3}

\crefname{paragraph}{}{}
\crefformat{paragraph}{#2#1#3}

\usepackage{stmaryrd}
\usepackage{enumitem}
\usepackage{mathtools}
\usepackage{multirow}
\usepackage{blkarray}

\usepackage{footnote}
\makesavenoteenv{tabular}
\makesavenoteenv{table}
\usepackage{bm}

\usepackage[normalem]{ulem}
\usepackage{xifthen}

\renewcommand{\thesubsection}{\thesection(\@roman\c@subsection)}
\makeatletter
\newcounter{number}
\makeatother
\newtheorem{Theorem}[equation]{Theorem}
\newtheorem{Corollary}[equation]{Corollary}
\newtheorem{Lemma}[equation]{Lemma}
\newtheorem{Proposition}[equation]{Proposition}

\newtheorem{Fact}[equation]{Fact}

\theoremstyle{definition}
\newtheorem{Definition}[equation]{Definition}
\newtheorem{Example}[equation]{Example}

\theoremstyle{remark}
\newtheorem{Remark}[equation]{Remark}

\newtheorem{Assumption}[equation]{Assumption}

\numberwithin{equation}{section}

\newcommand{\lsp}[2]{\prescript{#1}{}{#2}}
\newcommand{\defeq}{\overset{\operatorname{\scriptstyle def.}}{=}}
\newcommand{\CC}{{\mathbb C}}

\newcommand{\ZZ}{{\mathbb Z}}

\newcommand{\RR}{{\mathbb R}}
\newcommand{\proj}{{\mathbb P}}

\newcommand{\SL}{\operatorname{\rm SL}}
\newcommand{\SU}{\operatorname{\rm SU}}
\newcommand{\GL}{\operatorname{GL}}

\providecommand{\U}{}
\renewcommand{\U}{\operatorname{\rm U}}
\newcommand{\SO}{\operatorname{\rm SO}}

\newcommand{\grpSp}{\operatorname{\rm Sp}}
\newcommand{\grpO}{\operatorname{\rm O}}
\newcommand{\algsl}{\operatorname{\mathfrak{sl}}} 

\newcommand{\gl}{\operatorname{\mathfrak{gl}}}

\newcommand{\so}{\operatorname{\mathfrak{so}}}
\newcommand{\algsp}{\operatorname{\mathfrak{sp}}}

\newcommand{\End}{\operatorname{End}}
\newcommand{\Hom}{\operatorname{Hom}}
\newcommand{\Ext}{\operatorname{Ext}}
\newcommand{\Ker}{\operatorname{Ker}}

\newcommand{\Ima}{\operatorname{Im}}

\newcommand{\Sym}{\operatorname{Sym}}

\newcommand{\codim}{\mathop{\text{\rm codim}}\nolimits}

\newcommand{\id}{\operatorname{id}}
\newcommand{\ve}{\varepsilon}
\renewcommand{\MR}[1]{}

\newcommand{\Wedge}{{\textstyle \bigwedge}}

\newcommand{\vin}[1]{\operatorname{i}(#1)} 
\newcommand{\vout}[1]{\operatorname{o}(#1)} 
\newcommand{\bM}{\mathbf M}

\newcommand{\shfO}{\mathcal O}
\newcommand{\tslash}{/\!\!/\!\!/}
\newcommand{\tslabar}{\mathbin{
    \setbox0=\hbox{/\!\!/\!\!/}\rule[0.4\ht0]{\wd0}{.3\dp0}\kern-\wd0\box0}}

\newcommand{\scH}{{\mathscr H}}
\newcommand{\scR}{{\mathscr R}}

\makeatletter
\newcommand{\cA}[1][{}]{%
  \@ifmtarg{#1}%
  {\mathcal A}
  {\mathcal A(#1)}
}
\newcommand{\cAh}[1][{}]{%
  \@ifmtarg{#1}%
  {\mathcal A_\hbar}
  {\mathcal A_\hbar(#1)}
}
\makeatother

\newcommand{\Stab}{\operatorname{Stab}}
\newcommand{\Lie}{\operatorname{Lie}}

\newcommand{\gr}{\operatorname{gr}}

\newcommand{\po}{\ar@{}[dr]|{\text{\pigpenfont R}}}
\newcommand{\pb}{\ar@{}[dr]|{\text{\pigpenfont J}}}
\newcommand{\pp}{\ar@{}[dr]|{\text{\pigpenfont P}}}

\newcommand{\BZ}{{\mathbb{Z}}}

\newcommand{\bq}{{\mathbf{q}}}

\newcommand{\bff}{{\mathbf{f}}}

\newcommand{\CU}{{\mathcal{U}}}
\newcommand{\CV}{{\mathcal{V}}}
\newcommand{\CW}{{\mathcal{W}}}

\newcommand{\fg}{{\mathfrak{g}}}
\newcommand{\fh}{{\mathfrak{h}}}
\newcommand{\fH}{{\mathfrak{H}}}

\newcommand{\fL}{{\mathfrak L}}
\newcommand{\fM}{{\mathfrak{M}}}

\newcommand{\sr}{{\mathsf{r}}}

\newcommand{\sA}{{\mathsf{A}}}
\newcommand{\sS}{{\mathsf{S}}}
\newcommand{\sT}{{\mathsf{T}}}
\newcommand{\sX}{{\mathsf{X}}}
\newcommand{\sY}{{\mathsf{Y}}}

\newcommand{\bv}{{\mathbf v}} 
\newcommand{\bw}{{\mathbf w}} 
\newcommand{\bC}{{\mathbf C}} 
\newcommand{\bc}{{\mathbf c}} 
\newcommand{\bU}{{\mathbf U}} 

\DeclareSymbolFont{symbolsC}{U}{pxsyc}{m}{n}
\SetSymbolFont{symbolsC}{bold}{U}{pxsyc}{bx}{n}
\DeclareFontSubstitution{U}{pxsyc}{m}{n}
\DeclareMathSymbol{\medcirc}{\mathbin}{symbolsC}{7}

\newcommand{\ch}{\operatorname{ch}}

\tikzset{SO/.style={draw,
    minimum size=12pt,inner sep=0pt, outer sep=0pt,font={$+$}}}
\tikzset{Sp/.style={draw,
    minimum size=12pt,inner sep=0pt, outer sep=0pt,font={$-$}}}
\tikzset{GL/.style={draw,
    minimum size=12pt,inner sep=0pt, outer sep=0pt,
    font={$\phantom{-}$}}}
\tikzset{pm/.style={draw,
    minimum size=12pt,inner sep=0pt, outer sep=0pt,font={$\pm$}}}
\tikzset{mp/.style={draw,
    minimum size=12pt,inner sep=0pt, outer sep=0pt,font={$\mp$}}}

\newcommand{\invast}[1][]{\ifthenelse{\isempty{#1}}{\star}{{#1}^\star}}

\newcommand{\TT}{\mathbb T}

\newtagform{plus}{(}{)${}^+$}

\title[Involutions on quiver varieties and quantum symmetric pairs]
{Instantons on ALE spaces for classical groups, involutions on
  quiver varieties, and quantum symmetric pairs}
\dedicatory{Dedicated to the memory of Masatoshi Noumi}

\author{Hiraku Nakajima}

\begin{document}

\begin{abstract}
Moduli spaces of instantons on ALE spaces for classical groups are examples of fixed point sets of involutions on quiver varieties,
i.e., $\sigma$-quiver varieties.
In 2018 Yiqiang Li considered their equivariant cohomology, and by stable envelope of Maulik-Okounkov, constructed representations of coideal subalgebras of Maulik-Okounkov Yangian, 
called twisted Yangian.
We calculate $K$-matrices as matrices in examples, identified the twisted Yangians with
ones studied in other literature,
and clarify conditions which we should impose to make them well-defined.
\end{abstract}
\maketitle

\section{Introduction}

In 1992, the author \cite{Na-quiver} discovered a relation between homology of moduli spaces of
instantons on ALE spaces and representations of affine Kac-Moody Lie algebras.\footnote{
  Announced in \cite{Hyogen1992}.
}
There have been many developments since then, extensions both in geometry and representation theory sides. 
One of applications of these developments was an explicit character formula
of arbitrary simple modules of quantum loop algebras \cite{MR2144973}.
See \cite[Appendix~B]{2019arXiv190706552N} for reduction of non-simply laced cases to simply laced cases.
The formula itself, like in the case of Kazhdan-Lusztig polynomials, involves
only combinatorial inputs, but its proof used geometry.
No alternative proof is known up to this day.

Most of developments are restricted to the case when the gauge group is unitary,
or moduli spaces of coherent sheaves, quiver varieties on geometric side,
and quantum loop algebras or Yangians on representation theory side.
The focus of this paper is to study examples when the gauge group is a classical group,
i.e., a special orthogonal or symplectic group.
%
Let us recall sporadic earlier results in this direction.

When the underlying ALE space is $\RR^4 = \CC^2$, 
Braverman, Finkelberg and the author equipped 
the equivariant intersection cohomology of the moduli space of
$G$-instantons
with a structure of a representation of the affine $\mathscr W$-algebra
associated with $\fg = \operatorname{Lie}G$ \cite{2014arXiv1406.2381B}.
Here the moduli space is, more precisely, Uhlenbeck partial compactification of the
\emph{genuine} moduli space, which has singularities. And $G$ could
be even an exceptional group, but assumed to be of type ADE.
This construction was motivated by the extension of the Alday-Gaiotto-Tachikawa correspondence
\cite{AGT} to arbitrary $G$ \cite{ABCDEFG}.
When the underlying ALE space is of type A, it is expected that the
same result holds where the affine $\mathscr W$-algebra is replaced by
the coset vertex algebra \cite{PhysRevD.84.046009}. 
But it is not clear
what is the corresponding algebra for ALE spaces of type D, E.

In 2018, Yiqiang Li \cite{MR3900699} considered an involution on
a quiver variety of finite type, which is a moduli space of
unitary instantons on ALE spaces by \cite{KN}. He called
the fixed point set a \emph{$\sigma$-quiver variety}, and equipped
its equivariant cohomology with a structure of a representation of a coideal subalgebra $\sX^{\mathrm{tw}}$ of the Maulik-Okounkov Yangian $\sX$ of the corresponding quiver variety.
(We will show that Maulik-Okounkov Yangian is isomorphic to the extended Yangian \cite{MR3849990} associated with
the corresponding simple Lie algebra.)
An example of involutions was given earlier in \cite{Na-reflect}, and
the fixed points set is the moduli space of orthogonal or symplectic instantons on ALE spaces, 
but Li considered more general involutions. 
We will work in his generality in this paper, hence study $\sigma$-quiver varieties.
The type of a $\sigma$-quiver variety is classified by (i) the type of the underlying Dynkin graph,
(ii) a choice of a diagram involution $\sigma'$, possibly trivial, and (iii) a choice of the type $(+)$ or $(-)$.
Moduli spaces of $\SO$ and $\grpSp$ instantons correspond to the case when $\sigma' = \id$, and $(+)$ or $(-)$ is $\SO$ and $\grpSp$ respectively.

However, one should keep in mind that we assume $\sigma$-quiver varieties are \emph{smooth},
hence the classification above is not complete.
We exclude examples in \cite{2014arXiv1406.2381B}, as they are singular.
Therefore, generalization of this paper to singular $\sigma$-quiver varieties is an important future problem.
For instantons on $\RR^4$, the deformed affine $\mathscr W$-algebras of classical type
are related to a coideal subalgebra of the quantum toroidal algebra
type $\mathbf U_q(\mathbf L\hat{\mathfrak{gl}}_1)$ \cite{MR4274696}.
This result suggests that our construction should have $K$-theoretic analogue,
even though we need to deal with singular moduli spaces.

Li's construction is based on the stable envelope of Maulik-Okounkov \cite{MR3951025}.
The stable envelope gives a geometric construction of the $R$-matrix on
the equivariant cohomology of quiver varieties, and the Yangian $\sX$ was
constructed by the $RTT$ construction \cite{MR1015339}.
Similarly, Li constructed $K$-matrices using stable envelope,
and the coideal subalgebra $\sX^{\mathrm{tw}}$ is given by matrix coefficients of $K$-matrices.
This construction is similar to one for twisted Yangian by Olshanski, Molev-Ragoucy, MacKay and others.
(See the textbook by Molev \cite{MR2355506} for the construction.)
Note that coideal subalgebras of quantum groups have been studied
in the context of quantum symmetric pairs, which are quantizations of symmetric pairs, i.e.,
pairs of Lie algebras and their fixed point subalgebras under an involution.
Twisted Yangians are quantization of the fixed point subalgebras of the current algebras
$\fg[z]$ of complex simple Lie algebras $\fg$.
The involution is in the form $x u^k\mapsto \sigma(x)(-u)^k$ for $x\in\fg$, $k\ge 0$,
where $\sigma$ is an involution on $\fg$.
See \cite{MR3615052} and references therein.

Li's construction is \emph{abstract}, in the sense that $K$-matrices are defined as linear operators
on the localized equivariant cohomology of $\sigma$-quiver varieties.
In a sense, the $K$-matrices should be called $K$-\emph{linear maps}.
See \cref{rem:explicitR}. 
In order to compare his twisted Yangian with the one in other literature,
we need to calculate $K$-matrices explicitly as \emph{matrices} 
after choosing bases of the localized equivariant cohomology.
This is the main purpose of this paper.
We will find that Li's construction indeed give twisted Yangians in the literature,
Olshanski twisted Yangian and Molev-Ragoucy reflection equation algebra.
See \cref{thm:O_twisted,thm:MR_refl}.
%
The Lie algebra $\fg$ is determined by (i), and the involution $\sigma$ on $\fg$
is the composite of the Chevalley involution of $\fg$ and $\sigma'$. Both are
independent of (iii). See \cref{lem:sigma_on_g}.
The embedding $\sX^{\mathrm{tw}}\to\sX$ depends on the data (iii) in general.

In the course of our study, we found that an assumption in Li's paper,
$a(W^2) = W^2$ in \cite[\S5.2, the first paragraph]{MR3900699},
is unnecessary.
This is important for our purpose, as moduli spaces of orthogonal and symplectic instantons
do not satisfy this condition in general. See \cref{sec:ex1}.

We also find that we need to impose a compatibility condition on the polarization, which is
a choice of a square root of the Euler class of the normal bundle of the fixed point set
of a torus action, in order to make $K$-matrices well-defined. See \cref{subsec:induced-polarization,subsec:K-matrix}
and Assumption~\ref{assum:induced-polarization}.
This is a \emph{new} phenomenon, not seen in the case of $R$-matrices from quiver varieties.
It is an interesting question to understand the homology group when this condition is \emph{not} satisfied.

We find a polarization satisfying the compatibility condition in
type A, D or $\mathrm{E}_6$ with $\sigma' = \id$, $(+)$ type, and
$\mathrm{A}_1$ with $\sigma'=\id$, $(-)$ type.
We also find a polarization in type A with $\sigma'\neq\id$ for both $(+)$ and $(-)$ types.
For type $\mathrm{A}_{\ell-1}$ ($\ell>2$), $\sigma' = \id$ and $(-)$ type, we show that there is no choice of polarization
satisfying the compatibility condition.
In the remaining case, we do not know whether a compatible polarization exists or not.

Except Li's assumption mentioned above, and the compatibility condition on the polarization,
the only new part in this paper is calculation of $K$-matrices in examples.


There are earlier works on geometric realization of twisted Yangians and
coideal subalgebras of quantum groups 
\cite{MR4329193,MR4717572,2024arXiv240706865S},
and references therein.
Their approaches are different from Li's one, hence from ours.
They are based on explicit presentations of coideal subalgebras, given
under $\iota$quantum group program \cite{MR4680353}. In particular,
the larger algebra, containing the coideal subalgebra, 
does \emph{not} play any role in their approaches.\footnote{
  In \cite{MR4717572} it was shown that equivariant cohomology in these cases
  are representations of the pre-twisted Yangian, defined by generators and relations.
  We show that the same cohomology groups are representations of
  the reflection equation algebra \cite[\S2.16 Example~4]{MR2355506}. But,
  relation between pre-twisted Yangian and reflection equation algebra
  is not known to the author.}
\begin{NB}
\begin{equation*}
  \iota\text{quantum group} =
  \text{the coideal subalgebra in quantum symmetric pair}
\end{equation*}
\end{NB}%
Moreover, all these papers use cotangent bundles of partial
flag varieties, which are special cases of $\sigma$-quiver varieties,
studied in \cref{sec:ex2}.
However, they were studied in the spirit of $\iota$quantum group program,
namely relation between $\sigma$-quiver varieties and quiver varieties 
played no direct role in their works.
Also, our approach can be applied to more examples, even though we still impose the strong assumption
that $\sigma$-quiver varieties are smooth.

There is also another work \cite{2025arXiv250106643D}.
It is an orthosymplectic analogue of the cohomological Hall algebra approach to the upper half subalgebra
of the Yangian \cite{MR3018956,MR2851153}. For the Yangian, cohomological Hall algebra has a 
representation on the equivariant cohomology of moduli spaces of unitary instantons or quiver varieties
\cite{MR3805051,MR4069884}.
One of key properties of relation between two constructions was the claim that a quiver variety is an open
subset of the moduli \emph{stack} of representations of the preprojective algebra.
This claim has no obvious analogue for $\sigma$-quiver varieties.
On representation theory side, it is not clear what is the upper half subalgebra of the twisted Yangian.
Hence the relation between two constructions is not clear.


Another, more technical motivation of the author to study moduli spaces of orthogonal and symplectic instantons
is an issue in the definition of de Campos Affonso’s symmetric bow varieties \cite{Henrique}.
They were introduced as candidates of Coulomb branches of quiver gauge theories
of type B, C, D, but do not have expected properties in general.
Our study of moduli spaces resolves the issue.
We will mention relation to Coulomb branches at several places in the main text.

As we mentioned above, geometric realization of Yangian and quantum loop algebras
has applications to representation theory. 
We expect that our result will have similar applications to representation theory of
twisted Yangians. Many ingredients used in \cite{MR2144973} have analogues in our setting.
For example, $t$-analogue of Knight's spectral characters (aka $q$-characters) of standard modules
\cite[\S13.5]{Na-qaff} have analogue as generating functions of
Poincar\'e polynomials of the fixed point sets of a $\CC^\times$-action on
$\sigma$-quiver varieties.
One of important tools, not discussed in this paper, is the Grassmann bundle 
in \cite[\S4]{Na-alg}, \cite[\S5.4]{Na-qaff}.
It is an interesting direction to pursue, as the Grassmann bundle
is an essential ingredient of geometric realization of Kashiwara crystals.
See e.g., \cite[\S4]{Na-Tensor} for a derivation of the tensor product rule.

We focus on the twisted Yangian in this paper, and leave modification to quantum coideal
subalgebras to future works. We need to replace equivariant cohomology by equivariant $K$-theory.
If we restrict ourselves to smooth $\sigma$-quiver varieties as in this paper, the only issue
is to clarify the condition on the polarization. If we want to deal with more examples of 
quantum symmetric pairs, we need to consider singular $\sigma$-quiver varieties.
Then the definition of $K$-theoretic stable envelope should be more involved.

The paper is organized as follows. 
In \cref{sec:ext_Yangian} we introduce the extended Yangian $\sX$ following
\cite{MR3849990}.
In \cref{sec:inv-quiver} we introduce an involution on a quiver variety
following \cite{MR3900699}. We consider the involution for both generic
stability parameter and zero stability parameter, although the latter
one plays no significant role in this paper. They give rise examples
of gauge theories, whose Coulomb branches are expected to be interesting.
See \cref{subsec:cat_quotient}.
In \cref{sec:stable-envelope} we introduce the stable envelope \cite{MR3951025}
and show that $\sX$ is isomorphic to one given by the stable envelope.
Contrary to other parts of this paper, we also discuss the case of affine quivers.
See \cref{subsec:affine}.
In \cref{sec:twisted-yangian} we introduce $K$-matrices and define
the twisted Yangian $\sX^{\mathrm{tw}}$ following \cite{MR3900699}.
In \cref{sec:ex1} we study the $K$-matrices for moduli spaces of $\SO$ and $\grpSp$ instantons
on ALE spaces in type A in detail. 
%
%
We observe that there is
no polarization satisfying the condition~\ref{assum:induced-polarization}
for $\grpSp$-instantons. On the other hand, the twisted Yangian is
Olshanski twisted Yangian for $\SO$-instantons.
In \cref{sec:ex2} we study the $K$-matrices for cotangent bundles of partial flag varieties
of type BCD. 
This is the case when $\sigma'\neq \id$.
We show that the twisted Yangian is Molev-Ragoucy reflection equation algebra.
In \cref{sec:ADHM} we summarize the ADHM description of classical group instantons
on ALE spaces.

\subsection*{Convention}\label{subsec:convention}

\subsubsection*{Transpose}

Suppose $C\colon V_1\to V_2$ is a linear map from a finite dimensional
complex vector space $V_1$ to another space $V_2$. Let $V_1^*$,
$V_2^*$ denote the dual spaces of $V_1$, $V_2$. The transpose
$\lsp{t}C\colon V_2^*\to V_1^*$ of $C$ is defined by
$\langle C e, f\rangle = \langle e, \lsp{t}C f\rangle$ for
$e\in V_1$, $f\in V_2^*$. We have $\lsp{t}(\lsp{t}C) = C$ under
the natural isomorphism $V_1^{**} \cong V_1$, $V_2^{**} \cong V_2$.

\subsubsection*{Adjoint}

Let $\ve$ be $1$ or $-1$. An \emph{$\ve$-form} on a finite dimensional
complex vector space $V$ is a non-degenerate form $(\ ,\ )$ such that
$(x,y) = \ve(y,x)$ for $x$, $y\in V$. Namely it is an orthogonal form
if $\ve=1$, and a symplectic form if $\ve = -1$. Throughout this
paper, an \emph{orthogonal} form means a non-degenerate symmetric
form.

Let $G_\ve(V)$ denote the subgroup of $\GL(V)$ preserving the
$\ve$-form. Namely $G_\ve(V) = \grpO(V)$ if $\ve=1$, $\grpSp(V)$ if $\ve=-1$.
We denote it simply by $G(V)$ if the $\ve$-form on $V$ is clear
from the context.
We denote the connected component of $G(V)$ containing $\id$ by
$G^0(V)$.
The corresponding Lie algebra is denoted by $\mathfrak g_\ve(V)$ or
$\mathfrak g(V)$. It is $\so(V_2)$ ($\ve=1$) or $\algsp(V_2)$ ($\ve=-1$).

Suppose that $V_1$, $V_2$ are equipped with $\ve$-forms
$(\ ,\ )_{V_1}$, $(\ ,\ )_{V_2}$ respectively. We define the
\emph{adjoint} $C^*$ of a linear map $C\colon V_1\to V_2$ by
\begin{equation*}
  (Cv_1, v_2)_{V_2} = (v_1, C^* v_2)_{V_1}\qquad
  v_1\in V_1, v_2\in V_2.
\end{equation*}
We then have $(C^*)^* = C$.

Suppose $V_1$, $V_2$ are equipped with $\ve$-form and $(-\ve)$-form
respectively. We define $C^*$ by the same formula as above. We then have
$(C^*)^* = -C$.

If we identify $(\ ,\ )_{V_1}$ with $\varphi_1\colon V_1\to V_1^*$ by
$\langle v_1, \varphi_1(v_1')\rangle = (v_1, v_1')_{V_1}$, and similarly for $V_2$
with $\varphi_2$, we have $C^* = \varphi_1^{-1}\, \lsp{t}C\, \varphi_2$.

\begin{NB}
  Indeed,
  $(Cv_1,v_2)_{V_2} = \langle Cv_1, \varphi_2(v_2)\rangle
  = \langle v_1, \lsp{t}C \varphi_2(v_2)\rangle$, while
  $(v_1,C^*v_2)_{V_1} = \langle v_1, \varphi_1(C^* v_2)\rangle$.
\end{NB}%

\subsection*{Acknowledgments}

The author is grateful to Yiqiang Li about the discussion on his paper \cite{MR3900699}.
He is also grateful to Masatoshi Noumi for recalling earlier works on quantum analogs of symmetric spaces
at the workshop in June 2024 at Tokyo Institute of Technology.
It was the last conversation that the author had with him, and the author is saddened by his passing.
Parts of results in this paper were announced at Oberseminar Darstellungsthorie in Bonn in April 2025,
``Categorification and Symplectic Duality" workshop at Northeastern University in June 2025, and
``Representations, Moduli and Duality" workshop at the Bernoulli Center in August 2025.
However, the condition \ref{assum:induced-polarization} was found only at the last moment.
The author is grateful to the organizers for the invitation and hospitality.
The research was supported in part by JSPS Kakenhi Grant Numbers 23K03067.

\section{Extended Yangian}\label{sec:ext_Yangian}

In this section, we introduce the extended Yangian $\sX$ of
Drinfeld Yangian $\sY$ \cite{Drinfeld}. 
The definition was given in \cite{MR3849990}, but it appears in
earlier literature in various forms. For example, in \cite{MR2355506},
$\sY$ appears as $\mathrm{Y}(\mathfrak{sl}_N)$, while $\sX$ appears as
$\mathrm{Y}(\mathfrak{gl}_N)$.
We introduce $\sX$ as we will prove that the Maulik-Okounkov Yangian \cite{MR3951025} in
type ADE is isomorphic to $\sX$ in \cref{subsec:MOyangian}.

\subsection{Yangian}\label{subsec:Yangian}

Let $\fg$ be a finite-dimensional complex simple Lie algebra.
We will assume that $\fg$ is of type ADE later, but we do not need this assumption at the moment.
We fix an invariant orthogonal form $(\ ,\ )$ on $\fg$.
\begin{NB}
  This is a part of data, used in the definition of $\sY$.
\end{NB}%
We normalize it so that $(\alpha_i,\alpha_i) = 2$ for \emph{long} simple roots $\alpha_i$ of $\fg$.
This is \emph{the} normalized invariant form used for untwisted affine Lie algebras
in \cite[\S6]{Kac}.
\begin{NB}
  See the second paragraph of \cite[Proof of Cor.~3.6]{MR3917347} for
  some different conventions.
\end{NB}%
Since our primary interest is in the case when $\fg$ is of type ADE, all
roots have the same length. Hence we do not need to pay much attention to our
convention.

Let $\Bbbk = 
\CC[\hbar]$. The Yangian $\sY$ is a $\Bbbk$-algebra defined by
generators $X$ and $J(X)$ for $X\in\fg$ with the certain defining relations, including
\begin{equation*}
  XY - YX = [X,Y], J([X,Y]) = [X, J(Y)], \text{$J$ is linear in $X$},
  \dots
\end{equation*}
as in \cite[Theorem~2]{Drinfeld}.
It contains the universal enveloping algebra $\bU(\fg)$ of $\fg$ as a subalgebra, generated
by $X\in\fg$.
For later convenience, we recall the Casimir elements
\begin{equation}\label{eq:h:38}
  \Omega = \sum_\lambda X_\lambda\otimes X_\lambda,\quad
  \omega = \sum_\lambda X_\lambda X_\lambda,
\end{equation}
where $\{ X_\lambda\}$ is an orthonormal basis of $\fg$ with respect to $(\ ,\ )$.
It appears in the formula for the coproduct, e.g.,
\begin{equation*}
  \Delta J(h_i) = J(h_i)\otimes 1 + 1\otimes J(h_i) + \frac{\hbar}2 [h_{i,0},\Omega].
\end{equation*}
See \cite[(4.11)]{2017arXiv170105288G}.

Let $c_\fg$ denote the eigenvalue of $\omega$ on the adjoint representation of $\fg$.
In other words, it is the ratio of the Killing form and $(\ ,\ )$ on $\fg$. Under our normarlization of $(\ ,\ )$,
it is equal to twice of the dual Coxeter number of $\fg$. See \cite[Ex.~6.2]{Kac}. It appears in the definition of the antipode:
\begin{equation}\label{eq:h:27}
  S(X) = -X, \quad S(J(X)) = -J(X) + \frac{\hbar}4 c_\fg X.
\end{equation}
See \cite[(3.5)]{MR3849990}.

There is another presentation of $\sY$, called \emph{a new realization},
given in \cite[Theorem~1]{MR914215}. It is generated by $x^\pm_{i,r}$, $h_{i,r}$ for
$i\in I$, $r\in\BZ_{\ge 0}$, where $I$ is the index set of vertices of the Dynkin diagram corresponding to $\fg$.
We will use the convention in \cite[Def.~2.1]{2017arXiv170105288G}, but
the details of the defining relations are not important in this paper.
The subalgebra $\bU(\fg)$ is generated by $x_{i,0}^\pm$, $h_{i,0}$. They basically coincide
with standard generators of $\fg$, but rescaled according to $(\ ,\ )$, say
$h_i = \frac{(\alpha_i,\alpha_i)}2 \alpha_i^\vee$ for a simple coroot $\alpha_i^\vee$.
The equivalence of the two presentations is explicit, and is due to Drinfeld, and the
details are available in \cite{MR3917347}. One of formulas in the equivalence is
\begin{equation*}
  h_{i,1} = J(h_i) - \hbar v_i
  \quad\text{ where }
  v_i = \frac14 \sum_{\alpha\in\Delta^+}(\alpha,\alpha_i) \{ x_\alpha^+, x_\alpha^-\}
  - \frac12 h_i^2,
\end{equation*}
and $\{ a,b\} = ab + ba$.

The Yangian $\sY$ is equipped with an ascending filtration 
  $\sY_{\le 0} \subset \sY_{\le 1} \subset \sY_{\le 2} \subset \cdots$
defined by
$\deg X = 0$, $\deg J(X) = 1$ for $X\in\fg$.
The associated graded algebra $\gr \sY = \sY_{\le 0}\oplus\sY_{\le 1}/\sY_{\le 0}
\oplus\cdots$
is isomorphic to the universal enveloping algebra $\bU(\fg[u])$
of the current algebra $\fg[u] = \fg\otimes_\CC \Bbbk[u]$ of 
$\fg$.\footnote{
  It is isomorphic to the algebra given by $\hbar = 0$, and tensored with $\Bbbk$ over $\CC$.
  In Maulik-Okounkov Yangian, which makes sense for arbitrary quiver
  including Jordan quiver, this is \emph{not} the case. Indeed,
  Maulik-Okounkov Lie algebra $\fg^{\mathrm{ext}}_{\mathrm{MO}}$, in general, 
  is defined \emph{a priori} only over $\Bbbk$.}
The embedding $\bU(\fg)\to \sY$ specializes to the embedding $\fg\to \fg[u]$ at the associated graded.

Yangian is a Hopf algebra, in particular has a coproduct $\Delta$. Its
definition is given in the original presentation. The
details about the proof of the well-definedness of $\Delta$ are available,
if $\fg \neq \algsl_2$, and more generally affine Lie algebras 
for $\fg\neq A^{(1)}_1$, $A^{(2)}_2$, in \cite[Th.~4.9]{2017arXiv170105288G}.

It is known that $\sY$ has a one parameter family of Hopf algebra
automorphism $\tau_\zeta$ for $\zeta\in\CC$. See \cite[\S3.2]{2017arXiv170105288G}.
It gives a translation $u\mapsto u+\zeta$ on the associated graded algebra $\bU(\fg[u])$.
From \eqref{eq:h:27}, we have
\begin{equation}\label{eq:h:28}
  S^2 = \tau_{-\frac{\hbar}2 c_\fg}.
\end{equation}
See \cite[Cor.~3.3]{MR3849990}. This will play an important role later in comparison
with the extended Yangian and Drinfeld Yangian.

Yangian has a distinguished family of irreducible representations $\{ F_i\}$
indexed by $i\in I$.
Following \cite[\S1.4]{Na-Tensor}, we call them \emph{$\ell$-fundamental representations}.
They are called \emph{fundamental representations} in \cite{MR1103907}.
It should be noted that their restriction to $\fg$ are not irreducible in general,
hence different from the fundamental representations of $\fg$.
To avoid confusion, we introduce the \emph{new} terminology, which is almost the same as the original one.
Let us denote it by $\rho_{F_i}$:
\begin{equation*}
  \rho_{F_i}\colon \sY\to \End_{\Bbbk}(F_i).
\end{equation*}
One way to define them
\begin{NB}
  at least after specializing $\hbar = 1$ ?
\end{NB}%
is to specify the Drinfeld polynomials as
$P_i(u) = u$, $P_j(u) = 1$ for $j\neq i$. See \cite[\S2.13]{MR1103907}.
It means that $F_i$ has a distinguished vector $v_i$, called the \emph{$\ell$-highest weight vector}, such that
\begin{equation*}
  \begin{split}
    &  \sY v_i = F_i, \qquad
      x^+_{j,r} v_i = 0 \quad \text{for $j\in I$, $r\ge 0$},\\
    & \left(1 + \hbar\sum_{r=0}^\infty h_{j,r} u^{-r-1}\right) v_i 
      = \left(
      \frac{P_j(u+\frac{\hbar}2\frac{(\alpha_i,\alpha_i)}2)}{P_j(u-\frac{\hbar}2\frac{(\alpha_i,\alpha_i)}2)}
      \right)^- v_i
      \quad \text{for $j\in I$},
  \end{split}
\end{equation*}
where $(\ )^-$ means the expansion at $u=\infty$.
These conditions fixes $v_i$ uniquely up to a scalar multiple. We choose and fix $v_i$ hereafter.

In type ADE, they are constructed explicitly as the equivariant cohomology of the
lagrangian subvarieties 
in the quiver variety (\cite{Varagnolo}). 
In order to follow the convention in \cite{MR3951025}, we define it
as the equivariant cohomology of the whole quiver variety:
\begin{equation}\label{eq:h:48}
  F_i = H^*_{\CC^\times_\hbar}(\fM(\varpi_i))
  = \bigoplus_{\bv} H^*_{\CC^\times_\hbar}(\fM(\bv,\varpi_i)).
\end{equation}
The construction will be reviewed briefly later in \cref{subsec:MOyangian}.
Strictly speaking, this is different from one given by the lagrangian subvarieties.
But they are isomorphic if we invert $\hbar$. It is also known that the above
$F_i$ is the dual module of one given by the lagrangian subvarieties, see \cite[Th.~7.3.5]{Na-qaff}.

The component $\fM(0,\varpi_i)$ is a single point, hence we have a distinguished class
$1\in H^*_{\CC^\times_\hbar}(\fM(0,\varpi_i))$. We take it as the $\ell$-highest weight vector $v_i$ of $F_i$,
in type ADE.

It is known that $F_i$ is free of finite rank over $\Bbbk$.
See \cite[\S7]{Na-qaff}.
We fix a $\Bbbk$-base of $F_i$.

\subsection{\texorpdfstring{$R$}{R}-matrix}

Let us consider $F_i[u] = F_i\otimes \Bbbk[u]$, where
$\otimes$ is over $\Bbbk$ hereafter.
We can make it a $\sY$-module through $\tau_u$: $\left.F_i[u]\right|_{u=a}$ is
the $\sY$-module $F_i$ pull-backed by $\tau_a$ for $a\in\CC$.

The tensor product $F_i[u_1]\otimes F_j[u_2]$ has an $R$-matrix (\cite[Th.~4]{Drinfeld} and \cite[Th.~7.2]{MR4391348} for a published proof):
\begin{Fact}\label{fact:R-matrix}
  There exists $R_{F_i,F_j}(u_1-u_2)\in \End(F_i\otimes F_j)$ depending
  rationally on $u_1 - u_2$ with $R_{F_i,F_j}(\infty) = \id$ such that
  \begin{equation*}
    (12) R_{F_i,F_j}(u_1-u_2) \colon F_i[u_1]\otimes F_j[u_2] \to
    F_j[u_2]\otimes F_i[u_1]
  \end{equation*}
  intertwines the $\sY$-module structures on both sides. Here $(12)$ is the
  exchange of the first and second factors.
  Moreover, it comes from the universal $R$-matrix $\scR(u)\in 1 + \frac{\hbar}{u}\sY\otimes\sY[[u^{-1}]]$ by
  evaluation up to rescaling by an element $f(u)\in 1 + u^{-1}\Bbbk[[u^{-1}]]$.
\end{Fact}

We normalize the $R$-matrix so that the tensor product of the $\ell$-highest
weight vectors of $F_i[u_1]$, $F_j[u_2]$ is preserved.
\begin{NB}
This is the convention in \cite{MR3951025}.
\end{NB}%
Since this normalized $R$-matrix is different from one from the universal $R$-matrix,
let us denote the universal $R$-matrix by $\scR^{\mathrm{uni}}(u)$.
Both will play a role below.

The $R$-matrices satisfy the Yang-Baxter equation:
\begin{equation}
  \begin{multlined}
    R_{F_1,F_2}(u_1-u_2) R_{F_1,F_3}(u_1-u_3) R_{F_2,F_3}(u_2-u_3) \qquad \\
    \qquad = R_{F_2,F_3}(u_2-u_3) R_{F_1,F_3}(u_1-u_3) R_{F_1,F_2}(u_1-u_2)
  \end{multlined}
\end{equation}
as rational operators on $F_1[u_1]\otimes F_2[u_2]\otimes F_3[u_3]$,
where $F_1$, $F_2$, $F_3$ are chosen among $\{ F_i\}$.

\subsection{Monodromy matrices}

For a sequence of indices $i_1,\dots, i_n$, we define a $\sY$-module by
\begin{equation}\label{eq:h:11}
  \scH 
  = F_{i_1}[u_1]\otimes F_{i_2}[u_2]\otimes\dots\otimes F_{i_n}[u_n],
\end{equation}
where the $\sY$-module structure
is given by the coproduct $\Delta$.

Choose $F$ among $\{ F_i\}$ as an auxiliary space, we introduce a \emph{monodromy} matrix:
\begin{equation}\label{eq:h:12}
  \begin{split}
    \sT_F(u) \defeq R_{F[u],\scH} &=
    R_{F,F_{i_n}}(u-u_n) R_{F, F_{i_{n-1}}}(u-u_{n-1})\cdots
    R_{F,F_{i_1}}(u-u_1) \\
    &\in\id + \frac{\hbar}u \End_{\Bbbk[u_1,u_2,\dots, u_n]}(F\otimes \scH)[[u^{-1}]].
  \end{split}
\end{equation}
In the view point in \cite[\S4.2.4]{MR3951025},
this operator moves from the chamber
$\{ u > u_1 > u_2 > \cdots > u_n\}$ to the chamber 
$\{ u_1 > u_2 > \cdots > u_n > u\}$ in the space of cocharacters.
The first step goes from $\{u > u_1\}$ to $\{ u_1 > u\}$, hence
$R_{F,F_{i_1}}(u-u_1)$. Then continue.
In terms of tensor product representations, it intertwines
between $F \otimes (F_{i_1}\otimes\dots\otimes F_{i_n})$ and
$(F_{i_1}\otimes\dots\otimes F_{i_n})\otimes F$ after swapping the two factors.

By the base of $F$, we regard $\sT_F(u)$ as a matrix of size $=\dim F$ whose entries are
elements of $\End_{\Bbbk[u_1,u_2,\dots, u_n]}(\scH)[[u^{-1}]]$.
We run all $\scH$, hence consider $\sT_F(u)-\id$ as elements of
$\frac{\hbar}u \prod_{\scH}\End_{\Bbbk[u_1,u_2,\dots, u_n]}(\scH)[[u^{-1}]]$.
So we do not include $\scH$ in the notation of $\sT_F(u)$.
Note also that $\sT_F(u)$, once $\scH$ is fixed, is a rational function in $u-u_1$, \dots, $u-u_n$ by its construction.
It is regular at $u=\infty$, the expansion at $u=\infty$ gives a formal power series in $u^{-1}$ as above.

\subsection{$RTT$ relation}

Let us extend the action of $\sT_{F_i}(u)$ 
(resp.\ $\sT_{F_j}(v)$) to
$F_i[u]\otimes F_j[v]\otimes\scH$ by the identity on the 
remaining factor $F_j[v]$ (resp.\ $F_i[u]$).
From the Yang-Baxter equation, one derives the $RTT$ relation:
\begin{equation}\label{eq:h:5}
  R_{F_i,F_j}(u-v) \sT_{F_i}(u)\sT_{F_j}(v) = \sT_{F_j}(v)\sT_{F_i}(u) R_{F_i,F_j}(u-v),
\end{equation}
where $R_{F_i,F_j}(u-v)$ acts by the identity on the factor $\scH$.

\begin{NB}
The $RTT$ relation means the commutativity of the following diagram
for arbitrary choice of $\scH$:
\begin{equation}\label{eq:h:4}
  \begin{tikzcd}[column sep = 10em]
    F_1[u]\otimes F_2[v]\otimes \scH \arrow[r,"{(1\otimes\tau_{v-u})\Delta(\sT(u)) = \sT_1(u)\sT_2(v)}"]
    \arrow[d,"R_{F_1 F_2}(u-v)"']
    & F_1[u]\otimes F_2[v]\otimes \scH
    \arrow{d}{R_{F_1 F_2}(u-v)}\\
    F_1[u]\otimes F_2[v]\otimes \scH \arrow[r,"{(1\otimes\tau_{v-u})\Delta^{\mathrm{op}}(\sT(u)) = \sT_2(v)\sT_1(u)}"']
    & F_1[u]\otimes F_2[v]\otimes \scH\rlap{.}
  \end{tikzcd}
\end{equation}
\end{NB}%

\subsection{Extended Yangian}\label{subsec:ext_yangian}

We choose a finite dimensional representation $V$ of $\sY$ as
$\bigoplus_{i\in I} F_i$, and apply 
the construction \cite[\S5]{MR3849990} of the extended Yangian $\sX$.
The index set $\mathcal I$, appearing as $X_{\mathcal I}(\fg)$ in the notation of the extended Yangian in \cite{MR3849990},
is nothing but $I$ in our choice of $V$:
By \cite[\S4.3]{MR3849990}, $\mathcal I$ is the index set of a chosen base of $\End_{\sY}(V)$. In our case,
we choose $\{ \id_{F_i}\}_{i\in I}$ as a base of $\End_{\sY}(V)$.
We then formulate the construction as follows.

\begin{Definition}\label{def:ext_yangian}
The \emph{extended Yangian} $\sX$ is the algebra generated by elements
$\{ t_{i;ab}^{(r)} \mid i\in I, a,b = 1,\dots, \dim F_i, r\ge 1\}$, subject to the
defining $RTT$ relations \eqref{eq:h:5} viewed
as an equality in $\End(F_i\otimes F_j)\otimes \sX[u, u^{-1}]][[v^{-1}]]$.
Here $\sT_{F_i}(u) = \id_{F_i} + \left(\sum_{r} t_{i;ab}^{(r)} u^{-r-1}\right)_{a,b}$
is regarded as an element of $\End(F_i)\otimes\sX[[u^{-1}]]$
under the chosen base of $F_i$.
\end{Definition}

The extended Yangian $\sX$ has a representation on $F$ for one of $\{ F_i\}$ given by
\begin{equation}\label{eq:h:29}
  \rho_F^{\mathrm{ext}}\colon \sX \to \End(F);\qquad
  \sT_{F_i}(u) \mapsto R_{F_i,F}(u).
\end{equation}  
The $RTT$ relation follows from the Yang-Baxter equation.

\begin{NB}
We take a slightly different formulation.
Following \cite[\S5.2]{MR3951025}, we define the \emph{extended Yangian} $\sX$ as the subalgebra of
\begin{equation}\label{eq:h:18}
  \prod_{n=0}^\infty\prod_{i_1,\dots,i_n} \End_{\Bbbk[u_1,u_2,\dots, u_n]}(\scH),
\end{equation}  
generated by coefficients of $\hbar u^{-N-1}$ of matrix entries of
$\sT_F(u)-\id$ for all $N\ge 0$ and all auxiliary spaces $F$.

Since $F$ runs all $\ell$-fundamental representations, our
construction corresponds to choosing the finite dimensional
representation $V = \bigoplus_{i\in I} F_i$ in 
the construction \cite[\S5]{MR3849990} of the extended Yangian $\sX$.
The index set $\mathcal I$, appearing as $X_{\mathcal I}(\fg)$ in the notation of the extended Yangian in \cite{MR3849990},
is nothing but $I$ in our choice of $V$:
By \cite[\S4.3]{MR3849990}, $\mathcal I$ is the index set of a chosen base of $\End_{\sY}(V)$. In our case,
we choose $\{ \id_{F_i}\}_{i\in I}$ as a base of $\End_{\sY}(V)$.

We further embed the extended Yangian in \cite{MR3849990} into
$\prod_{\scH} \End(\scH)$ from the defining representation on $F_i$,
\begin{NB2}
   Given by the $\sT_{F_j}(u) \mapsto R_{F_j,F_i}(u)$. See
   \cite[a paragraph below Remark~5.2]{MR3849990}.
\end{NB2}%
coproduct and the translation automorphism below.
Since the homomorphism from $\sY$ this algebra is injective, 
and 
there is no
difference in this formulation in practice.
\end{NB}%

\begin{Definition}\label{def:filt}
We define an ascending \emph{filtration} on $\sX$ by the degree of $u$,
i.e., $\sX_{\le N}$ is generated by coefficients of $\hbar u^{-n-1}$ with $n\le N$.
\end{Definition}

\begin{Remark}\label{rem:explicitR}
In the conventional $RTT$ approach to the extended Yangian, one starts
with $R$, given as an \emph{explicit matrix}, rather than a \emph{linear operator}
on a vector space, which has no specific choice of a base. One can avoid
the use of Drinfeld's definition of $\sY$, and define the extended Yangian
directly in this approach. See \cite{MR2355506}.
In this approach, people usually assume $\fg$ is of classical type, and take
$V$ as a vector representation of $\fg$.
In particular, this construction \emph{explains} that the vector representation
is an example of $\ell$-fundamental representations, which
are also fundamental representations of $\fg$.
%
%
See \cref{rem:smaller_ext-Yang} below.
\end{Remark}

\subsection{Hopf algebra structure}\label{subsec:coproduct}

One has a Hopf algebra structure on $\sX$ with coproduct $\Delta$,
antipode $S$, and counit $\ve$ given by
\begin{equation*}
  \Delta(\sT_{F_i}(u)) = (\sT_{F_i}(u)\otimes 1)(1\otimes\sT_{F_i}(u)),\quad
  S(\sT_{F_i}(u)) = \sT_{F_i}(u)^{-1}, \quad
  \ve(\sT_{F_i}(u)) = \id_{F_i}.
\end{equation*}

\begin{NB}
The tensor product $\scH_1\otimes\scH_2$ of two spaces $\scH_1$, $\scH_2$ of form \eqref{eq:h:11} is again
of the same form. The projection
\begin{equation*}
   \prod_{\scH} \End(\scH) \to \prod_{\scH_1,\scH_2} \End(\scH_1\otimes\scH_2)
\end{equation*}
gives a coproduct
\begin{equation*}
  \Delta\colon \sX\to \sX\otimes\sX
\end{equation*}
by sending $\sT_F(u)$ to $(\sT_F(u)\otimes 1)(1\otimes\sT_F(u))$.
\begin{NB2}
  We multiply $F$ factor, i.e., $t_{ij}(u)\mapsto \sum_k t_{ik}(u)t_{kj}(u)$.
\end{NB2}%
It is clear that $\tilde\Phi\colon \sX\to \sY$
intertwines the coproducts.

We define the counit $\delta\colon \sX\to\Bbbk$ by the projection to the choice $\scH = \Bbbk$ for $n=0$ in \eqref{eq:h:11}.

Since the leading term is $\id$, $\sT_F(u)$ is invertible. We denote the inverse by $\sT(u)^{-1}$.
(Note that it is denoted by $\sT_F^{-1}(u)$ in \cite{MR2355506}.) 
\begin{NB2}
  If $\sT(u) = \id + \hbar\sum_{N=1}^\infty \sT^{(N)} u^{-N}$,
  then $\sT(u)^{-1} = \id - \hbar\sum_{N=1}^\infty \sT^{(N)} u^{-N} + \hbar^2\sum_{N=1}^\infty (\sT^{(N)})^2 u^{-2N} - \cdots$.
  Therefore $\sT^{(1)}\mapsto -\sT^{(1)}$, $\sT^{(2)}\mapsto -\sT^{(2)} + \hbar(\sT^{(1)})^2$, \dots.
\end{NB2}%
As in \cite[\S1.5]{MR2355506}, it gives the antipode $S$, which makes
$\sX$ a Hopf algebra. One shows that $\tilde\Phi$ is a surjective Hopf
algebra homomorphism. See \cite[Lemma~6.1]{MR3849990}.
\end{NB}%

\subsection{Translation automorphism}\label{subsec:translation}

We have a family of automorphisms $\tau_\zeta\colon \sX\to \sX$ for $\zeta\in\CC$ given by
\begin{equation*}
  \sT_{F_i}(u)\mapsto \sT_{F_i}(u-\zeta).
\end{equation*}

\begin{NB}
The space $\scH$ in \eqref{eq:h:11} has an automorphism
$\tau_c\colon \scH\to\scH$ given by $u_i\mapsto u_i + c$. It
induces an automorphism of $\sX$, denoted by the same
symbol $\tau_c$. 
\end{NB}%

By the coproduct and translation automorphisms, we equip $\scH=F_{i_1}[u_1]\otimes F_{i_2}[u_2]\otimes\dots\otimes F_{i_n}[u_n]$ in \eqref{eq:h:11} with a $\sX$-module
structure. Thus we have a homomorphism
\begin{equation*}
  \prod_{\scH}\rho^{\mathrm{ext}}_\scH \colon \sX\to 
    \prod_{\scH}
  \End_{\Bbbk[u_1,u_2,\dots, u_n]}(
  \scH  
  ).
\end{equation*}
In \cite{MR3951025} the extended Yangian $\sX$ is defined as
a subalgebra of the right hand side of the above equation.
Since the above homomorphism $\prod_{\scH}\rho^{\mathrm{ext}}_\scH$ is injective, say by 
\cite[\S5.5.3]{MR3951025}, there is no difference in practice.

\subsection{Homomorphism to the Yangian}

We have a homomorphism $\tilde\Phi\colon \sX\to \sY$ given by
\begin{equation}\label{eq:h:31}
  \tilde\Phi\colon \sT_{F_i}(u) \mapsto (\rho_{F_i}\otimes 1)(\scR^{\mathrm{uni}}(u)).
\end{equation}
Recall $\rho_{F_i}\colon \sY\to \End(F_i)$ is the representation of $\sY$ on $F_i$. This is a surjective homomorphism
of Hopf algebras by \cite[Lemma~6.1]{MR3849990}.
It is also a filtered homomorphism, as shown in the proof of \cite[Th.~6.2]{MR3849990}.
It intertwines the translation automorphisms $\tau_\zeta$ of $\sX$ and $\sY$, thanks
to a property of the universal $R$-matrix \cite[(3.10)]{MR3849990}.

Remark that the definition above uses the universal $R$-matrix, not the normalized one.
This is necessary since we want to get a homomorphism to $\sY$.

Let us define
\begin{equation}\label{eq:h:30}
  Z_{F_i}(u) = S^2(\sT_{F_i}(u)) \sT_{F_i}(u+\frac\hbar2 c_\fg)^{-1}
  \in \End(F_i)\otimes \sX[[u^{-1}]].
\end{equation}
By \eqref{eq:h:28}, it is killed by $\tilde\Phi$, i.e., $\tilde\Phi(Z_{F_i}(u)) = 1$.

Since our normalized $R$-matrix differs from the universal one by a rescaling
by $f(u)$, we see that $Z_{F_i}(u)$ is $f(u) f(u+\frac{\hbar}2 c_\fg)^{-1}$ times
the identity operator in the representation $F_i$ of $\sX$ defined by \eqref{eq:h:29}. 
We write $Z_{F_i}(u) = \id_{F_i}\otimes z_{F_i}(u)$ with $z_{F_i}(u)\in \sX[[u^{-1}]]$.
It has the expansion $z_{F_i}(u) = 1 + z_i^{(2)} u^{-2} + z_i^{(3)} u^{-3} + \cdots$.
There is no term of $u^{-1}$ by the definition. See \cite[Lemma~7.5]{MR3849990}.
Since $z_{F_i}(u)$ commutes with any elements of $\sX$ in $\scH$,
it is contained in the center.
See also \cite[Th.~1.9.9]{MR2355506} for $\fg = \algsl_N$.

Furthermore, we have
\begin{Theorem}[\protect{\cite[Th.~6.2, Prop.~7.6, Th.~7.3]{MR3849990}}]\label{thm:ext-Yang-center}
  \textup{(1)}
  Let $\sY_R$ be the quotient $\sX/(z_{F_i}(u) - 1)$.
  Here $(z_{F_i}(u) - 1)$ is the ideal generated by matrix entries of coefficients
  of $z_{F_i}(u) - 1$ for all $i\in I$.
  Then $\tilde\Phi$ induces an isomorphism 
  \begin{equation*}
  \sY_R \xrightarrow[\cong]\Phi\sY.
  \end{equation*}

  \textup{(2)} The center of $\sX$ is isomorphic to the polynomial
  ring of generators $z_i^{(r)}$ for $i\in I$, $r\ge 2$, appeared in the expansion
  of $z_{F_i}(u)$.

  \textup{(3)} The extended Yangian $\sX$ is isomorphic to the tensor
  product
  \begin{equation*}
    \sY_R\, \otimes\, \Bbbk[z_i^{(r)}\mid i\in I, r\ge 2].
  \end{equation*}
\end{Theorem}

\subsection{Extended Lie algebra \texorpdfstring{$\fg^{\mathrm{ext}}$}{g^ext}}
\label{subsec:ext_Lie}

Let $\fg^{\mathrm{ext}}$ be the linear span of matrix entries of coefficients of $\hbar u^{-1}$ of $\sT(u) - \id$.
From the expansion of the $R$-matrix as
\begin{equation*}
  R_{F_i,F_j}(u) = \id + \frac{\hbar}{u} \sr_{F_i,F_j} + O(u^{-2}),
\end{equation*}
where $\sr_{F_i,F_j}\in\End(F_i\otimes F_j)$ is the \emph{classical $R$-matrix},\footnote{We follow
the \emph{unconventional} terminology of the classical $R$-matrix,
taken in \cite[\S4.8.4]{MR3951025}.}
one can show that $\fg^{\mathrm{ext}}$ forms a Lie subalgebra.
See \cite[\S4]{MR3849990}, \cite[\S5.3]{MR3951025}.
One of important ingredients in the proof is that the classical $R$ for the
universal $\scR^{\mathrm{uni}}(u)$ is the Casimir operator $\Omega$ in \eqref{eq:h:38}.

Indeed, the extended Lie algebra $\fg^{\mathrm{ext}}$ is 
$\fg\oplus\mathfrak z_{I}$ where $\mathfrak z_I$ is the
abelian Lie algebra of dimension $|I|$.
See \cite[(4.25)]{MR3849990}.
\begin{NB}
  See also \cite[\S4.2.1]{MR3849990} for $\mathfrak z_{\mathcal J}$, which
  is the Poisson version.
\end{NB}%
Note that $\fg^{\mathrm{ext}}$ is naturally equipped with a
representation on $V = \bigoplus_{i\in I} F_i$. We define $\delta_i\in
\mathfrak z_I\subset \fg^{\mathrm{ext}}$ as the identity
operator on $F_i$, extended by $0$ to the other $F_j$ for $j\neq i$.
It means that $\fg^{\mathrm{ext}}$ is the \emph{maximal} realization
of the Lie algebra, used in the context of Kac-Moody algebras.
See \cite[Remark~2.1.2]{MR1881971}. ($d_i$ there is
$\delta_i + \sum \bC^{ij} h_j$ where $(\bC^{ij})$ is the inverse of the
Cartan matrix $\bC = (\bC_{ij})$.)
\begin{NB}
  $\langle \delta_i + \sum \bC^{ij} h_j,\alpha_k\rangle
  = \sum \bC^{ij} a_{jk} = \delta_{ik}$.
\end{NB}%
See also \cite[\S5.3.4]{MR3951025}. The Cartan subalgebra $\fh^{\mathrm{ext}}$ of $\fg^{\mathrm{ext}}$ has a basis
$\{ h_i, \delta_i\}_{i\in I}$. It is convenient if we regard
the fundamental weights $\varpi_i$ and simple roots $\alpha_i$ as linearly independent
elements of $(\fh^{\mathrm{ext}})^*$ by $\langle \delta_i, \varpi_j\rangle = \delta_{ij}$,
$\langle \delta_i, \alpha_j\rangle = 0$, while
$\langle h_i, \varpi_j\rangle = \delta_{ij}$, $\langle h_i,\alpha_j\rangle = a_{ij}$.

\begin{NB}
It contains a subalgebra $\overline{\fh}_Q$, which acts by linear functions on $\bv$, $\bw$.
More precisely, let $\bw = \sum_p \varpi_{i_p}$ in \eqref{eq:h:11}, and
decompose $\scH$ as $\scH = \bigoplus_{\bw_i} H^*_{\CC^\times_\hbar\times\sA}(\fM(\bv,\bw)^\sA)$.
We define $h_i$, $\delta_i\in\overline{\fh}_Q$ for each $i\in I$ such that
$\langle \delta_i, \varpi_j\rangle = \delta_{ij}$, $\langle \delta_i, \alpha_j\rangle = 0$,
$\langle h_i, \varpi_j\rangle = \delta_{ij}$, $\langle h_i, \alpha_j\rangle = a_{ij}$.
Then $\delta_i$, $h_i$ acts on 
a summand $H^*_{\CC^\times_\hbar\times\sA}(\fM(\bv,\bw)^\sA)$ of $\scH$ so
that it is a weight space of weight $\bw - \bv$.
Therefore we have $[h_j, e_i] = a_{ji} e_j$, $[\delta_i, e_j] = \langle \delta_i, \alpha_j\rangle e_j = 0$.

  Let $\overline{\bC} = \begin{pmatrix}
    -\bC & 1 \\ 1 & 0
  \end{pmatrix}$. Therefore we have $\overline{\bC}^{-1} = \begin{pmatrix}
    0 & 1 \\ 1 & -\bC
  \end{pmatrix}$. The dual basis with respect to $\overline{\bC}$ is given by
  $h^i$, $\delta^i$ acting on $H^*_{\CC^\times_\hbar\times\sA}(\fM(\bv,\bw)^\sA)$ by
  $\bv_i - \sum_j \bC^{ij}\bw_j$, $\sum_j \bC^{ij}\bw_j$ respectively.
  Therefore
  \begin{equation*}
    \begin{split}
      &\phantom{{}={}} \sum h_i\otimes h^i + \sum \delta_i \otimes \delta^i
    = \sum_i (\bw_i - \sum_j \bC_{ij} \bv_j)\otimes (\bv_i - \sum_j \bC^{ij} \bw_j))
    + \sum_i \bw_i \otimes \sum_j \bC^{ij} \bw_j\\
    &= \sum_i \bw_i\otimes \bv_i + \bv_i\otimes\bw_i - \sum_j \bC_{ij} \bv_j\otimes\bv_i.
    \end{split}
  \end{equation*}
  This is the diagonal part of the classical $R$-matrix. See \cite[(4.21)]{MR3951025}.
\end{NB}%

In this situation, contrary to more general cases in \cite{MR3951025},
the extended Lie algebra $\fg^{\mathrm{ext}}$ is defined over $\CC$, instead of $\Bbbk$.

We have the following result, which has been known for many versions of Yangian.

\begin{Theorem}[\protect{\cite[Th.~7.7]{MR3849990}}]\label{thm:PBW}
The associated graded algebra of $\sX$ with respect to the filtration defined in
\cref{def:filt} is isomorphic to the universal enveloping algebra $\bU(\fg^{\mathrm{ext}}[u])$
of the current algebra $\fg^{\mathrm{ext}}[u] = \fg^{\mathrm{ext}}\otimes_\CC \Bbbk[u]$.
\end{Theorem}

\subsection{Recovering the Yangian}\label{subsec:recoverY}

Let $f_i(u) = 1 + f_i^{(1)} u^{-1} + f_i^{(2)} u^{-2} + \cdots \in 1 + u^{-1}\Bbbk[[u^{-1}]]$ be a formal power series in $u^{-1}$
starting with $1$. 
We consider a collection $\bff(u) = (f_i(u))_{i\in I}$ of such series.
It defines an automorphism of $\sX$ by
\begin{equation}\label{eq:h:26}
  m_\bff\colon \sT_{F_i}(u) \mapsto f_i(u) \sT_{F_i}(u) \quad\text{for $i\in I$}.
\end{equation}
See \cite[Lemma~5.3]{MR3849990}.

In our realization of $\sX$ in $\prod_{\scH} \End(\scH)[[u^{-1}]]$,
the action of $\sT_F(u)$ on 
the summand $\scH = F_{i_1}[u_1]\otimes \dots \otimes F_{i_n}[u_n]$ is multiplied by
$f_{i_1}(u-u_1) f_{i_2}(u-u_2) \dots f_{i_n}(u-u_n)$, as $\sT_F(u)$ acts
through the coproduct $\Delta$ in \cref{subsec:coproduct}.

On $\fg^{\mathrm{ext}}$, more precisely, the lowest part of the filtration $\sX_{\le 0}$,
it is given by
\begin{equation*}
  \text{elements in $\fg$ is fixed}, \quad
  \delta_i \mapsto \delta_i + f_i^{(1)}\id.
\end{equation*}
\begin{NB}
On the summand $\scH = F_{i_1}[u_1]\otimes \dots \otimes F_{i_n}(u_n)$, $\delta_i$
acts by $n(\delta_i + f_{i,1})$ where $n$ is the number of $i_j = i$.
\end{NB}%
In particular, the $m_{\bff}$-fixed part of $\sX_{\le 0}$ 
is the corresponding part in Drinfeld Yangian $\sY$.

In order to extend this observation to the whole $\sX$, we prepare
rescaled operators of $\sT_{F_i}(u)$ as follows.

We introduce a formal power series $y_{F_i}(u) = 1 + y_i^{(1)} u^{-1} + y_i^{(2)} u^{-2} + \dots$, valued in the center of $\sX$, for each $i\in I$,
by solving the equation
\begin{equation*}
  y_{F_i}(u) = z_{F_i}(u) y_{F_i}(u+\frac\hbar2 c_\fg).
\end{equation*}
One see easily that $y_i^{(r)} = z_i^{(r+1)}$ modulo terms with $z_i^{(s)}$ for $s\le r$,
and we have the unique solution. In \cite[\S7.1]{MR3849990}, it is defined as
a matrix $\mathcal Y(u)$, but it is a scalar matrix in our case, as $Z_{F_i}(u)$ is so.

Under $m_\bff$, the element $z_{F_i}(u)$ is multiplied by $f_i(u) f_i(u+\frac\hbar{2}c_\fg)^{-1}$
from \eqref{eq:h:30}. Therefore $y_{F_i}(u)$ is multiplied by $f_i(u)$.

\begin{Theorem}[\protect{\cite[Th.~7.3, Th.~7.11]{MR3849990}}]
\textup{(1)} Let $\tilde\sY_R$ be the subalgebra of $\sX$ generated by matrix entries of coefficients of
$y_{F_i}(u)^{-1} \sT_{F_i}(u)$ for $i\in I$. It is isomorphic to $\sY$ under the restriction of $\tilde\Phi$.

\textup{(2)} The subalgebra $\tilde\sY_R$ is equal to the $m_{\bff}$-fixed subalgebra
  $\{ x\in \sX \mid m_{\bff}(x) = x\}$
  of $\sX$.%
\end{Theorem}

We thus recover the original Yangian $\sY$ both as a quotient and subalgebra of the extended
Yangian $\sX$.

The following will be used later in the last part of \cref{subsec:MOyangian}.
\begin{Lemma}\label{lem:tildeSYR}
  \textup{(1)} Let $F$ be one of $\{ F_i\}_{i\in I}$.
  The following diagram commutes:
  \begin{equation*}
  \begin{CD}
     \tilde\sY_R @>{\rho_F^{\mathrm{ext}}|_{\tilde\sY_R}}>> \End(F) \\
      @V{\tilde\Phi|_{\tilde\sY_R}}V{\cong}V @| \\
      \sY @>>{\rho_F}> \End(F)
  \end{CD}
  \end{equation*}
  Here ${}|_{\tilde\sY_R}$ means the restriction of homomorphisms to $\tilde\sY_R$.
  The same result holds if we replace $F$ by $\scH$ in \eqref{eq:h:11}.

  \textup{(2)} The image of $\tilde\sY_R$ in $\prod_\scH \End(\scH)$
  under the injective homomorphism $\prod_\scH \rho^{\mathrm{ext}}_\scH$ 
  coincides with the image of $\sY$.
\end{Lemma}

\begin{proof}
By \eqref{eq:h:31}, the composition passing through the left bottom corner is
given by $(\rho_{F_i}\otimes \rho_F)(\scR^{\mathrm{uni}}(u))$. On the other hand,
the composition passing through the right top corner is given by
$R_{F_i,F}(u)$ by \eqref{eq:h:29}.
The two $R$-matrices differ by a rescaling by $f(u)$, but it can be absorbed
if we restrict to the $m_{\bff}$-fixed subalgebra $\tilde\sY_R$.

The same is true for $\scH$, as $\tilde\Phi$ is a Hopf algebra isomorphism,
compatible with translation automorphisms $\tau_\zeta$.

(2) follows from (1).
\end{proof}

\begin{Remark}\label{rem:smaller_ext-Yang}
  The definition of the extended Yangian $\sX$ depends on the choice of a finite
  dimensional vector space $V = \bigoplus_{i\in I} F_i$.
  We can replace it by a nontrivial irreducible representation, say a single $F_i$.
  Let $\sX'$ be the corresponding extended Yangian. We have a homomorphism $\sX'\to \sX$
  sending $\sT_F(u)$ to $\sT_{F_i}(u)$. It is isomorphic to $\sY\otimes \Bbbk[z_i^{(r)}\mid r\ge 2]$
  as in \cref{thm:ext-Yang-center}. (Here $i$ is fixed.)
  Moreover, we can recover $\sY$ from $\sX'$ either as a quotient or a subalgebra. 
\end{Remark}

\subsection{Unitarity}

\begin{NB}
The $R$-matrix satisfies the unitarity
\begin{equation}\label{eq:h:3}
  R_{21}(u) = R(-u)^{-1},
\end{equation}
where the subscript in $R_{21}(u)$ means the permutation of the tensor
factors. More precisely, let $(12)$ denote the exchange of the first and
second factors. Then $R_{21}(u) = (12) R(u) (12)$.
See \cite[\S4.5]{MR3951025}.

We consider the diagram \eqref{eq:h:4} with $F_1[u]$, $F_2[v]$
exchanged, and compose it with $(12)$ to the top and bottom:
\begin{equation*}
  \begin{CD}
    F_1[u] \otimes F_2[v]\otimes \scH @>{\Delta^{\mathrm{op}}}>> F_1[u] \otimes F_2[v]\otimes \scH \\
    @V{(12)}VV @VV{(12)}V  \\
    F_2[v]\otimes F_1[u]\otimes \scH @>{\Delta}>> F_2[v]\otimes F_1[u]\otimes \scH \\
    @V{R_{F_2 F_1}(v-u)}VV @VV{R_{F_2 F_1}(v-u)}V \\
    F_2[v]\otimes F_1[u]\otimes \scH @>{\Delta^{\mathrm{op}}}>> F_2[v]\otimes F_1[u]\otimes \scH \\
    @V{(12)}VV @VV{(12)}V  \\
    F_1[u] \otimes F_2[v]\otimes \scH @>{\Delta}>> F_1[u] \otimes F_2[v]\otimes \scH\rlap{.}
  \end{CD}
\end{equation*}
Here we suppress the $1\otimes\varsigma_{v-u}$, etc in the notation of $\Delta$, $\Delta^{\mathrm{op}}$.
Comparing this diagram with \eqref{eq:h:4}, both $R_{F_1 F_2}(u-v)$
and $(12) R_{F_2 F_1}(v-u)^{-1} (12)$ intertwine $\Delta$ and
$\Delta^{\mathrm{op}}$. Moreover, both are normalized so that both
fix the highest weight vector. Therefore we are lead to the
unitarity equation
\begin{equation}\label{eq:h:9}
  (12) R_{F_2F_1}(v-u) (12) = R_{F_1 F_2}(u-v)^{-1}.
\end{equation}
This is a detailed version of the unitarity equation \eqref{eq:h:3}.
\end{NB}%

Consider $R_{F_2,F_1}(u_2-u_1)^{-1}(12)$ as a rational operator
$F_1[u_1]\otimes F_2[u_2]\to F_2[u_2]\otimes F_1[u_1]$.
It intertwines the $\sY$-module structures, hence
must coincide with $R_{F_1,F_2}(u_1-u_2)$ as the uniqueness
of the $R$-matrix, and also the normalization condition on the tensor product of the $\ell$-highest weight vectors.
We have the \emph{unitarity equation}:
\begin{equation}\label{eq:h:25}
  R_{F_2,F_1}(v-u)_{21} \defeq
  (12) R_{F_2,F_1}(v-u) (12) = R_{F_1,F_2}(u-v)^{-1}.
\end{equation}
See \cite[\S4.5]{MR3951025}. The subscript $21$ means the exchange
of the first and second factors, that is the conjugation by $(12)$.

\section{Involutions on quiver varieties}\label{sec:inv-quiver}

In this section, we will define an involution on a quiver variety.
This has been done already in \cite{MR3900699}, but we will explain from
the scratch, as our main example was only briefly mentioned in \cite[Remark~9.2.4]{MR3900699}.

{This section, as well as the appendix~\cref{sec:ADHM}, is an edit of an article \cite{2018arXiv180106286N},
written in response to \cite{MR3900699}. The article was not submitted to 
any journal.}

\subsection{Quiver varieties}\label{subsec:quiver_variety}

Let $(I,\Omega)$ be a quiver of type ADE. Let $\bv = (\bv_i)$,
$\bw = (\bw_i)\in\ZZ_{\ge 0}^{I}$ be dimension vectors for gauge
nodes and framing nodes respectively. Let $\fM_\zeta(\bv,\bw)$ be the
corresponding quiver variety. When $\bv$, $\bw$ are clear from the context, we simply denote it by $\fM_\zeta$.

We will use the convention about quiver varieties in \cite{Na-reflect} in
`holomorphic' description. In particular,
$\zeta = (\zeta_\RR,\zeta_\CC)\in (\sqrt{-1}\RR\times \CC)^I$. The
complex part $\zeta_\CC$ appears in the moment map equation
$\mu = -\zeta_\CC$, while the real part $\zeta_\RR$ appears in the
stability condition.
We drop the subscript $\CC$ from the complex part of the moment map $\mu_\CC$,
as we only use the holomorphic description. We also denote $\prod_i \GL(V_i)$ by
$G_{\bv}$ instead of $G^\CC_{\bv}$.
We say $\zeta$ is \emph{generic} if it is not contained in any of 
$(\sqrt{-1}\RR\times\CC)\otimes D_\theta$ for a hyperplane $D_\theta$
given by a positive root $\theta$ with $\theta\le\bv$.
cf.\ \cite[Prop.~2.6]{Na-reflect}.
If $\zeta$ is generic, $\fM_\zeta(\bv,\bw)$ is smooth of dimension
$\bv\cdot (2\bw - \bC\bv)$, where $\bC$ is the Cartan matrix.

A standard stability condition used in \cite[Def.~3.9]{Na-alg} is
one satisfying $\sqrt{-1}\zeta_{i,\RR} > 0$ for all $i\in I$.

The symplectic vector space, used to define $\fM_\zeta(\bv,\bw)$, is
denoted by $\bM$. When we want to emphasize the underlying $I$-graded vector
spaces $V = \bigoplus_i V_i$, $W=  \bigoplus_i W_i$, we denote it by $\bM(V,W)$.

We associate weights $\lambda$, $\mu$ of $\fg$ to $\bw$, $\bv$ by
\begin{equation}\label{eq:h:43}
  \lambda = \sum_{i\in I} \bw_i \varpi_i,\quad
  \mu = \lambda - \sum_{i\in I} \bv_i \alpha_i
  = \sum_{i\in I} (\bw_i - \sum_j \bC_{ij} \bv_j) \varpi_i
\end{equation}
as usual. We can also regard $\lambda$, $\mu$ as weights of the extended
Lie algebra $\fg^{\mathrm{ext}}$ in \cref{subsec:ext_Lie}.

The assumption that $(I,\Omega)$ is of type ADE is restrictive,
but it is necessary for the involution to be well-defined on
the quiver variety with \emph{generic} parameter $\zeta$, as the involution
first sends $\zeta$ to $-\zeta$, and then sends back to the original
$\zeta$ by the longest element of the Weyl group $w_0$ plus a diagram automorphism.
It means that we need to deal with possibly singular quiver varieties with non-generic $\zeta$,
more general $\sigma$-quiver varieties. For example, $\SO$ or $\grpSp$-instanton moduli
spaces on $\CC^2$ are $\sigma$-quiver varieties, but with parameter $\zeta = 0$.
See \cite{MR3508922}. More general $\sigma$-quiver varieties will be focus of the future work.

We have a natural projective morphism
$\pi\colon \fM_\zeta(\bv,\bw)\to \fM_{(0,\zeta_\CC)}(\bv,\bw)$, defined
in \cite[\S4]{Na-quiver} and \cite[(3.18)]{Na-alg}.
Here $(0,\zeta_\CC)$ is the parameter
obtained by replacing the real part $\zeta_\RR$ of $\zeta$ by $0$, but
keeping the complex parameter.

We also use the notation $\fM_\zeta(\bw)$, which is the disjoint union of
$\fM_\zeta(\bv,\bw)$ for all $\bv$ with fixed $\bw$.

\begin{Remark}[convention on $\CC_{\hbar}^\times$-action]\label{rem:CCx-hbar}
  Suppose $\zeta_\CC=0$. Then $\fM_\zeta(\bv,\bw)$ has a natural
  $\CC^\times$ action, which scales all linear maps in $\bM(V,W)$ by weight $1$.
  We denote this $\CC^\times$ by $\CC^\times_\hbar$ to distinguish it from the
  other $\CC^\times$-actions. 

  Let $\bq^m$ denote the character of $\CC_{\hbar}^\times$ with weight $m\in\ZZ$.
  We denote the Lie algebra element in the corresponding to $\bq$ by
  $\hbar/2$, i.e., $\bq^m = \exp(m\hbar/2)$. The division `$/2$' is necessary
  to make the geometric convention compatible with the standard defining relations
  of Yangian. See the formula below \cite[(4.1)]{Varagnolo}.
\end{Remark}

\subsection{Involution defined by McKay correspondence}\label{subsec:diagram}

Let $\Gamma$ be a finite subgroup of $\SU(2)$. Recall that the McKay correspondence
associate an irreducible representation $\rho_i$ of $\Gamma$ to each vertex $i$ of the Dynkin diagram of affine type ADE.

\begin{Definition}
  Let us define an involution $\invast$\footnote{
  This involution $\invast$ was denoted by $*$ in \cite[\S9]{Na-reflect}.
} 
  on the Dynkin diagram of affine type ADE by
  $\rho_i^*\cong\rho_{\invast[i]}$, where $\rho_i^*$ is the dual representation of $\rho_i$.
\end{Definition}

It fixes the trivial representation $\rho_0$, and hence induces a
diagram involution on the Dynkin diagram of finite type ADE. It is the same
diagram involution given by the longest element $w_0$ of the Weyl
group as $-w_0(\alpha_i) = \alpha_{\invast[i]}$, where $\alpha_i$ is the
simple root corresponding to the vertex $i$. 
Explicitly, it is given as follows according to the type of the Dynkin diagram
as in \cref{fig:diagram-involution}.
It is the identity for other types, not listed in \cref{fig:diagram-involution}.

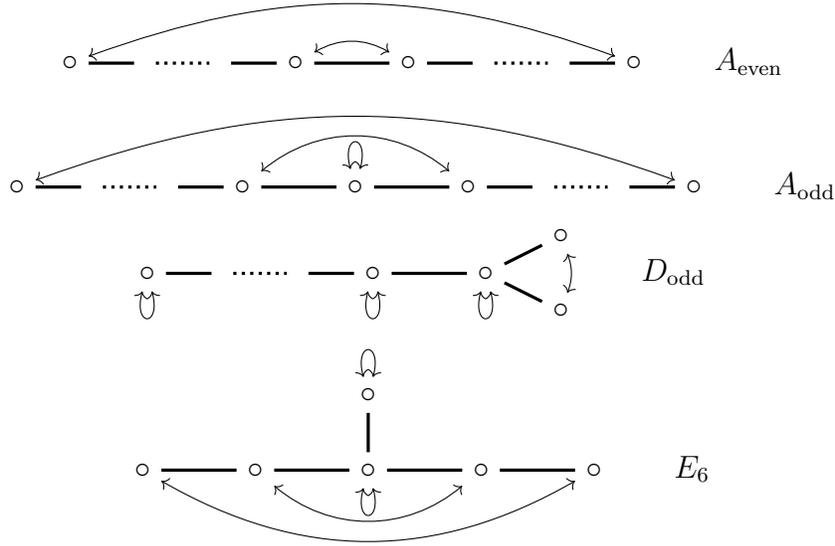
\begin{figure}[htbp]
\centering
\begin{gather*}
\begin{tikzpicture}[baseline=(a1.base)]
\node (a1) at (0,0) {$\circ$};
\node (a2) at (1,0) {};
\node (a3) at (2,0) {};
\node (a4) at (3,0) {$\circ$};
\node (a5) at (4.5,0) {$\circ$};
\node (a6) at (5.5,0) {};
\node (a7) at (6.5,0) {};
\node (a8) at (7.5,0) {$\circ$};
\draw[very thick] (a1) -- (a2);
\draw[dotted,very thick] (a2) -- (a3);
\draw[very thick] (a3) -- (a4) -- (a5) -- (a6);
\draw[dotted,very thick] (a6) -- (a7);
\draw[very thick] (a7) -- (a8);
\draw[<->, bend left=20] (a1) to (a8);
\draw[<->, bend left=30] (a4) to (a5);
\end{tikzpicture}
\qquad \text{$A_{\mathrm{even}}$}
\\
\begin{tikzpicture}[baseline=(a1.base)]
\node (a1) at (0,0) {$\circ$};
\node (a2) at (1,0) {};
\node (a3) at (2,0) {};
\node (a4) at (3,0) {$\circ$};
\node (a5) at (4.5,0) {$\circ$};
\node (a6) at (6,0) {$\circ$};
\node (a7) at (7,0) {};
\node (a8) at (8,0) {};
\node (a9) at (9,0) {$\circ$};
\draw[very thick] (a1) -- (a2);
\draw[dotted,very thick] (a2) -- (a3);
\draw[very thick] (a3) -- (a4) -- (a5) -- (a6) -- (a7);
\draw[dotted,very thick] (a7) -- (a8);
\draw[very thick] (a8) -- (a9);
\draw[<->, bend left=20] (a1) to (a9);
\draw[<->, bend left=40] (a4) to (a6);
\draw[<->] (a5) to[loop above] (a5);
\end{tikzpicture}
\qquad \text{$A_{\mathrm{odd}}$}
\\
\begin{tikzpicture}[baseline=(a1.base)]
  \node (a1) at (0,0.5) {$\circ$};
  \node (a2) at (1,0.5) {};
  \node (a3) at (2,0.5) {};
  \node (a4) at (3,0.5) {$\circ$};
  \node (a5) at (4.5,0.5) {$\circ$};
  \node (a6) at (5.5,0) {$\circ$};
  \node (a7) at (5.5,1) {$\circ$};
  \draw[very thick] (a1) -- (a2);
  \draw[dotted,very thick] (a2) -- (a3);
  \draw[very thick] (a3) -- (a4) -- (a5) -- (a6);;
  \draw[very thick] (a5) -- (a7);
  \draw[<->,bend right=20] (a6) to (a7);
  \draw[<->] (a1) to[loop below] (a1);
  \draw[<->] (a4) to[loop below] (a4);
  \draw[<->] (a5) to[loop below] (a5);
\end{tikzpicture}
\qquad \text{$D_{\mathrm{odd}}$}
\\
\begin{tikzpicture}[baseline=(a1.base)]
  \node (a1) at (0,0) {$\circ$};
  \node (a2) at (1.5,0) {$\circ$};
  \node (a3) at (3,0) {$\circ$};
  \node (a4) at (4.5,0) {$\circ$};
  \node (a5) at (6,0) {$\circ$};
  \node (a6) at (3,1) {$\circ$};
  \draw[very thick] (a1) -- (a2) -- (a3) -- (a4) -- (a5);
  \draw[very thick] (a3) -- (a6);
  \draw[<->] (a3) to[loop below] (a3);
  \draw[<->] (a6) to[loop above] (a6);
  \draw[<->, bend right=30] (a1) to (a5);
  \draw[<->, bend right=40] (a2) to (a4);
\end{tikzpicture}
\qquad \text{$E_6$}
\end{gather*}
\caption{The involution induced by the longest element of the Weyl group}
\label{fig:diagram-involution}
\end{figure}

\begin{NB}
In the labeling in
\cite[Ch.~4]{Kac}, it is given by $\invast[i] = \ell-i+1$ for type $\mathrm{A}_\ell$,
$\invast[1] = 5$, $\invast[2] = 4$, $\invast[3] = 3$, $\invast[4] = 2$, $\invast[5] = 1$, $\invast[6] = 6$ for
type $\mathrm{E}_6$ respectively.
For type $\mathrm{D}_\ell$ with odd $\ell$, it is given by
\begin{equation*}
    \invast[i] =
    \begin{cases}
        \ell -1 & \text{if $i=\ell$}, \\
        \ell & \text{if $i=\ell-1$},\\
        i & \text{otherwise},
    \end{cases}
\end{equation*}
In other words: it exchanges two `tails' for $D_\ell$ with odd $\ell$. It is a reflection at the
middle for $A_{\ell}$ and $E_6$.
It is the identity for other types.  
\end{NB}%

\begin{Lemma}
  Suppose $\invast[i] = i$. The irreducible representation $\rho_i$
  has either an orthogonal or a symplectic form, which is $\Gamma$-invariant.  
\end{Lemma}

\begin{proof}
By the assumption we have an isomorphism
$\rho_i\cong\rho_i^*$ of $\Gamma$-modules. Let us fix it and denote it
by $\varphi$. We take
$\lsp{t}\varphi\colon (\rho_i^*)^* = \rho_i\to \rho_i^*$. Since
$\rho_i$ is irreducible, there is a nonzero constant $\varepsilon$ with
$\lsp{t}\varphi= \varepsilon\varphi$. Since
$\lsp{t}{(\lsp{t}\varphi)} = \varphi$, $\varepsilon^2 = 1$, hence $\varepsilon = \pm 1$.
If $\varepsilon=1$ (resp.\ $-1$), $\varphi$ gives an orthogonal (resp.\
symplectic) form on $\rho_i$.
\end{proof}

The sign $\varepsilon$ above is \emph{not} a choice from the proof above, and is
automatically determined by the McKay graph as follows:
The trivial representation, assigned to $i=0$, is orthogonal.
For type $\mathrm{A}_{\ell}$ with odd $\ell$ with $i= {(\ell+1)/2}$, it is
orthogonal. This is because $\rho_i$ is $1$-dimensional.
For other types, $\rho_{i_0}$ for the vertex $i_0$ adjacent to the vertex $i=0$ 
in the affine Dynkin diagram is the
representation given by the inclusion $\Gamma\subset\SU(2)$. Therefore
$\rho_{i_0}\cong \rho_{i_0}^*$ is symplectic. If $i$ is adjacent to
$i_0$ and $\invast[i] = i$, then $\rho_i\cong\rho_i^*$ is orthogonal. 
This is because $\rho_i$ is a direct summand of $\rho_{i_0}\otimes\rho_{i_0}$
which has $\Gamma$-invariant orthogonal form.
If $j$ is adjacent to such an $i$ with $\invast[j]=j$, then $\rho_j\cong\rho_j^*$ is
symplectic, and so on.
As a side remark, two vertices $i\neq \invast[i]$ at the tail in type
$D_{\mathrm{odd}}$ satisfy this rule: If they were fixed by $\invast$,
then they must be symplectic vertices. But $\rho_i$,
$\rho_{\invast[i]}$ are $1$-dimensional representations of
$\Gamma$. Hence they cannot be fixed by $\invast$.

\begin{NB}
We choose an orientation $\varepsilon$ of the affine Dynkin diagram such that
$\varepsilon(\invast[h]) = \varepsilon(h)$ for all $h$, or
$\varepsilon(\invast[h]) = -\varepsilon(h)$ for all $h$. Then we have an induced isomorphism
$\fM_\zeta(\bv,\bw)\to \fM_{\invast[\zeta]}(\invast[\bv],\invast[\bw])$, where
$\invast[\zeta]_i = \zeta_{\invast[i]}$, $\invast[\bv]_i = \bv_{\invast[i]}$, $\invast[\bw]_i = \bw_{\invast[i]}$.
It is given by sending $B_h$, $a_i$, $b_i$ to $B_{\invast[h]}$, $a_{\invast[i]}$, $b_{\invast[i]}$ when
$\varepsilon(\invast[h]) = \varepsilon(h)$ for all $h$. Otherwise, we send them to
$-B_{\invast[h]}$, $a_{\invast[i]}$, $b_{\invast[i]}$.
\end{NB}%

We have an induced isomorphism
$\fM_\zeta(\bv,\bw)\to \fM_{\invast[\zeta]}(\invast[\bv],\invast[\bw])$, where
$\invast[\zeta]_i = \zeta_{\invast[i]}$, $\invast[\bv]_i = \bv_{\invast[i]}$, $\invast[\bw]_i = \bw_{\invast[i]}$.
It is given by sending $B_h$, $a_i$, $b_i$ to $B_{\invast[h]}$, $a_{\invast[i]}$, $b_{\invast[i]}$.
%
We also change the orientation $\Omega$ to $\invast[\Omega]$, hence $\varepsilon$, by
$\invast$, i.e., we use $\varepsilon_{\invast[\Omega]}$ instead of
$\varepsilon$ where
$\varepsilon_{\invast[\Omega]}(h) = \varepsilon(\invast[h])$.

\subsection{Transpose}\label{subsec:transpose}

We replace vector spaces $V = \bigoplus_i V_i$, 
$\bigoplus_i W_i$ by their dual spaces $V^* =
\bigoplus_i V_i^*$, $W^* = \bigoplus_i W_i^*$,
and define the symplectic vector space $\bM(V^*,W^*)$ as in 
\cite[\S2(ii)]{Na-reflect}. 
We also replace the orientation $\Omega$ by the opposite one $\bar\Omega$.
This is used for the definition of the symplectic form on $\bM(V^*,W^*)$, and
hence the definition of the moment map $\mu$.
\begin{NB}
Therefore we use $\varepsilon_{\overline\Omega}(h) = -\varepsilon(h)$ in
place of $\varepsilon$ in the definition of quiver varieties.
\end{NB}

\begin{Definition}
We define an isomorphism $t\colon \bM(V,W)\to \bM(V^*,W^*)$ by
replacing linear maps $B_h$, $a_i$, $b_i$ to
\begin{equation*}
  \lsp{t}B_{\overline{h}}\colon V_{\vout{h}}^*\to V_{\vin{h}}^*,
  \quad
  - \lsp{t}b_i \colon W_i^*\to V_i^*,
  \quad
  \lsp{t}a_i\colon V_i^*\to W_i^*.
\end{equation*}  
\end{Definition}

The parameter $\zeta$ for the hyperk\"ahler moment map 
equation is replaced by $-\zeta$.
\begin{NB}
  $\lsp{t}{\left(\varepsilon(h) B_h B_{\overline{h}}\right)}
  = \varepsilon(h) \lsp{t}B_{\overline{h}} \lsp{t}B_h
  = -\varepsilon_{\bar\Omega}(h) \lsp{t}B_{\overline{h}}\,\lsp{t}B_h$
\end{NB}%
Hence this construction gives an involutive isomorphism
\begin{equation*}
  t\colon \fM_\zeta(\bv,\bw)\to
  \fM_{-\zeta}(\bv,\bw).
\end{equation*}

A little more precisely, the space $\fM_\zeta(\bv,\bw)$
depends on the $I$-graded vector space $W$, not only on its dimension
vector $\bw$. (On the other hand, it does not depend on $V$ as we take
the quotient by $\prod\GL(V_i)$.) 
Thus $\fM_{-\zeta}(\bv,\bw)$ is associated with $W^*$, instead of $W$.

Let us combine this isomorphism with one in
\cref{subsec:diagram}, i.e., consider
$t\circ\invast = \invast\circ t$. Then we replace
a $\Gamma$-representation
$\bigoplus W_i\otimes\rho_i$ with its dual
$\bigoplus W_i^*\otimes\rho_i^* = \bigoplus W_i^* \otimes\rho_{\invast[i]}$.
It is more natural in view of \cite{KN}.

\subsection{Reflection functor}

Let $S_{w_0}$ be the reflection functor associated with the longest element $w_0$ of the Weyl group.
(It was denoted by $\mathcal F$ in \cite[(9.1)]{Na-reflect}.)
We use the definition given in \cite{Hyogen1992,Na-reflect}. See \cite{MR1775358,MR1990675} for other definitions.

If $w_0 = s_{i_1}\dots s_{i_\nu}$
($i_p\in I$) is a reduced expression, we have
$S_{w_0} = S_{i_1} \dots S_{i_\nu}$, where $S_{i_p}$ is the reflection functor for a
simple reflection $s_{i_p}\equiv s_i$ defined as follows. We assume
$(\zeta_{i,\RR},\zeta_{i,\CC})\neq(0,0)$, and
choose $z_1$, $z_2\in\CC$ so that
$z_1 z_2 = \zeta_{i,\CC}$,
$-|z_1|^2 + |z_2|^2 = 2\sqrt{-1}\zeta_{i,\RR}$. They are unique up to
$(tz_1, t^{-1}z_2)$ for $t\in \U(1)$. The ambiguity can be absorbed by
the action of the gauge group $\prod_i \GL(V'_i)$ used in the definition of quiver varieties.

We consider a complex
\begin{equation}\label{eq:33}
  V_i \xrightarrow{\alpha_i} \widetilde V_i \defeq V_i\oplus \bigoplus_{\substack{h\in H\\ \vin{h}=i}} V_{\vout{h}} \oplus W_i
  \xrightarrow{\beta_i} V_i
\end{equation}
where
\begin{equation*}
  \alpha_i =
  \begin{bmatrix}
    z_1 \id_{V_i} \\ \bigoplus
    B_{\overline{h}} \\ b_i
  \end{bmatrix},
  \qquad
  \beta_i =
  \begin{bmatrix}
    z_2\id_{V_i} & \sum
    \varepsilon(h) B_{h} & a_i
  \end{bmatrix}.
\end{equation*}
Then this is indeed a complex, moreover $\alpha_i$ is injective, and $\beta_i$
is surjective. We define $V'_i \defeq \Ker\beta_i/\Ima\alpha_i$. We
further define $B'_{h}$, $B'_{\overline{h}}$ for $\vin{h}=i$ as
\begin{equation*}
  \begin{gathered}
  B'_{h} =
  \begin{bmatrix}
    B_h \\ -\varepsilon(h) z_2 \iota_h\\ 0
  \end{bmatrix}\colon V_{\vout{h}} \to \widetilde V_i, \quad
  B'_{\overline{h}} =
  \begin{bmatrix}
    B_{\overline{h}} & -z_1 \pi_h & 0
  \end{bmatrix}\colon \widetilde V_i \to V_{\vout{h}}, \\
  a'_i =
  \begin{bmatrix}
    a_i \\ 0 \\ -z_2 \iota_{W_i}
  \end{bmatrix}\colon W_i \to \widetilde V_i, \quad
  b'_i =
  \begin{bmatrix}
    b_i & 0 - z_1 \pi_{W_i}
  \end{bmatrix}\colon \widetilde V_i\to W_i,
  \end{gathered}
\end{equation*}
where $\iota_h$, $\iota_{W_i}$ are the natural inclusions
of $V_{\vout{h}}$ and $W_i$ into $\widetilde V_i$, and
$\pi_h$, $\pi_{W_i}$ are the natural projections from $\widetilde V_i$ to $V_{\vout{h}}$ and $W_i$ respectively.
One see that $\beta_i B'_h = 0$, $B'_{\overline{h}} \alpha_i = 0$,
$\beta_i a'_i = 0$, $b'_i \alpha_i = 0$. Therefore they are well-defined linear maps
between $V'_i$ and $V_{\vout{h}}$, $W_i$.
\begin{NB}
  We have
  \begin{equation*}
    C'_{a-\frac12} D'_{a-\frac12} =
    \begin{bmatrix}
      C_{a-\frac12} D_{a-\frac12} & -z_1 C_{a-\frac12} & 0 \\
      -z_2 D_{a-\frac12} & z_1 z_2 & 0 \\
      0 & 0 & 0
    \end{bmatrix},\quad
    D'_{a+\frac12} C'_{a+\frac12} =
    \begin{bmatrix}
      D_{a+\frac12} C_{a+\frac12} & 0 & - z_1 D_{a+\frac12} \\
      0 & 0 & 0 \\
      z_2 C_{a+\frac12} & 0 & -z_1z_2
    \end{bmatrix}.
  \end{equation*}
  Suppose $\left[
  \begin{smallmatrix}
    v_i \\ v_{a-1} \\ v_{a+1}
  \end{smallmatrix}\right]$ is in $\ker\beta_i$, i.e.,
  $z_2 v_i + C_{a-\frac123} v_{a-1} - D_{a+\frac12} v_{a+1} = 0$. Then
  \begin{equation*}
    (C'_{a-\frac12} D'_{a-\frac12} - D'_{a+\frac12} C'_{a+\frac12})
    \begin{bmatrix}
    v_i \\ v_{a-1} \\ v_{a+1}
  \end{bmatrix} =
  \begin{bmatrix}
    0 \\ -z_2 D_{a-\frac12} v_i + z_1z_2 v_{a-1} \\
    -z_2 C_{a+\frac12} v_i + z_1 z_2 v_{a+1}
  \end{bmatrix}.
\end{equation*}
This is equal to $(\zeta_{a+\frac12,\CC}-\zeta_{a-\frac12,\CC})\left[
  \begin{smallmatrix}
    v_i \\ v_{a-1} \\ v_{a+1}
  \end{smallmatrix}\right]$ modulo
$\Ima \alpha_i$.
\end{NB}%
Moreover they satisfy the moment map equation
$\mu(B',a',b') = -s_i(\zeta_\CC)$, as well as
the $s_i(\zeta_\RR)$-stability condition. See \cite[Cor.~3.7]{Na-reflect}.
This definition literally applies to the last operator $S_{i_\nu}$ for
$i=i_\nu$. For the next $S_{i_{\nu-1}}$ we use the above definition
with the parameter $s_{i_\nu}\zeta$. And so on.

The composition $S_{w_0} = S_{i_1}\dots S_{i_\nu}$ gives an isomorphism
\begin{equation*}
  \fM_{\zeta} \to \fM_{-\invast[\zeta]}
\end{equation*}
as $\zeta = w_0(-\invast[\zeta])$. In general, $S_{w_0}$ changes
dimensions of vector spaces: $S_{w_0}$ does not change
${\bw}$, but sends ${\bv}$ to
$w_0\ast {\bv}$ given by
${\bw} - \mathbf C(w_0\ast {\bv}) =
w_0({\bw}-\mathbf C{\bv})$.
In terms of weights $\lambda$, $\mu$, they are changed as
$\lambda\mapsto \lambda$, $\mu\mapsto w_0(\mu) = -\invast[\mu]$ respectively.
Here $\invast$ on weights is defined in an obvious way.
%

When $\zeta_\RR = 0$, $\fM_\zeta$ is an affine algebraic variety. The
definition of a reflection functor in \cite{MR1775358} is defined for
the coordinate ring of $\fM_\zeta$ and \emph{does} make sense even
when $\zeta$ is not necessarily generic.

Suppose $\zeta$ is generic again. Recall $(0,\zeta_\CC)$ the parameter
obtained by replacing the real part $\zeta_\RR$ of $\zeta$ by $0$, but
keeping the complex parameter. It may be generic or not. We have a
projective morphism $\pi\colon \fM_\zeta\to \fM_{(0,\zeta_\CC)}$. The
reflection functor in \cite{MR1775358} is defined for
$\fM_{(0,\zeta_\CC)}$ even if $(0,\zeta_\CC)$ is not generic, as we
remarked above. Moreover $\pi$ commutes with reflection functors
\cite[\S8]{Na-reflect}.
Furthermore, the reflection functor in \cite{MR1775358} is
the identity when $\zeta_\CC = 0$.

\begin{Example}
  Consider a type $\mathrm{A}_1$ quiver variety with dimension vector $\bv$,
  $\bw$ with $\bv\le \bw$. Suppose $\zeta_\CC\neq 0$. By
  $\fM_\zeta(\bv,\bw)\ni [(a,b)] \mapsto ba\in\End(W)$, we see that
  $\fM_\zeta(\bv,\bw)$ is isomorphic to the adjoint orbit of
  semisimple matrices with eigenvalue $0$ with multiplicity $\bw-\bv$,
  and eigenvalue $-\zeta$ with multiplicity $\bv$. More concretely,
  the eigenspace for the eigenvalue $0$ (resp.\ $\zeta$) is $\Ker a$
  (resp.\ $\Ima b$). By the reflection functor $S$, we get
  \begin{equation*}
    \End(W)\ni b'a' = ba + \zeta,  \quad
    V' \cong W/\Ima b \cong \Ker a.
  \end{equation*}
  Note $ba + \zeta$ has eigenvalues $\zeta$ and $0$, and their eigenspaces
  are identified with $\Ker a$ and $\Ima b$ respectively.

  Next suppose $\zeta_\CC = 0$ and $\sqrt{-1}\zeta_\RR > 0$. Then $b$
  is injective, and $\fM_\zeta(\bv,\bw)$ is isomorphic to the
  cotangent bundle of the Grassmannian of $\bv$-dimensional subspaces
  in $W$. The isomorphism is given by $(ba,\Ima b)$ while $ba$ is
  viewed as $\Hom(W/\Ima b, \Ima b)$. Under the reflection functor
  $S$, $\fM_\zeta(\bv,\bw)$ is sent to $\fM_{-\zeta}(\bw-\bv,\bw)$,
  which is the cotangent bundle of the Grassmannian of
  $(\bw-\bv)$-dimensional \emph{quotients} of $W$. Then $S$ is given by
  $(ba, W/\Ima a)$.
\end{Example}

\subsection{Tautological bundles under the reflection functor}

Let $\CV_i$ be the tautological bundle on $\fM_\zeta(\bv,\bw)$
associated with the vertex $i$. Similarly, let $\CW_i$ be the
tautological bundle associated with the framing node $i$.
Under the reflection $S_i$ at $i$, $\CV_i$ is changed to $\CV'_i$.
By the definition of the reflection functor, we have the following
relation in the K-group of vector bundles on
$\fM_\zeta(\bv,\bw)$:
\begin{equation*}
  \CV'_i = \CW_i \oplus 
  \bigoplus_{\substack{h\in H\\\vin{h}=i}} \CV_{\vout{h}}
  \ominus \CV_i,
\end{equation*}
where the left hand side is regarded as a class on $\fM_\zeta(\bv,\bw)$
by the pull-back by $S_i$.
We use the Cartan matrix $\mathbf C = (a_{ij})$ to define a
K-theory class
\begin{equation}\label{eq:h:13}
  \CU_i \defeq \CW_i + \sum_j (-a_{ij})\CV_j
  = \CW_i\oplus\bigoplus_{\substack{h\in H\\\vin{h}=i}} 
  \CV_{\vout{h}} \ominus \CV_i^{\oplus 2}.
\end{equation}
Then we have $\CU'_j = \CU_j - \sum_j a_{ij} \CU_i$, i.e., 
$\CU_j$ changes according to the usual formula for change of
weights under the simple reflection.
By the composition, we have 
\begin{equation*}
  \CU_i' = -\CU_{\invast[i]},
\end{equation*}
where $\CU_i'$ is defined on $\fM_{-\invast[\zeta]}(w_0\ast\bv, \bw)$,
and pull-backed to $\fM_\zeta(\bv,\bw)$ by $S_{w_0}$ in this
equation.

We would like to consider $\CC_{\hbar}^\times$-equivariant K-classes
when $\zeta_\CC = 0$ (and generic $\zeta_\RR$).
Recall that $\bq^m$ denote the character of $\CC_{\hbar}^\times$ with weight $m\in\ZZ$.
We consider it as an equivariant $K$-theory class.

Let us first note that $z_1 = 0$ or $z_2=0$ is $0$ if $\zeta_{i,\CC} = 0$
according to the sign of $\sqrt{-1}\zeta_{i,\RR}$ in the definition of $S_i$.
If $z_2 = 0$, we modify \eqref{eq:33} as
\begin{equation*}
  \bq^{-2} V_i \xrightarrow{\alpha_i} 
  \bq^{-2} V_i \oplus \bigoplus \bq^{-1} V_{\vout{h}} 
    \oplus \bq^{-1} W_i
  \xrightarrow{\beta_i} V_i.
\end{equation*}
It becomes a $\CC_{\hbar}^\times$-equivariant complex. If $z_1 = 0$,
the modification is almost the same, but put $V_i$ instead of
$\bq^{-2} V_i$ in the middle term.
The difference arises as we need to make $z_1\id_{V_i}$ or $z_2\id_{V_i}$
$\CC_{\hbar}^\times$-equivariant.
Then we define $V'_i$ as $\bq^2\Ker\beta_i/\Ima\alpha_i$
if $z_2 = 0$, and $\Ker\beta_i/\Ima\alpha_i$ if $z_1 = 0$.
This makes $S_i$ a $\CC_{\hbar}^\times$-equivariant isomorphism.

\begin{NB}
  Note also that $\begin{bmatrix} \bigoplus B_{\overline{h}} \\ b_i \end{bmatrix}$
  is injective if $z_1 = 0$ (i.e., $\sqrt{-1}\zeta_i^\RR > 0$), while $\begin{bmatrix}
    \sum \varepsilon(h) B_{h} & a_i
  \end{bmatrix}$ is surjective if $z_2 = 0$ (i.e., $\sqrt{-1}\zeta_i^\RR < 0$).
  This is useful to check which stability condition is used.
\end{NB}%

We modify \eqref{eq:h:13} as
\begin{equation}\label{eq:h:44}
  \CU_i \defeq \bq^{-1}\CW_i + \sum_j \bq^{-1}[-a_{ij}]\CV_j,
\end{equation}
where $[m]$ is the $q$-integer defined by
$[m] = \frac{\bq^m - \bq^{-m}}{\bq-\bq^{-1}}$.
In practice, $m= 0$, $-1$, or $2$, hence $[0] = 0$, $[-1] = -1$,
$[2] = \bq + \bq^{-1}$.
We have
\begin{equation}\label{eq:h:14}
  \CU'_j
  = \CU_j - \sum_j \bq^{\pm}[a_{ij}] \CU_i,
\end{equation}
where $\pm$ is $+$ if $z_2 = 0$, $-$ if $z_1=0$.
This computation was given in \cite[(4.6.3)]{VV-can}.

For $S_{w_0}$, we suppose $\fM_\zeta(\bv,\bw)$ appearing in the
domain of $S_{w_0}$ is one for $\sqrt{-1}\zeta_{i,\RR} < 0$ for all $i$.
Then it follows from \eqref{eq:h:14} that
\begin{equation}\label{eq:h:49}
  \CU_i'
  = - \bq^\bc\,\CU_{\invast[i]},
\end{equation}
where $\bc$ is the Coxeter number of the Dynkin diagram.
See \cite[Lemma~4.6]{VV-can}.\footnote{
Note that $\omega$ in \cite[Lemma~4.6]{VV-can} is the composition of $S_{w_0}$ and the transpose $t$. Therefore 
the formula involves the dual of $\CU_i$.}

\begin{NB}
  We need to understand the tangent bundle of the $\sigma$-quiver variety
  $\fM_\zeta^\sigma(\bv,\bw)$. The following might be useful.
\begin{equation*}
  T\fM \cong \sum_{i,j} \bC_q^{ij}\Hom(\CW_i,\CW_j)
  - \bC_q^{ij}\Hom(\CU_i, \CU_j). 
\end{equation*}
\end{NB}%

\begin{Example}
  Consider type $\mathrm{A}_{\ell-1}$ quiver variety with $\bv = 0$, $\bw = \varpi_1$,
  where the numbering of vertices is $1$, $2$, \dots, $\ell-1$ from the left.
  Then $\fM_\zeta$ is a single point. 
  Suppose $\sqrt{-1}\zeta_{i,\RR} < 0$ for all $i$ as above. Then
  $\fM_{-\invast[\zeta]}(w_0\ast\bv,\bw)$ is given in the notation of
  \eqref{eq:31} by $D_{1/2} = D_{3/2} = \dots = D_{\ell-3/2} = 1$,
  $C_{1/2} = C_{3/2} = \dots = C_{\ell-3/2} = 0$,
  where $W_1 = V_1 = \dots = V_{\ell-1} = \CC$, $W_{\ell-1} = 0$.
  We have
  \begin{equation*}
    \CU_1 = \bq^{-1}, \qquad
    \CU_{\ell-1}' = -\bq^{\ell-1}.
  \end{equation*}
  Thus \eqref{eq:h:49} is checked. If we use the opposite stability parameter
  $\sqrt{-1}\zeta_{i,\RR} > 0$ for all $i$, we have $D_{1/2} = D_{3/2} = \dots = D_{\ell-3/2} = 0$,
  $C_{1/2} = C_{3/2} = \dots = C_{\ell-3/2} = 1$ instead. Then
  $\CU_1 = \bq^{-1}$ is unchanged, but $\CU_{\ell-1}' = -\bq^{-\ell-1}$.
\end{Example}

\subsection{Orthogonal/symplectic forms on framing vector spaces
}\label{subsec:form}

Let $\sigma'$ be a diagram automorphism. Except the $D_{\mathrm{even}}$ case, it is either $\id$ or $\invast$.
There is a nontrivial diagram automorphism for $D_{\mathrm{even}}$, which is
given by swap of the two tails as in the $D_{\mathrm{odd}}$ case. Therefore
$\sigma'$ is either $\id$ or this nontrivial automorphism.
The diagram automorphism is a choice. A different choice will give a different
involution on $\fM_\zeta(\bv,\bw)$.

\begin{Definition}\label{def:involution}
We consider the composite
$\sigma \defeq \sigma'\circ S_{w_0}\circ\invast\circ t$, which is an
isomorphism
\begin{equation*}
  \fM_\zeta(\bv,\bw) \xrightarrow{\cong}
  \fM_{\sigma(\zeta)}(\sigma(\bv),\sigma(\bw)),
\end{equation*}
where
\begin{equation*}
    \sigma(\zeta) = \sigma'(\zeta),\quad
    \sigma(\bv) = \sigma'(w_0\ast\invast[\bv]),\quad
    \sigma(\bw) = \sigma'(\invast[\bw]).
\end{equation*}
\end{Definition}

For the change of $\zeta$, we have used $w_0(-\invast[\zeta]) = \zeta$.
Changes of $\bv$, $\bw$ are obvious. In terms of weights,
\begin{equation}\label{eq:h:50}
  \sigma(\lambda) = \sigma'(\invast[\lambda]), \qquad
  \sigma(\mu) = -\sigma'(\mu).
\end{equation}

In order to make $\sigma$ an involution, we assume
$\sigma(\zeta) = \zeta$, $\sigma(\bv) = \bv$, $\sigma(\bw) = \bw$ hereafter.

Since the quiver varieties depend on the $I$-graded vector space $W$, the equality $\sigma(\bw) = \bw$
is not sufficient. We need to fix an identification between $W_i$ and
$W^*_{\sigma'(\invast[i])}$. This will be done in the
remainder of this subsection.
\begin{NB}
Here we view $\sigma(W)$ as a vector space $W^*$ equipped with the $I$-grading
given by $\sigma(W)_{i} = W_{\sigma'(\invast[i])}^*$. We cannot view
$\sigma$ as a map from $W$ to $W^*$, contrary to the notation.
\end{NB}%

We first choose an orthogonal $(+)$ or symplectic $(-)$ type, which determines
signs $\varepsilon_i$ on vertices $i$ by the rule described below.
If we exchange $(+)$ and $(-)$, each $\varepsilon_i$ changes the sign for
all $i$ simultaneously. This is the consequence of the rule below.

Consider the composition $\sigma'\circ\invast$. Looking at
\cref{fig:diagram-involution}, we find that the non
$(\sigma'\circ\invast)$-fixed edges forms a subdiagram of type
$A_{\mathrm{odd}}$ if $Q\neq A_{\mathrm{even}}$ and
$\sigma'\circ\invast\neq\id$. Let us first consider this case.

Consider a vertex $i$, which is fixed by $\sigma'\circ\invast$. Since
we assume $\sigma'\circ\invast\neq\id$, we have $\sigma'=\id$ if $Q$
is not of type $\mathrm{D}_{\mathrm{even}}$, and $\invast=\id$ if $Q$ is of
type $\mathrm{D}_{\mathrm{even}}$. Therefore the vertex $i$ is also fixed by
$\invast$.
Recall that we have an isomorphism $\rho_i\cong\rho_i^*$ for
$i=\invast[i]$ gives either an orthogonal or symplectic form on
$\rho_i$. We put an $\varepsilon_i$-form on $W_i$ so that the
tensor product form on $W_i\otimes\rho_i$ is the given type $(+)$ or
$(-)$.
Forms on $W_i$ and $W_j$ are alternative when $i$ and $j$ are adjacent
and $\invast[i] = i$, $\invast[j] = j$.
We have an isomorphism $\varphi_{i}\colon W_i\to W_i^*$
given by the form. It satisfies $\lsp{t}\varphi_{i} = \varepsilon_i \varphi_{i}$.

Let $i_1$ be the middle of the subdiagram. It is fixed by
$\sigma'\circ\invast$, hence we have already assigned an $\varepsilon_{i_1}$-form
on $W_{i_1}$.
For other vertices $i$ in the subdiagram of type $\mathrm{A}_{\mathrm{odd}}$, we
assign the form of the \emph{same} type as $W_{i_1}$ on
$W_i\oplus W_{\sigma'(\invast[i])}$ so that
$W_i\perp W_i$, $W_{\sigma'(\invast[i])}\perp W_{\sigma'(\invast[i])}$.
It means that we choose a pair of isomorphisms
$\varphi_i\colon W_i\to W_{\sigma'(\invast[i])}^*$,
$\varphi_{\sigma'(\invast[i])}\colon W_{\sigma'(\invast[i])}\to W_i^*$ such that
$\lsp{t}\varphi_i = \varepsilon_i \varphi_{\sigma'(\invast[i])}$. The sign $\varepsilon_i$ is
equal to $\varepsilon_{i_1}$ given by the form on $W_{i_1}$.

Next consider the case $Q=A_{\mathrm{even}}$ and
$\sigma'\circ\invast\neq\id$. There is no vertex fixed by
$\sigma'\circ\invast$, but the rule is similar to above: We assign an
$\varepsilon_i$-form on $W_i\oplus W_{\sigma'(\invast[i])}$ according to the
given type $(\pm)$.
\begin{NB}
  For type $\mathrm{A}_{\mathrm{odd}}$, the middle vertex $i_0$ is an
  orthogonal type.  Therefore, the rule above for $\mathrm{A}_{\mathrm{even}}$
  is the same as for $\mathrm{A}_{\mathrm{odd}}$.
\end{NB}%

The rest is $\sigma'\circ\invast=\id$.
Suppose $Q$ is not of type $\mathrm{A}$.
We put an $\varepsilon_i$-form on $W_i$ so that
the tensor product form on $W_i\otimes\rho_i$ is given type when
$\invast[i] = i$.
We then continue to put $\varepsilon_i$-forms on $W_i$ for the remaining
vertices alternatively, i.e., the type of the forms on $W_i$ and $W_j$
are opposite if $i$ and $j$ are adjacent.
Next suppose $Q$ is of type $\mathrm{A}_{\ell-1}$. There is a vertex, denoted
by $1$, which gives a vector representation of the corresponding Lie
algebra of type $\algsl_\ell$. It is leftmost or rightmost one
depending on the convention.
We assign an $\varepsilon_1$-form on $W_1$ according to the given type.
Then we assign $\varepsilon_i$-forms alternatively.

\begin{Definition}\label{def:framing-involution}
  Let $G_\bw = \prod_i \GL(W_i)$. It acts on $\fM_\zeta(\bv,\bw)$.

  We introduce the corresponding group for $\sigma$ version: Let $G_\bw^\sigma$ denote
  \begin{equation*}
    \prod_{i: \sigma'(\invast[i])=i} G_{\ve_i}(W_i)
    \times \prod_{\text{$\{ i, j \}$ as below}} \GL(W_i),
  \end{equation*}
  where the product in the second factor runs over \emph{unordered} pairs
  $\{ i,j\}$ with $\sigma'(\invast[i])=j$, but $i\neq j$.
\end{Definition}

The notation is justified, as $G_\bw^\sigma$ is the fixed point subgroup
of $G_\bw$ with respect to an involution induced from $\sigma$:
it sends $g_i\in\GL(W_i)$ to $\varphi_i^{-1}\,\lsp{t}g_j^{-1}\,\varphi_i$ with $j = \sigma'(\invast[i])$.
If it is fixed by $\sigma$, $g_j$ is determined by $g_i$, hence the factor in $G_\bw^\sigma$ is $\GL(W_i)$.
We can choose $j$ from the pair instead.

\begin{NB}
  Old version:

Suppose $\sigma'$ is the identity first.
Recall that we have an isomorphism $\rho_i\cong\rho_i^*$ for $i=\invast[i]$ gives
either an orthogonal or symplectic form on $\rho_i$. We put
an orthogonal or symplectic form on $W_i$ so that the tensor product
form on $W_i\otimes\rho_i$ is the given type $(+)$ or $(-)$.
Forms on $W_i$ and $W_j$ are alternative when $i$ and $j$ are adjacent
and $\invast[i] = i$, $\invast[j] = j$.
We have an isomorphism $\varphi_{W_i}\colon W_i\to W_i^*$
given by the form. We have $\lsp{t}\varphi_{W_i} = \pm \varphi_{W_i}$.

Except in type $\mathrm{A}_{\mathrm{even}}$, non $\invast$-fixed edges form a subdiagram of type $\mathrm{A}_{\mathrm{odd}}$. The middle vertex $i_1$ of the subdiagram
is fixed by $\invast$, hence we have an orthogonal or symplectic form
on $W_{i_1}$. We then assign a form of the same type on
$W_i\oplus W_{\invast[i]}$ so that $W_i\perp W_i$, $W_{\invast[i]}\perp W_{\invast[i]}$.
It means that we choose a pair of isomorphisms $\varphi_{W_i}\colon W_i\to W_{\invast[i]}^*$,
$\varphi_{W_{\invast[i]}}\colon W_{\invast[i]}\to W_i^*$
such that $\lsp{t}\varphi_{W_i} = \pm \varphi_{W_{\invast[i]}}$
according to the type of the form on $W_{i_1}$.
For type $\mathrm{A}_{\mathrm{even}}$, we put an orthogonal or symplectic form on $W_i\oplus W_{\invast[i]}$
according to the given type $(+)$ or $(-)$.

Next suppose $Q$ is of type $\mathrm{D}_{\mathrm{even}}$ and $\sigma'$ is nontrivial.
This is very similar to the case of type $\mathrm{D}_{\mathrm{odd}}$ with $\sigma' = \id$.
We put an orthogonal or symplectic form on $W_i$ when $i=\sigma'(i)$
according to the type of $\rho_i$ and the given type of the quiver.
We then put an orthogonal or symplectic form on $W_i\oplus W_{\sigma'(i)}$
exactly as in type $\mathrm{D}_{\mathrm{odd}}$ by looking at the form $W_{i_0}$
at the middle of the subdiagram of type $\mathrm{A}_3$.

Finally, suppose $Q$ is not of type $\mathrm{D}_{\mathrm{even}}$ and $\sigma'$ is nontrivial.
Then $\sigma'\circ\invast$ is the identity.
We put an orthogonal or symplectic form on $W_i$ alternatively,
starting from the leftmost vertex. The type of the form on the leftmost
vertex is the given one. The adjacent one is the opposite type, and so on.
\end{NB}%

\begin{NB}
The following computation is not precise enough, as we fix $V_i$ implicitly,
though we are working on the quotient space.

We now compose $\varphi_i \colon W_i\to W^*_{\sigma'(\invast[i])} = \sigma(W)_i$ as
$\sigma(a)_i \circ \varphi_i$, $\varphi_i^{-1}\sigma(b)_i$ so that linear maps
are defined for $W$. After the composition, $\sigma^2$ sends as
\begin{equation}\label{eq:h:15}
  \sigma^2(a)_i = - a_i \circ \lsp{t}\varphi_{\sigma'(\invast[i])}^{-1}\varphi_i
  = - \varepsilon_i a_i,
  \quad
  \sigma^2(b)_i = - \varphi_i^{-1}\lsp{t}\varphi_{\sigma'(\invast[i])} \circ b_i
  = - \varepsilon_i b_i.
\end{equation}
  Note $\lsp{t}\varphi_{\sigma'(\invast[i])}\colon W_i\to W_{\sigma'(\invast[i])}^*$.
\end{NB}%

\subsection{Orientation}\label{subsec:orientation}

The quiver varieties depend on the choice of the orientation, hence
the domain and target of $\sigma$ are still different spaces. We fix
this issue in this subsection.

We choose an orientation $\Omega$ according to $\invast$ and $\sigma'$
as follows.

Suppose that either (i) $\sigma'\circ\invast$ contains a subdiagram of type A
of non-fixed edges,
or (ii) $Q$ is of type $\mathrm{A}_{\mathrm{even}}$ and $\sigma' = \id$.
In either case, we choose a linear orientation on the subdiagram
or the entire diagram in case (ii).
For other parts, including the case $\sigma' \circ \invast = \id$, 
we choose an arbitrary orientation.

Under the transpose $t$, we change the orientation to the opposite
one. Then we change it further by $\sigma'\circ\invast$, as
$S_{w_0}$ does not change the orientation. Therefore the linear
orientation of the subdiagram of type A is preserved under $\sigma$.

On the other hand, the orientation in the $\sigma'\circ\invast$-fixed
part is still opposite. Therefore we compose $\sigma$ with the multiplication of $-\varepsilon(h)$ 
 for edges $h$, not in the subdiagram. See \eqref{eq:h:52}.
\begin{NB}
  Note
  $\varepsilon(h) (-\varepsilon(h) B_h)\, (-\varepsilon(\overline{h})
  B_{\overline{h}}) = -\varepsilon(h) B_h B_{\overline{h}}
  = \varepsilon_{\overline{\Omega}}(h) B_h B_{\overline{h}}$. Hence the moment
  map equation is preserved. Moreover, by the same reason, the symplectic
  form (defined via $\varepsilon$ vs $\varepsilon_{\overline{\Omega}}$) is
  preserved.
\end{NB}%
We denote the composite also by $\sigma$ for brevity. Thus we have defined
an automorphism $\sigma$ on $\fM_\zeta(\bv,\bw)$.

\subsection{Categorical quotients}

Let us calculate $\sigma$ on $\fM_0(\bv,\bw)$, the categorical quotient at the
level set $\mu = 0$. Then $S_{w_0}$ is the identity, hence the calculation is greatly simplified.
In particular, we define an involution $\sigma$ on the symplectic vector space $\bM$
(see \cite[\S2(i)]{Na-reflect}), which is used to define $\fM$ as a symplectic quotient.

We put forms on $V_i$ or $V_i\oplus V_{\sigma'(\invast[i])}$, according to whether $\sigma'(\invast[i]) = i$ or not, as in the previous subsection.
The rule is that the sign is the \emph{opposite} of $W_i$ or $W_i\oplus W_{\sigma'(\invast[i])}$.
We have $\varphi_{V_i}\colon V_i\to V_{\sigma'(\invast[i])}^*$ given by the form. 
It satisfies $\lsp{t}\varphi_{V_i} = - \varepsilon_i \varphi_{V_{\sigma'(\invast[i])}}$.

Given a collection of linear maps $B$, $a$, $b$ in $\bM$, we define
a new collection $\sigma(B)$, $\sigma(a)$, $\sigma(b)$ in $\bM$ by
\begin{equation}\label{eq:h:52}
\begin{gathered}
  \sigma(B)_h = \begin{cases}
  B^*_{\sigma'(\invast[\overline{h}])} & \text{if $h$ in the subdiagram in \cref{subsec:orientation}} \\
  -\varepsilon(h) B^*_{\sigma'(\invast[\overline{h}])} & \text{otherwise},
  \end{cases}
  \\
  \sigma(a)_i = - b^*_{\sigma'(\invast[i])}, \quad
  \sigma(b)_i = a^*_{\sigma'(\invast[i])}.
\end{gathered}
\end{equation}
This is related to the above definition by, for example,
\begin{equation*}
  a^*_{\sigma'(\invast[i])} = 
  \varphi_i^{-1} \circ \lsp{t}a_{\sigma'(\invast[i])} \circ \varphi_{V_i}.
\end{equation*}
This formula is a consequence of the definition of the adjoint. 
The choice of $\varphi_{V_i}$ is irrelevant when we define $\sigma$ on $\fM$ as it is a quotient space.
We have the induced involution $\sigma$ on $\fM_0(\bv,\bw)$, and
the projective morphism $\pi\colon \fM_{(\zeta_\RR,0)}(\bv,\bw)\to \fM_0(\bv,\bw)$ intertwines $\sigma$.

Now let us show that $\sigma$ is an involution.
Recall that $(C^*)^* = C$ if $C$ is a linear map between vector spaces with
forms with the same type. On the other hand, $(C^*)^* = -C$ if forms are opposite types.
In the above formula, we use that $V_i$ and $W_i$ have the opposite form,
hence $\sigma^2(a) = a$, $\sigma^2(b) = b$.
As for $B_h$, we note $V_{\vout{h}}$ and $V_{\vin{h}}$ have the same types
if $h$ is in the subdiagram, and the opposite type otherwise.
Together with $\sigma'(\invast[\overline{h}]) = \overline{h}$ when
$h$ is in the subdiagram, we see $\sigma^2(B)_h = B_h$.
\begin{NB}
  $\varepsilon(h)\varepsilon(\overline{h}) = -1$. 
\end{NB}%

Since $\pi$ is a birational morphism onto its image, we see that
$\sigma^2 = \id$ on $\fM_\zeta(\bv,\bw)$ as well.

This finishes the definition of the involution $\sigma$ on $\fM_\zeta(\bv,\bw)$. Let us denote
the fixed point subset by $\fM_\zeta^\sigma(\bv,\bw)$. It is smooth
and symplectic if $\zeta$ is generic (as in \cite[Prop.~2.6]{Na-reflect}),
as $\sigma$ preserves the symplectic form by the construction.
If $\zeta$ is not generic, we equip $\fM_\zeta^\sigma(\bv,\bw)$ with the fixed point subscheme
structure, i.e., the quotient by ideals generated by functions of forms $f - \sigma^*(f)$.

Following \cite{MR3900699}, we call $\fM_\zeta^\sigma(\bv,\bw)$ a \emph{$\sigma$-quiver variety}.


By $\sigma = \sigma'\circ S_{w_0}\circ\invast \circ t$, we have
\begin{equation}\label{eq:h:51}
  \CU_i' = -\bq^\bc\,\CU_{\sigma'(i)}^\vee,\qquad
  \CW_i' = \CW_{\sigma'(\invast[i])}^\vee
\end{equation}
where $\vee$ in the right hand side means the dual (virtual) vector bundle. Compare with the formula
\eqref{eq:h:50}.

\begin{NB}
\begin{equation*}
    c_k(\CV_i^\vee) = (-1)^k c_k(\CV_i).
\end{equation*}
\end{NB}%

\subsection{The fixed point set as a symplectic
  reduction}\label{subsec:cat_quotient}

It is instructive to study the fixed point set $\bM^\sigma$ in $\bM$ before
going to the quotient space $\fM_0(\bv,\bw)$.

\subsubsection{Orthosymplectic quiver}\label{subsub:orthsymplectic-quiver}
The description of $\bM^\sigma$ is particularly simple when $\sigma'\circ\invast = \id$.
Then $V_i$, $W_i$ have $(-\varepsilon_i)$, $\varepsilon_i$-forms, and the signs are alternating.
Following physics conventions, we draw the corresponding diagram as follows:
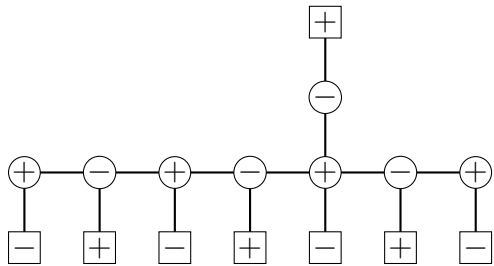
\begin{equation*}
  \begin{tikzpicture}[scale=.5,anchor=base, baseline]
    \node[SO,circle] (SO1) at (0,0) {};
    \node[Sp,circle] (Sp1) at (2,0) {};
    \node[SO,circle] (SO2) at (4,0) {};
    \node[Sp,circle] (Sp2) at (6,0) {};
    \node[SO,circle] (SO3) at (8,0) {};
    \node[Sp,circle] (Sp3) at (10,0) {};
    \node[SO,circle] (SO4) at (12,0) {};
    \node[Sp,circle] (Sp4) at (8,2) {};
    \draw[thick] (SO1) -- (Sp1) -- (SO2) -- (Sp2) -- (SO3) -- (Sp3) -- (SO4);
    \draw[thick] (SO3) -- (Sp4);
    \node[Sp,rectangle] (SO1f) at (0,-2) {};
    \node[SO,rectangle] (Sp1f) at (2,-2) {};
    \node[Sp,rectangle] (SO2f) at (4,-2) {};
    \node[SO,rectangle] (Sp2f) at (6,-2) {};
    \node[Sp,rectangle] (SO3f) at (8,-2) {};
    \node[SO,rectangle] (Sp3f) at (10,-2) {};
    \node[Sp,rectangle] (SO4f) at (12,-2) {};
    \node[SO,rectangle] (Sp4f) at (8,4) {};
    \draw[thick] (SO1) -- (SO1f);
    \draw[thick] (Sp1) -- (Sp1f);
    \draw[thick] (SO2) -- (SO2f);
    \draw[thick] (Sp2) -- (Sp2f);
    \draw[thick] (SO3) -- (SO3f);
    \draw[thick] (Sp3) -- (Sp3f);
    \draw[thick] (SO4) -- (SO4f);
    \draw[thick] (Sp4) -- (Sp4f);
  \end{tikzpicture}
\end{equation*}
Here the circled nodes are $V_i$'s, and the rectangular nodes are $W_i$'s.
The $\pm$ signs in circles and rectangles are the types of the forms on $V_i$ and $W_i$.
For each edge in this quiver, we assign $B_h$, choosing $h$ in the orientation.
We usually regard it as a vector in $V_{\vin{h}}\otimes V_{\vout{h}}$
via $\varphi_{V_{\vout{h}}}^{-1}\colon V_{\vout{h}}^* \cong V_{\vout{h}}$ given by the form on $V_{\vout{h}}$.
The opposite component $B_{\overline{h}}$ is determined by $B_h$, as
we consider a point in the fixed point set $\bM^\sigma$.
The same rule applies for edges between circled and rectangular nodes.
The group assigned to the vertex $i$ is $G_{-\varepsilon_i}(V_i)$, where
$\varepsilon_i$ is the type of the form on $V_i$.
The symplectic space $\bM^\sigma$ is the direct sum of tensor
products $V_{\vin{h}}\otimes V_{\vout{h}}$ for edges $h$,
and $V_i\otimes W_i$.
It is a symplectic representation of $G_\bv^\sigma \defeq \prod G_{-\varepsilon_i}(V_i)$.
This subgroup $G_\bv^\sigma$ is the fixed point subgroup with respect to an involution on $G_\bv$ 
defined by $(g_i)\mapsto (\lsp{t}g_i^{-1})$ as in \cref{def:framing-involution}.
We have a natural residual action of $G_\bw^\sigma = \prod G_{\varepsilon_i}(W_i)$
on the framing vector spaces.

The above diagram is an \emph{orthosymplectic quiver}, appeared first in \cite{MR694606}
(in type A).
It was used in order to realize nilpotent orbits in the classical Lie algebras
there.
It was also the main example in \cite{MR3900699}.
The corresponding Coulomb branches will be studied in \cite{FHN}.

\subsubsection{Mixture of orthosymplectic and usual quivers}
In type $\mathrm{A}_{\mathrm{odd}}$ with $\sigma' = \id$,
we only need to record the left half of the original quiver,
as the right half is determined by the left. This leads to the following diagram:
\begin{equation}\label{eq:h:23}
  \begin{tikzpicture}[scale=.5,anchor=base, baseline]
    \node[GL,circle] (SO1) at (0,0) {};
    \node[GL,circle] (Sp1) at (2,0) {};
    \node[GL,circle] (SO2) at (6,0) {};
    \node[Sp,circle] (SO3) at (8,0) {};
    \draw[thick] (SO1) -- (Sp1);
    \draw[thick,dotted] (Sp1) -- (SO2);
    \draw[thick] (SO2) -- (SO3);
    \node[GL,rectangle] (SO1f) at (0,-2) {};
    \node[GL,rectangle] (Sp1f) at (2,-2) {};
    \node[GL,rectangle] (SO2f) at (6,-2) {};
    \node[SO,rectangle] (SO3f) at (8,-2) {};
    \draw[thick] (SO1) -- (SO1f);
    \draw[thick] (Sp1) -- (Sp1f);
    \draw[thick] (SO2) -- (SO2f);
    \draw[thick] (SO3) -- (SO3f);
  \end{tikzpicture}
\end{equation}
Here nodes without $\pm$ signs are just vector spaces, and
we assign $\Hom(V_{\vout{h}}, V_{\vin{h}})\oplus
\Hom(V_{\vin{h}}, V_{\vout{h}})$ as usual. The same
applies to the rightmost edge between $\begin{tikzpicture}
  \node[GL,circle] at (0,0) {};
\end{tikzpicture}$ and $\begin{tikzpicture}
  \node[Sp,circle] at (0,0) {};
\end{tikzpicture}$.
We assign $V_i\otimes W_i$ to the rightmost vertical
edge between  $\begin{tikzpicture}
  \node[Sp,circle] at (0,0) {};
\end{tikzpicture}$ and $\begin{tikzpicture}
  \node[SO,rectangle] at (0,0) {};
\end{tikzpicture}$.
The group assigned to the vertex is $\GL(V_i)$ except
the rightmost one is $\grpSp(V_{i_1})$. Hence $\bM^\sigma$ is a symplectic representation
of $G_\bv^\sigma = \prod_{i\neq i_1} \GL(V_i)\times \grpSp(V_{i_1})$.
We have the residual action of $G_\bw^\sigma = \prod_{i\neq i_1} \GL(W_i)\times \grpO(W_{i_1})$.
This example is the case when the type is $(+)$. If
the type is $(-)$, we exchange $+$ and $-$ in the quiver.
In type D, $\mathrm{E}_6$ with $\sigma'\circ\invast\neq\id$, we
also have a similar combination of the usual and orthosymplectic quivers.

\subsubsection{Second exterior/symmetric power at the end of the quiver}

Finally, we consider type $\mathrm{A}_{\mathrm{even}}$ with $\sigma' = \id$.
Let $\ell$ be the number of vertices, and $\ell/2$ and $\ell/2+1$ be the middle vertices,
which are exchanged by $\invast$.
We only need to record the left half of the original quiver as above,
but we have an additional contribution from the middle edge.
The diagram is as follows:
\begin{equation}\label{eq:h:22}
  \begin{tikzpicture}[scale=.5,anchor=base, baseline]
    \node[GL,circle] (SO1) at (0,0) {};
    \node[GL,circle] (Sp1) at (2,0) {};
    \node[GL,circle] (SO2) at (6,0) {};
    \node[GL,circle] (SO3) at (8,0) {};
    \draw[thick] (SO1) -- (Sp1);
    \draw[thick,dotted] (Sp1) -- (SO2);
    \draw[thick] (SO2) -- (SO3);
    \draw[thick] (SO3) to[out=45,in=315,loop, "$\Wedge^2$"] (SO3) ;
    \node[GL,rectangle] (SO1f) at (0,-2) {};
    \node[GL,rectangle] (Sp1f) at (2,-2) {};
    \node[GL,rectangle] (SO2f) at (6,-2) {};
    \node[GL,rectangle] (SO3f) at (8,-2) {};
    \draw[thick] (SO1) -- (SO1f);
    \draw[thick] (Sp1) -- (Sp1f);
    \draw[thick] (SO2) -- (SO2f);
    \draw[thick] (SO3) -- (SO3f);
  \end{tikzpicture}
\end{equation}
Here $\Wedge^2$ means the second exterior power of the vector space $V_{\ell/2}$
at the rightmost vertex, which was $\ell/2$ in the original quiver. 
We add its cotangent space $T^*\Wedge^2 V_{\ell/2}
= \Wedge^2 V_{\ell/2}\oplus (\Wedge^2 V_{\ell/2})^*$ to $\bM$ for the usual quiver variety.
This summand comes from the endomorphisms between $V_{\ell/2}$ and $V_{\ell/2+1}$ in the original quiver,
under the identification $V_{\ell/2+1}= V_{\ell/2}^*$.
Note that $T^*\Wedge^2 V_{\ell/2}$ is a symplectic representation of $\GL(V_{\ell/2})$. 
This is the case when the type is $(+)$. For the type $(-)$, we replace $\Wedge^2$ by
the second symmetric power $\Sym^2$.
\begin{NB}
  This means that we add $\Wedge^2$ for orthogonal case, while $\Sym^2$ for symplectic case.
\end{NB}%
Groups are $G_\bv^\sigma = \prod \GL(V_i)$, $G_\bw^\sigma = \prod \GL(W_i)$,
where $i$ runs over the vertices in the above diagram.

Coulomb branches of gauge theories in (B,C) will be studied in \cite{Examples_of_S-dual}. See \cref{rem:coulombC,rem:coulombD}.
Just before submitting this manuscript to arXiv, we learn that there is a proposal \cite{2025arXiv251010652L}
to relate Coulomb branches of gauge theories associated with $(G^\sigma_\bv,\bM^\sigma)$
to fixed point subschemes in generalized affine Grassmannian slices with respect to involutions.

Let us return back to the general case.

\begin{Definition}\label{def:cat_quot}
  Let $\bM^\sigma\tslash G_\bv^\sigma$ be the categorical quotient
  of $\mu^{-1}(0)\cap\bM^\sigma$ by the action of $G_\bv^\sigma$.
\end{Definition}

Since the moment map for $\bM^\sigma$ is the restriction of $\mu$,
this quotient is regarded as a symplectic quotient of $\bM^\sigma$.

By a generalization of the LeBryun-Procesi theorem \cite{MR958897} in 
\cite{Lusztig-On-Quiver} (with framing) and \cite{MR3900699} (orthosymplectic),
we have a closed immersion
\begin{equation*}
  \iota\colon \bM^\sigma\tslash G_\bv^\sigma \to \fM_0^\sigma(\bv,\bw).
\end{equation*}
See \cite[Prop.~9.2.1]{MR3900699}.

By the argument in \cite[\S Fixed point set in the ordinary bow variety]{FHN},
the closed immersion is a set theoretical bijection, whose restriction to the
regular locus consisting of $0$-stable orbits in $\bM^\sigma\tslash G_\bv^\sigma$
is an isomorphism. Here the $0$-stability is defined with respect to the
action of the larger group $G_\bv$ (with parameter $\zeta_\RR = 0$), and we take its intersection with
$\bM^\sigma$.

\begin{Remark}\label{rem:framing-symmetry}
Note that the reflection functor $S_{w_0}$ has no effect on $W$. Therefore
we have a well-defined action of $G^\sigma_\bw$ on $\fM^\sigma_\zeta(\bv,\bw)$ for
arbitrary $\zeta$.  
\end{Remark}

\subsection{Unknown properties of cohomology/K-theory of \texorpdfstring{$\sigma$}{σ}-quiver varieties}

One important properties of the equivariant $K$-theory of quiver varieties
is freeness over the equivariant $K$-theory of a point, established
in \cite[\S7]{Na-qaff}. See also \cite[Th.~3.4]{Na-Tensor}. It depends on the decomposition of the diagonal
of the quiver variety, found earlier in \cite{Na-alg}. Indeed, by using the
action of a maximal torus $A_\bw$ in $G_\bw^\sigma$, we can reduce the problem to the case when
$A_\bw$ is trivial, i.e., $G_\bw^\sigma = \prod\grpO(1)$. (Here we use \cite[\S7]{Na-qaff} to deal with
to deal with quiver varieties in factors of torus fixed points.)
Even in this special case, it is not clear that the decomposition of the diagonal
is available. Probably, it is not true.
Since we work on the cohomology groups, rather than $K$-theory, we can use Kaledin's theorem \cite{MR2283801}
to conclude that odd degree cohomology of $\fM_\zeta^\sigma(\bv,\bw)$ vanishes. Therefore,
this difficulty disappears.

Another important property is the so-called Kirwan surjectivity, which says
$H^*_{G}(\mathrm{pt})\to H^*(\fM_\zeta(\bv,\bw))$ is surjective. For quiver varieties,
this was proved by McGerty and Nevins \cite{MR3773791}. However, it is not clear
whether this is still true for $\fM^\sigma_\zeta(\bv,\bw)$. 

\section{Stable envelopes for quiver varieties and \texorpdfstring{$\sigma$}{σ}-quiver varieties}
\label{sec:stable-envelope}

We assume $\zeta_\CC = 0$, $\sqrt{-1}\zeta_\RR > 0$ in this section. We drop $\zeta$ from the notation
of quiver varieties and $\sigma$-quiver varieties, as $\fM(\bv,\bw)$ and $\fM^\sigma(\bv,\bw)$.
On the other hand, we keep $\fM_0(\bv,\bw)$ and $\fM_0^\sigma(\bv,\bw)$
for $\zeta = 0$. We have projective morphisms $\pi\colon \fM(\bv,\bw)\to \fM_0(\bv,\bw)$ and
its restriction $\pi^\sigma\colon \fM^\sigma(\bv,\bw)\to \fM_0^\sigma(\bv,\bw)$.

It is also useful to introduce the notation $\fM(\bw) = \bigsqcup_\bv \fM(\bv,\bw)$
and $\fM_0(\bw) = \bigcup_\bv \fM_0(\bv,\bw)$, and similarly for $\sigma$-quiver varieties.

We have $\CC^\times_\hbar$-action on $\fM(\bv,\bw)$, $\fM^\sigma(\bv,\bw)$ and also on their $\zeta=0$ versions by
\cref{rem:CCx-hbar}. The morphisms $\pi$ and $\pi^\sigma$ are equivariant with respect to this action.

The inverse image $\pi^{-1}(0)$ and $(\pi^\sigma)^{-1}(0)$ in $\fM(\bv,\bw)$ and
$\fM^\sigma(\bv,\bw)$ are denoted by $\fL(\bv,\bw)$ and $\fL^\sigma(\bv,\bw)$ respectively.
It is known that $\fL(\bv,\bw)$ is a lagrangian subvariety of $\fM(\bv,\bw)$.
But it was observed that $\fL^\sigma(\bv,\bw)$ may be smaller dimension by
\cite[Rem.~6.2.2(3)]{MR3900699}.

Suppose $X$ is either $\fM(\bv,\bw)$ or $\fM^\sigma(\bv,\bw)$.
Similarly $X_0$ is either $\fM_0(\bv,\bw)$ or $\fM_0^\sigma(\bv,\bw)$.
In this notation, we drop $\sigma$ from $\pi^\sigma$, hence
$\pi\colon X\to X_0$.
We take a torus in $G_\bw$ or $G_\bw^\sigma$, according to
$X = \fM(\bv,\bw)$ or $X = \fM^\sigma(\bv,\bw)$. Denote it by $A$.
The choice of $A$ will be explained below.
%


\subsection{Torus fixed points}\label{subsec:torus-fixed-points}

For $X = \fM(\bv,\bw)$, we choose
a basis of $W$, compatible with the decomposition $W = \bigoplus W_i$. Then we choose
a maximal torus $A_\bw$ of $G_\bw$ as the group of diagonal matrices.
The choice of the basis gives an isomorphism $A_\bw\cong(\CC^\times)^n$. Let $(u_1,\dots, u_n)$
be the coordinates of the Lie algebra $\mathfrak a_\bw$ of $A_\bw$ given by this isomorphism.
The $a$-th factor $\CC^\times$ of $A_\bw$ acts on the scaling with weight $1$  
on a $1$-dimensional subspace
$W^a \subset W$, corresponding to the $a$-th element in the basis.
We denote by $i_a$ the vertex corresponding to $W^a$, i.e., $W^a\subset W_{i_a}$.
Then the fixed point set $X^{A_\bw}$ is the disjoint union of
\begin{equation}\label{eq:h:32}
  \fM(\bv^1,\bw^1)\times \fM(\bv^2,\bw^2) \times \cdots \times \fM(\bv^n,\bw^n),
\end{equation}
where $\bw^a$ is the dimension vector of $W^a$, hence the corresponding weight is $\varpi_{i_a}$ by the definition of $i_a$.
We take the union over all possible choices of $\bv^a$ with constraint $\bv = \sum_{a=1}^n \bv^a$.
See \cite[Lemma~3.2]{Na-Tensor}.

Next consider the case $X = \fM^\sigma(\bv,\bw)$. We choose a base of $W$
compatible with the decomposition $W = \bigoplus W_i$ as above, and
also compatible with forms as follows: the $a$-th factor $\CC^\times$ of $A_\bw$
acts on a basis element in $W_{i_a}$ by weight $1$,
and on another basis element in $W_{\sigma'(\invast[i_a])}$ by weight $-1$.
The value of the form between these two basis elements is nonzero, and 
fixed by the $\CC^\times$-action.
Other basis elements are fixed by this $\CC^\times$. 
They are orthogonal to the above two basis elements with respect to the form,
hence $\CC^\times$ preserves the form.

We denote the $1$-dimensional
subspace spanned by the first basis element by $W^a$ as before. Its dimension vector
is denoted by $\bw^a$.
The second basis element gives another $1$-dimensional subspace,
whose dimension vector is $\sigma(\bw^a) = \sigma'(\invast[\bw^a])$.
We have a coordinate system $(u_1,\dots, u_n)$ of $\mathfrak a_\bw$ as above.

Looking at the definition of $G_\bw^\sigma$ in \cref{def:framing-involution}, we see
an additional subspace $W^0$ coming from vertices such that $G_{\ve_i}(W_i) = \grpO(\mathrm{odd})$.
Each such a vertex $i$ contributes a $1$-dimensional subspace $W^0_i$ to $W^0$.
If there are no such vertices, we set $W^0 = 0$.
It is the weight $0$ subspace of $W$ with respect to $A_\bw$-action. Let $\bw^0$ be its dimension vector.

Note also that $A_\bw$ is a maximal torus in the connected component $(G_\bw^\sigma)^0$ of $G_\bw^\sigma$.

The fixed point set $X^{A_\bw}$ is the disjoint union of
\begin{equation}\label{eq:h:33}
  \fM(\bv^1,\bw^1)\times \fM(\bv^2,\bw^2)\times\cdots\times \fM(\bv^n,\bw^n)
  \times\fM^\sigma(\bv^0,\bw^0).
\end{equation}
Note that the last factor is a $\sigma$-quiver variety, while all preceding factors are ordinary quiver varieties.
It is because $\fM(\bv^a,\bw^a)$ appears through the embedding
\begin{equation*}
  \fM(\bv^a,\bw^a) \to \fM^\sigma(\bv^a + \sigma(\bv^a), \bw^a + \sigma(\bw^a)).
\end{equation*}
Here $\sigma(\bv^a)$, $\sigma(\bw^a)$ are defined as in \cref{def:involution}. The embedding
is given by $\fM(\bv^a,\bw^a)\ni x\mapsto x\oplus \sigma(x)$, where the direct sum
is given by direct sum of linear maps. See \cite[Prop.~5.2.1]{MR3900699} for details.\footnote{
The assumption $\bw^2 = a(\bw^2)$ in \cite{MR3900699}, which corresponds to $\bw^a = \sigma(\bw^a)$ in our setting, is unnecessary.}
When $\bw^0 = 0$, we understand $\fM^\sigma(\bv^0,\bw^0)$ is a point if $\bv^0 = 0$,
and $\emptyset$ otherwise.

Dimension vectors $\bv^a$ satisfy a constraint:
\begin{equation*}
    \bv^0 + \sum_{a=1}^n \left(\bv^a + \sigma(\bv^a)\right) = \bv.
\end{equation*}
Note also $\sigma(\bv^0) = \bv^0$, $\sigma(\bw^0) = \bw^0$, as $\fM^\sigma(\bv^0,\bw^0)$ does not
make sense otherwise.

Note that we place the factor $\fM^\sigma(\bv^0, \bw^0)$ at the end of the product. It is because
we want to make the twisted Yangian as a \emph{left} coideal subalgebra of the Yangian.
Equivariant cohomology of \eqref{eq:h:33} will be a module of the twisted Yangian,
which is a tensor product of the equivariant cohomology of factors.
The factor $\fM^\sigma(\bv^0,\bw^0)$ is a genuine $\sigma$-quiver variety without torus action, hence
not much is known about it.

\begin{Example}\label{ex:no-solution}
  Consider type A quiver with $\sigma' = \id$. For $\mathrm{A}_{\mathrm{even}}$, there is no
  vertex $i$ fixed by $\sigma'\circ\invast$. Therefore $W^0 = 0$, hence $\bv^0 = \bw^0 = 0$.
  For $\mathrm{A}_{\mathrm{odd}}$, we have a vertex $i_1$ in the middle of the quiver, fixed by $\sigma'\circ\invast$.
  We may have a $1$-dimensional $W^0$ at the vertex $i_1$ a priori, but
  there is no $\bv^0$ for which $\fM^\sigma(\bv^0,\bw^0)$ can be defined. Indeed, 
  we must have $\varpi_{i_1} = \bC \bv^0$ from the equation $\sigma(\bv^0) = \bv^0$, but
  there is no integral solution $\bv^0$ to this equation.
\end{Example}

For the case of $\sigma' = \invast$, an example of $\bw^0$ is given by slices to \emph{distinguished}
nilpotent orbits. For example, we have the following.

\begin{NB}
\begin{Example}
  Consider the type $\mathrm{A}_8$ quiver with $\sigma' = \invast[]$. We take the
  type $(+)$. So we put orthogonal forms on $W_i$ for odd $i$.
  We choose $\bw = (1,0,1,0,1,0,0,0)$, $\bv = (2,3,4,4,4,3,2,1)$. The
  condition $\sigma(\bw) = \bw$ is automatically satisfied, and
  $\sigma(\bv) = \bv$ is satisfied by our choice of $\bv$.
  We also choose $\zeta$ so that $\sigma(\zeta) = \zeta$. Thus
  $G_\bw^\sigma = \grpO(1)\times\grpO(1)\times\grpO(1)$, hence $A_\bw = \{1\}$.
  By \cite[Th.~8.3.3]{MR3900699}, $\fM^\sigma(\bv,\bw)$ is the inverse image of the Slodowy
  slice to the nilpotent orbit for the partition $(5,3,1)$ in the
  cotangent bundle of the flag variety of $\SO(9)$.
\end{Example}
\end{NB}%

\begin{Example}\label{ex:distinguished}
  Consider the type $\mathrm{A}_9$ quiver with $\sigma' = \invast[]$. We take the
  type $(-)$. So we put orthogonal forms on $W_i$ for even $i$.
  We choose $\bw = (0,0,0,1,0,1,0,0,0)$, $\bv = (1,2,3,4,4,4,3,2,1)$,
  $\bv' = (1,2,3,3,4,5,3,2,1)$). The
  condition $\sigma(\bw) = \bw$ is automatically satisfied, and
  $\sigma(\bv) = \bv$, $\sigma(\bv') = \bv'$ are satisfied by our choice of $\bv$, $\bv'$.
  We also choose $\zeta$ so that $\sigma(\zeta) = \zeta$. Thus
  $G_\bw^\sigma = \grpO(1)\times\grpO(1)$, hence $A_\bw = \{1\}$.
  By \cite[Th.~8.3.3]{MR3900699}, $\fM^\sigma(\bv,\bw)$ is the inverse image of the Slodowy
  slice to the nilpotent orbit for the partition $(6,4)$ in the
  cotangent bundle of the flag variety of $\grpSp(10)$.
  Similarly, $\fM^\sigma(\bv',\bw)$ is the inverse image of the Slodowy
  slice to the same nilpotent orbit in the cotangent bundle of the $9$-step isotropic flag variety
  of type $(1^32\,0^21^3)$, where the numbers are difference of dimensions of consecutive subspaces in the flag.
  Note that entries of $(6,4)$ have multiplicity $1$. Therefore they are distinguished nilpotent orbits.

  Comparing $\bv'$ with $\bv$, we see $\bv_4$ decreases by $1$ while $\bv_6$ increases by $1$.
  This suggests there should be a natural correspondence in the product
  $\fM^\sigma(\bv,\bw)\times\fM^\sigma(\bv',\bw)$. Indeed, both are inverse images
  of the \emph{same} Slodowy slice, there is one from ambient spaces. In view of the construction below,
  we want to interpret this correspondence as a matrix entry of an $R$-matrix.
\end{Example}

\subsection{Attracting sets}

Let $\rho\colon\CC^\times\to A_\bw$ be a cocharacter. The \emph{torus roots} are the $A_\bw$-weights
$\{\alpha \}$ occurring in the normal bundle to $X^{A_\bw}$. See \cite[Def.~3.2.1]{MR3951025}.
For $X = \fM(\bv,\bw)$, they are contained in $\{ u_a - u_b \mid a\neq b\}$.
For $X = \fM^\sigma(\bv,\bw)$, they are contained in $\{ u_a - u_b, u_a \}$. 
Indeed, $u_a$ appears in the normal bundle to
\begin{equation*}
 \prod_{b\neq a} \fM(\bv^b,\bw^b) \times \fM^\sigma(\bv^0 + \bv^a + \sigma(\bv^a), \bw^0 + \bw^a + \sigma(\bw^a))
\end{equation*}
unless this is equal to $X^{A_\bw}$.

The root hyperplanes divide the Lie algebra $\mathfrak a_{\bw,\RR}$ of a maximal compact subgroup of $A_\bw$ into open chambers.
They are type A Weyl chambers for $X = \fM(\bv,\bw)$, and type B/C Weyl chambers for $X = \fM^\sigma(\bv,\bw)$.
For example $\mathfrak C = \{ u_1 > u_2 > \cdots > u_n > 0\}$ is an example for the latter case.
When $G^\sigma_\bw$ is an even orthogonal group, we should consider the Weyl chambers with respect to
the full $\grpO$, instead of $\SO$. Therefore, we get type B/C Weyl chambers.

For a cocharacter $\rho$, we replace $A_\bw$ by the Zariski closure
$\overline{\rho(\CC^\times)}$, and denote it by $A$. In other words,
$A$ is a subtorus of the maximal torus $A_\bw$ of $G_\bw$ or
$G_\bw^\sigma$, and $\rho$ is generic in $A$.

We define its attracting set by
\begin{equation*}
  X^{\rho\ge 0} \defeq \left\{ x\in X \,\middle|\, \lim_{t\to 0}\rho(t)x \text{   exists}\right\}.
\end{equation*}
It was introduced under the notation $\mathfrak Z$ in \cite{Na-Tensor}.
Similarly, we define the repelling set $X^{\rho\le 0}$ by the attracting set of $\rho^{-1}$.
%

It decomposes into the union of locally closed subsets 
\begin{equation*}
  X^{\rho\ge 0}(F) \defeq \left\{ x\in X \,\middle|\, \lim_{t\to 0}\rho(t)x \in F\right\},
\end{equation*}
according to the decomposition of $X^A$ into
connected components $X^A = \bigsqcup F$.
For $X = \fM(\bv,\bw)$ and maximal $A = A_\bw$, they are various choices of decompositions $\bv^a$ with
$\sum \bv^a = \bv$, as it is known that $\fM(\bv^a,\bw^a)$ is connected.
For $X = \fM^\sigma(\bv,\bw)$, we do not know $\fM^\sigma(\bv^0,\bw^0)$ is connected,
but we can still decompose $X^A$ and $X^{\rho\ge 0}$ according to the decomposition
$\bv = \bv^0 + \sum (\bv^a + \sigma(\bv^a))$ anyway.

We have an order $\preceq$ on the index set $\{ F\}$ of the decomposition so that
$\bigcup_{F'\preceq F} X^{\rho\ge 0}(F')$ is closed in $X$.

\subsection{Polarization}

Suppose $F$ is a fixed point component of $X^A$. 
Consider the equivariant Euler class $e(N_F)$ of normal bundle $N_F$ to $F$ in $X$. Let 
$e(N_F)|_{H^*_A(\mathrm{pt})}$ be its
$H^{2\codim F}_A(\mathrm{pt})$-part with respect to the decomposition
$H^{2\codim F}_A(F) = \bigoplus_i H^{2i}_A(\mathrm{pt})\otimes H^{2\codim F - 2i}(F)$.
Since $X$ has a symplectic form preserved by the $A$-action, weights appear as pairs
$\alpha$, $-\alpha$ in the normal bundle. Therefore,
\(
  (-1)^{\codim F/2} e(N_F)|_{H^*_A(\mathrm{pt})}
\)
is a perfect square. A choice $\ve$ of a square root of this element is called a \emph{polarization}:
\begin{equation*}
  \ve^2 = (-1)^{\codim F/2} e(N_F)|_{H^*_A(\mathrm{pt})}.
\end{equation*}

A choice of $\rho$ gives a decomposition $N_F = N_F^{\rho> 0}\oplus N_F^{\rho< 0}$ into 
positive and negative parts with respect to $\rho$. Since $N_F^{\rho < 0}$ is the dual
of $N_F^{\rho > 0}$ with respect to the symplectic form, we have $e(N_F) = (-1)^{\codim F/2} e(N_F^{\rho < 0})^2$.
We define the sign in $\pm e(N_F^{\rho < 0})$ by $\pm e(N_F^{\rho < 0})_{H^*_A(\mathrm{pt})} = \ve$.

For quiver varieties, there is a natural choice of the polarization given by the decomposition
$T\fM(\bv,\bw) = T^{1/2} \oplus (T^{1/2})^\vee$ in the equivariant $K$-theory, where
$T^{1/2}$ is defined by a choice of an orientation of the quiver:
\begin{equation*}
  T^{1/2} = \sum_{h\in\Omega} \Hom(\mathcal V_{\vout{h}}, \mathcal V_{\vin{h}})
  - \sum_{i\in I} \End(\mathcal V_i)
  + \sum_{i\in I} \Hom(\mathcal W_i,\mathcal V_i).
\end{equation*}
See \cite[\S2.2.7]{MR3951025}.
For $\sigma$-quiver varieties, it is not clear how to define $T^{1/2}$. For example,
$\bM^\sigma$ in \cref{subsec:cat_quotient}\ref{subsub:orthsymplectic-quiver} is not a cotangent type
in general.
We will choose an ad~hoc polarization for $\fM^\sigma$ in Examples in \cref{sec:ex1,sec:ex2}. 

\subsection{Stable envelopes}

Fix $\rho$ and $\ve$ as above.
Let $\TT$ be the product $\CC^\times_\hbar\times A$.

The $A$-degree of a class in $H^*_\TT(F)$ is the degree with respect
to $H^*_A(\mathrm{pt})=\CC[\mathfrak a]$ in the isomorphism
$H^*_\TT(F) = H^*_{\CC^\times_\hbar}(F)\otimes H^*_A(\mathrm{pt})$.

\begin{Theorem}[\protect{\cite[Th.~3.3.4]{MR3951025}}]\label{thm:stable-envelope}
  There exists a unique map of $H^*_\TT(\mathrm{pt})$-modules
  \begin{equation*}
    \Stab_{\mathfrak C,\ve}\colon H^*_\TT(X^A)\to H^*_\TT(X)
  \end{equation*}
  such that for any connected component $F$ of $X^A$ and any class
  $\gamma\in H^*_{\TT/A}(F)$, it satisfies
  \begin{enumerate}[label=\textup{(\roman*)}]
  \item $\Stab_{\mathfrak C,\ve}(\gamma)$ is contained in the kernel
    of the restriction homomorphism
    $H^*_\TT(X)\to H^*_\TT(X\setminus \bigcup_{F'\preceq F} X^{\rho\ge
      0}(F'))$.
  \item The restriction of $\Stab_{\mathfrak C,\ve}(\gamma)$ to $F$ is equal to
    $\pm e(N_F^{\rho < 0})\cup \gamma$, where the sign is determined by the polarization.
  \item The $A$-degree of the restriction of $\Stab_{\mathfrak C,\ve}(\gamma)$ to $F'$ is less than
    $\frac12 \codim F'$ for any $F'\prec F$.
  \end{enumerate}
\end{Theorem}

Main examples in \cite{MR3951025} were quiver varieties, but the proof
was given for symplectic resolution, hence is applicable for
$\sigma$-quiver varieties, as noted in \cite{MR3900699}. Moreover,
$\Stab_{\mathfrak C,\ve}(\gamma)|_{F'}$ for $F'\prec F$ is divisible
by $\hbar$.

The stable envelope $\Stab_{\mathfrak C,\ve}$ is induced from a lagrangian
cycle in $X\times X^A$ supported on
\begin{equation}\label{eq:h:45}
  (\pi\times\pi)^{-1}
  \left\{ (x,y) \in X_0\times X^A_0 \,\middle|\,
  \lim_{t\to 0} \rho(t)x = y \right\},
\end{equation}
where the second $\pi$ in $\pi\times\pi$ is the restriction of $\pi$ to $X^A$.
See \cite[\S3.3]{tensor2}, where the above subvariety is denoted by $Z_{\mathfrak T}$.

\subsection{Induced polarization}\label{subsec:induced-polarization}

One of fundamental properties of stable envelopes is a compatibility with respect to
the restriction to the subtorus $A'\subset A$. This is guaranteed by the uniqueness
property of stable envelopes. See \cite[\S3.6]{MR3951025}.

Let $\mathfrak C$ be a chamber and $\mathfrak C'\subset\mathfrak C$ be a face of some codimension.
Let $A'$ be the subtorus of $A$ spanned by $\mathfrak C'$, and $\mathfrak a'$ be its Lie algebra.
We consider $X^A \subset X^{A'}\subset X$. Consider a polarization $\ve$ of the normal bundle $N_F$
of a component $F$ of $X^A$ in $X$. We factor it as $\ve' \ve''$ into weights that are $0$ and
nonzero on $\mathfrak a'$ respectively. Then $\ve'$ is a polarization of the normal bundle
$N'_F$ of $F$ in $X^{A'}$.

A connected component $F'$ of $X^{A'}$ contains $F_1$, $F_2$, \dots as connected components
of $(X^{A'})^{A/A'} = X^A$. We suppose $F_1$, $F_2$, \dots have polarization $\ve_1$, $\ve_2$, \dots
and decomposes them as $\ve_1'\ve_1''$, $\ve_2'\ve_2''$, \dots as above. We say
$F'$ has the \emph{induced polarization} $\ve''$ if $\ve_1''$, $\ve_2''$, \dots are equal to $\ve''$, when
they are restricted to $\mathfrak a'$.

When all components of $X^{A'}$ have the induced polarization, the stable envelope factorizes as
\begin{equation*}
  \Stab_{\mathfrak C,\ve} = \Stab_{\mathfrak C',\ve''} \circ \Stab_{\mathfrak C/\mathfrak C',\ve'}.
\end{equation*}
See \cite[Lemma~3.6.1]{MR3951025}.

Although it was not emphasized in \cite[\S3.6]{MR3951025}, polarizations of $F_1$, $F_2$, $\cdots$ must
be chosen consistent so that $\ve''$ is well-defined. This may not be possible in general.
See \cref{subsec:K-matrix}.

\subsection{Geometric construction of the \texorpdfstring{$R$}{R}-matrix}
\label{subsec:geom_R}

By the localization theorem in equivariant cohomology, $\Stab_{\mathfrak C,\ve}$ becomes
an isomorphism if we take tensor products with the fractional field
$\operatorname{Frac}H^*_\TT(\mathrm{pt})$ of $H^*_\TT(\mathrm{pt})$.

For two chambers $\mathfrak C$, $\mathfrak C'$ in $\mathfrak a_\RR$, we define
\begin{equation*}
  R_{\mathfrak C',\mathfrak C} = \Stab_{\mathfrak C',\ve}^{-1}\circ \Stab_{\mathfrak C,\ve}
  \in \End(H^*_\TT(X^A))\otimes \operatorname{Frac} H^*_\TT(\mathrm{pt}).
\end{equation*}

Consider the case $X = \fM(\bv,\bw)$ and $A$ is a three dimensional
torus with coordinates $(u_1,u_2,u_3)$. We may divide the diagonal
$\CC^\times$ as it acts trivially on $X$. We have the chamber
structure divided by root hyperplanes $u_a - u_b = 0$ ($a\neq b$). Let
us denote the $R$-matrix from the chamber $\{ u_a > u_b\}$ to
$\{ u_a < u_b\}$ by $R_{ab}(u_a-u_b)$. It is the $R$-matrix defined by
the stable envelope for the subtorus generated by $u_a$, $u_b$,
if $X^{\CC^\times_{u_a=u_b}}$ has the induced polarization in the sense of \cref{subsec:induced-polarization}.

We have two ways to move from the chamber $\{ u_1 > u_2 > u_3 \}$ to
the opposite chamber $\{ u_1 < u_2 < u_3\}$ by crossing root
hyperplanes. Composites of $R$-matrices are the same for either choice
by the definition of the $R$-matrix. It leads to the \emph{Yang-Baxter
  equation}:
\begin{equation}\label{eq:h:35}
  R_{12}(u_1 - u_2) R_{13}(u_1 - u_3) R_{23}(u_2 - u_3)
  = R_{23}(u_2 - u_3) R_{13}(u_1 - u_3) R_{12}(u_1 - u_2).
\end{equation}
This is the fundamental observation in \cite[\S4.1.9]{MR3951025}.

\subsection{Maulik-Okounkov Yangian and extended Yangian}\label{subsec:MOyangian}

We assume the following hereafter.

\begin{Assumption}\label{assum:induced-polarization-1}
  We have a polarization of components of $\fM(\bw)^{A_\bw}$ for all $\bw$ so that
  all components of $\fM(\bw)^{A'}$ for any subtorus $A'$ from a face of a chamber
  have the induced polarization.
\end{Assumption}

In particular, the $R$-matrix $R_{F_i,F_j}(u_1-u_2)$ is defined for $H^*_{\TT_{\varpi_i}}(\fM(\varpi_i))\otimes
H^*_{\TT_{\varpi_j}}(\fM(\varpi_j))$ and satisfies the Yang-Baxter equation.

The $R$-matrices, or rather $R$-linear operators in view of \cref{rem:explicitR},
defines the extend Yangian $\sX$, as its definition only requires the $R$-linear operators
satisfying the Yang-Baxter equation \eqref{eq:h:35}.
This is the Maulik-Okounkov Yangian $\sX_{\mathrm{MO}}$ in \cite{MR3951025}.

In fact, Varagnolo \cite{Varagnolo} constructed representations of $\sY$
on the equivariant cohomology of quiver varieties 
in the same way as the author did earlier in \cite{Na-qaff}
for the equivariant $K$-theory, long before \cite{MR3951025}.
Therefore it is natural to expect that $\sX_{\mathrm{MO}}$ is $\sX$
when the quiver is of type ADE. 

\begin{Theorem}\label{thm:MO-Yangian = ext-Yangian}
  Assume the quiver is of type ADE.

\textup{(1)} Maulik-Okounkov Yangian $\sX_{\mathrm{MO}}$ is isomorphic to the extended Yangian $\sX$.
The isomorphism is given by the identification of the $R$-matrix of $\sY$ with the
geometric $R$-matrix in \cref{subsec:geom_R} under the identification $F_i = H^*_{\CC^\times_\hbar}(\fM(\varpi^i))$ in \eqref{eq:h:48}, reviewed below.

\textup{(2)} The core Yangian $\mathbb Y$, a subalgebra of $\sX_{\mathrm{MO}}$ defined below, is equal to $\tilde\sY_R$ under the
isomorphism of \textup{(1)}.

\textup{(3)} The center of $\sX_{\mathrm{MO}}$ is generated by $\ch_k (\CW_i)$ \textup($i\in I$, $k\in\ZZ_{\ge 0}$\textup),
 where $\CW_i$ is the tautological bundle over $\fM(\bv,\bw)$ associated with the framing node $i$.
\end{Theorem}

Indeed, slightly weaker statements $\sY\subset\sX_{\mathrm{MO}}$
and $\fg_{\mathrm{MO}}^{\mathrm{ext}} = \fg^{\mathrm{ext}}$
were shown by
McBreen in his thesis \cite{McBreen_thesis}. 
Here $\fg_{\mathrm{MO}}^{\mathrm{ext}}$ is the extended Lie algebra constructed
from the geometric $R$-matrix as in \cref{subsec:ext_Lie}.
We will not use his argument, which is a calculation of some matrix coefficients of the geometric $R$-matrix.

Let us first recall the construction \cite{Varagnolo} briefly. It was given by
defining generators of $\sY$ as operators on $H^*_{\TT_\bw}(\fM(\bw))$, and checking
the defining relations of $\sY$ are satisfied. Here $\TT_\bw = \CC^\times_\hbar\times A_\bw$.

Among generators in the Lie algebra $\fg$, $h_i$ is just given by claiming
the direct summand $H^*_{\TT_\bw}(\fM(\bv,\bw))$ has weight $\mu$ in \eqref{eq:h:43}.
The generators $x_{i,0}^\pm$ is given by an correspondence in
$\fM(\bv,\bw)\times\fM(\bv+\alpha_i,\bw)$, introduced in \cite[(10.4)]{Na-quiver}.
See also \cite[\S5]{Na-alg}. Finally, $h_{i,r}$ is given by
equivariant Chern classes of the complex $\CU_i$ in \eqref{eq:h:44}.
More general $x_{i,r}^\pm$ with $r > 0$ are given by commutator of $x_{i,0}^\pm$ with $h_{i,r}$,
hence unnecessary to define separately, but they are given by powers of the first Chern class
of a natural line bundle over the correspondence.
See \cite[(4.2)]{Varagnolo} for the detail.
\begin{NB}
  $h_{i,r}$ is the coefficient of $\hbar z^{-r-1}$ in 
  \begin{equation*}
    \left(
      -1 + \sum \Delta_* \frac{\lambda_{-1/z} (\CU_i)}{\lambda_{-1/z} (\bq^2 \CU_i)}
    \right)^-,
  \end{equation*}
  where the sum is over all $\bv$. We have $\lambda_{-1/z}(\CU_i) = 1 - \frac1z c_1(\CU_i) + \cdots$,
  $\lambda_{-1/z}(\bq^2\CU_i) = 1 - \frac1z (c_1(\CU_i) + \hbar\operatorname{rank}(\CU_i))+ \cdots$.
  Therefore the above is 
  \begin{equation*}
    \frac{\hbar}z \operatorname{rank}(\CU_i) + \frac{\hbar c_1(\CU_i) + \frac12 \hbar^2 
    \operatorname{rank}(\operatorname{rank}+1)(\CU_i)}{z^2} + \cdots.
  \end{equation*}
\end{NB}%

\begin{Theorem}[\protect{\cite{Varagnolo}}]\label{thm:Varagnolo}
  Operators $x_{i,r}^\pm$, $h_{i,r}$ on $H^*_{\TT_\bw}(\fM(\bw))$ satisfy the defining relations of
  $\sY$.
\end{Theorem}

As we have mentioned \cref{subsec:Yangian}, the $\ell$-highest weight vector is given by the
fundamental class of $\fM(0,\bw) = \{ \mathrm{pt}\}$.

Structures of representations given by equivariant cohomology are parallel to those given by equivariant $K$-theory
studied in \cite{Na-qaff}. 
The isomorphism \eqref{eq:h:48} between $F_i$ and $H^*_{\CC^\times_\hbar}(\fM(\varpi^i))$ follows
as both are irreducible $\sY$-modules, and Drinfeld polynomials are the same.

Another relevant result to us is about the tensor product of representations.

Let us decompose $\bw$ as $\bw^1 + \bw^2$, and take the corresponding
cocharacter $\rho\colon\CC^\times\to A_\bw$ so that $\bw^1$ is weight $0$, and $\bw^2$ is weight $1$.
From the construction of $\sY$-module structure, it is clear that
$H_*^{\TT_{\bw}}(\fM^{\rho\ge 0}(\bw))$, Borel-More homology of $\fM^{\rho\ge 0}(\bw)$ is also a $\sY$-module
so that the pushforward homomorphism with respect to the inclusion
$\fM^{\rho\ge 0}(\bw)\hookrightarrow \fM(\bw)$ is a morphism of $\sY$-modules.
\begin{NB}
  Here $\TT = \CC^\times_\hbar\times \CC^\times$, so that $\rho$ is generic in
  $A = \CC^\times$. 
\end{NB}%
\begin{Theorem}[\protect{\cite[Th.~6.12]{Na-Tensor}}]\label{thm:tensor}
  There exists a unique isomorphism of $\sY$-modules
  \begin{equation*}
    H^*_{\TT_{\bw^1}}(\fM(\bw^1))\otimes H^*_{\TT_{\bw^2}}(\fM(\bw^2))
    \cong H_*^{\TT_{\bw}}(\fM^{\rho\ge 0}(\bw)),
  \end{equation*}
  sending the tensor product of the $\ell$-highest weight vectors to the $\ell$-highest weight vector.
\end{Theorem}

Here the $\sY$-module structure on the left hand side is given by the coproduct $\Delta$.
As remarked in \cite[Remark~6.14]{Na-Tensor}, the existence of the coproduct must be shown
independently from quiver varieties, and certain properties must be checked. A written proof
of the existence of the coproduct was given in \cite{2017arXiv170105288G}, but it was known
to experts before. The necessary properties were checked even earlier in \cite[\S2.8]{MR1103907}.

The uniqueness is an obvious consequence of the localization theorem in equivariant cohomology.
Namely, the isomorphism is the unique isomorphism of $\sY$-modules after tensoring with
the fractional field $\operatorname{Frac} H^*_{\TT_\bw}(\mathrm{pt})$.
See \cite[Lemma~6.4]{Na-Tensor}. Therefore the point of the proof is to check that
the image of the left hand side is the equivariant cohomology of $\fM^{\rho\ge 0}(\bw)$
as a \emph{non-localized} module.

The isomorphism with respect to the opposite coproduct is given by
$H^{\TT_\bw}_*(\fM^{\rho\le 0}(\bw))$ instead. Therefore the $R$-matrix is 
given via the isomorphism
\begin{equation*}
  H^{\TT_\bw}_*(\fM^{\rho\ge 0}(\bw))\otimes \operatorname{Frac} H^{\TT_\bw}_*(\mathrm{pt})
  \cong H^{\TT_\bw}_*(\fM^{\rho\le 0}(\bw))\otimes \operatorname{Frac} H^{\TT_\bw}_*(\mathrm{pt}),
\end{equation*}
defined via the inclusions $\fM^{\rho\ge 0}(\bw)$, $\fM^{\rho\le 0}(\bw)\hookrightarrow \fM(\bw)$.

Since the statement of \cite[Th.~6.12]{Na-Tensor} is about equivariant $K$-theory of
lagrangian subvarieties, 
the isomorphism in \cref{thm:tensor} is replaced by
\begin{equation}\label{eq:h:46}
  \Psi\colon 
  H_*^{\TT_{\bw^1}}(\fL(\bw^1))\otimes H_*^{\TT_{\bw^2}}(\fL(\bw^2)) 
  \xrightarrow{\cong} H_*^{\TT_{\bw}}(\widetilde{\fM}^{\rho\ge 0}(\bw)),
\end{equation}
where $\widetilde{\fM}^{\rho\ge 0}(\bw)$ is the subvariety $\{ x \in \fM^{\rho\ge 0}(\bw) \mid
\lim_{t\to 0} \rho(t)x \in \fL(\bw^1)\times\fL(\bw^2)\}$.
We use the perfect pairing between corresponding modules (\cite[Th.~7.3.5]{Na-qaff}, \cite[Th.~3.13]{Na-Tensor}) to get the above statement.

Next we prove the following:
\begin{Lemma}\label{lem:alg_stable-envelope}
The isomorphism $\Psi$ in \eqref{eq:h:46} is given by the stable envelope in 
\cref{thm:stable-envelope}.
\end{Lemma}

\begin{proof}
The property (i) of \cref{thm:stable-envelope} is guaranteed as the image is contained
in \linebreak[4] $H_*^{\TT_{\bw}}(\widetilde{\fM}^{\rho\ge 0}(\bw))$. 

The property (iii) is a consequence of the fact
that $\Delta$ is cocommutative at $\hbar = 0$. The $R$-matrix is the identity at $\hbar=0$,
and the stable envelope is a block Gauss decomposition of the $R$-matrix. See \cite[\S4.9]{MR3951025}.

For the property (ii), we need to recall a part of the proof of \cref{thm:tensor}.
%

Consider a minimum component $F = F_{\min}$ with respect to the order $\preceq$,
which is $\fM(0,\bw^1)\times\fM(\bv,\bw^2)$.
%
For $F_{\min}$,
the isomorphism is given by the Thom isomorphism of 
$X^{\rho\ge 0}(F) = \fM^{\rho\ge 0}(\fM(0,\bw^1)\times\fM(\bv,\bw^2))$,
denoted by ${\mathfrak Z}_1$ in \cite[(5.4)]{Na-Tensor}. 
%
It is the total space of the vector bundle $N_{F_{\min}}^{\rho > 0}$.
Therefore the restriction to $F_{\min}$ is given by the multiplication
by the equivariant Euler class $e(N_{F_{\min}}^{\rho < 0})$ of
the normal bundle of $X^{\rho\ge 0}(F_{\min})$ in $X$.
\begin{NB}
  $\mathfrak Z_1$ is $\fM^{\rho\ge 0}$ direction. Hence the normal
  bundle is $\fM^{\rho < 0}$ direction.
\end{NB}%
This implies the assertion for $H^{\TT_{\bw^1}}_*(\fL(0,\bw^1))\otimes H^{\TT_{\bw^2}}_*(\fL(\bv,\bw^2))$.

For a cycle $\gamma$ supported in a lagrangian of other components $F$, we use the proof of \cite[Prop.~6.3]{Na-Tensor}, which
turns out to be the argument in \cite[12.3.2]{Na-qaff}. 
Let $\fM_{k;n}(\bv,\bw)$ be the subset of $\fM(\bv,\bw)$ defined in
\cite[(2.11)]{Na-Tensor} or \cite[(5.4.1)]{Na-qaff}. It is a closed subset
in an open subset $\fM_{k;\le n}(\bv,\bw) = \bigcup_{n'\le n} \fM_{k;n'}(\bv,\bw)$.
Irreducible lagrangians in $F'$ intersects with $\fM_{k;n}(\bv,\bw)$ for some $k\in I$ and $n > 0$
in their open dense subsets unless $\bv=0$, and the complements are contained in $\fM_{k;> n}(\bv,\bw) = \bigcup_{n' > n} \fM_{k;n'}(\bv,\bw)$.
By the descending induction on $n$, we may assume that the property (ii) holds for
cycles $\gamma'$ whose images under $\Psi$ are supported on $\fM_{k;> n}(\bv,\bw)$, as $\fM_{k; > n}(\bv,\bw) = \emptyset$ for $n \ge \dim V_k$.

We have the Grassmann bundle structure $p\colon \fM_{k;n}(\bv,\bw) \to \fM_{k;0}(\bv-n\alpha_k,\bw)$
(\cite[Prop.~2.13]{Na-Tensor}, \cite[Prop.~5.4.3]{Na-qaff} or \cite[Prop.~4.5]{Na-alg}).
By construction, $\fM^{\rho\ge 0}$ and $\widetilde{\fM}^{\rho \ge 0}$ are
compatible with $p$, i.e., the inverse images of $p$ of their intersection with $\fM_{k;0}(\bv-n\alpha_k,\bw)$ are equal to the intersection with $\fM_{k;n}(\bv,\bw)$.
Also the Grassmann bundle is the restriction of the correspondence used in the definition
of $x^-_{k,0}$ in \cite{Varagnolo}. 
(More precisely, we use the version which corresponds to the $n$-th divided power of $x_{k,0}^-$.)
Since \eqref{eq:h:46} is constructed as a $\sY$-homomorphism, the image $\Psi(\gamma)$
\begin{NB}
  restricted to the open subset $\fM_{k;\le n}(\bv,\bw)$, more precisely,
\end{NB}%
is the pull-back of $\Psi(\gamma')$ in $\fM_{k;0}(\bv-n\alpha_k,\bw)$,
capped with Chern classes of tautological bundles of the Grassmann bundle.

By the induction on $\bv$, we may assume that the property (ii) holds for $\gamma'$,
as the assertion is trivial for $\bv = 0$.
Let $F'$ be the component containing $\gamma'$.
The normal bundle $N^{\rho < 0}(F)$ of $\fM^{\rho\ge 0}(F)$ in $\fM(\bv,\bw)$ contains
the pull-back of $N^{\rho < 0}(F')$ by $p$ as a subbundle by the compatibility
of $\fM^{\rho\ge 0}$ with $p$, and the quotient 
$N^{\rho < 0}(F)/p^* N^{\rho < 0}(F')$ is 
the normal bundle of the Grassmann bundle $\fM_{k;n}(\bv,\bw)$ in $\fM_{k;\le n}(\bv,\bw)$.
Therefore the property (ii) for $\Psi(\gamma')$ implies that for $\Psi(\gamma)$.

Here we did not care the polarization $\ve$ in the proof, but it is contained in the
definition of generators $x_{i,r}^\pm$ in \cite{Varagnolo}.
\end{proof}

Thus the $R$-matrix defined by the stable envelope coincides with the
$R$-matrix defined as the $\sY$-module. Therefore, we conclude That
\begin{equation}
  \sX_{\mathrm{MO}} \cong \sX.
\end{equation}

Recall that the extended Lie algebra $\fg^{\mathrm{ext}}$ was introduced
in \cref{subsec:ext_Lie}, from the expansion of the $R$-matrix. This is
the same definition of $\fg_Q$ in \cite[\S5.3]{MR3951025}. Therefore,
$\fg_Q$ is nothing but $\fg^{\mathrm{ext}}$.

In \cite[\S6.1]{MR3951025}, a certain Lie subalgebra $\fg_Q'$ of $\fg_Q$ was
introduced. For finite type quivers, it is the subalgebra generated by
matrix elements of the normalized classical $r$-matrix
\begin{equation*}
  \sr - \sum_{i,j} \bC^{ij} \bw_i \otimes \bw_j,
\end{equation*}
where $\bC^{ij}$ is the inverse of the Cartan matrix.
(Note that The Cartan matrix is invertible for finite type quivers.)
The diagonal part of the classical $r$-matrix is
\begin{equation*}
\sum_i \bw_i\otimes \bv_i + \bv_i\otimes\bw_i - \sum_j \bC_{ij} \bv_j\otimes\bv_i.
\end{equation*}
See \cite[(4.21)]{MR3951025}. Therefore $\fg_Q'$ is nothing but the original Lie algebra $\fg$.

The \emph{core Yangian} $\mathbb Y$ in \cite[\S6.1]{MR3951025} is the subalgebra of $\sX_{\mathrm{MO}}$ generated by $\fg_Q'$ and
the operators of cup products by $\ch_k \widehat{\mathcal V}_i$ for $k \ge 0$ and $i\in I$.
(There is a typo in \cite[Th.~6.1.4]{MR3951025}: $\ch_0 \widehat{\mathcal V_i} = \bw_i - \sum_j \bC_{ij} \bv_j$
should be included.)
The normalized tautological bundle $\widehat{\mathcal V}_i$ is defined
by using the inverse of the $q$-analog of the Cartan matrix $\bC^{ij}_q$.
See \cite[(6.3)]{MR3951025}.
Comparing \eqref{eq:h:44} with \cite[(6.3)]{MR3951025}, we
find that $\widehat{\mathcal V}_i$ is
$\sum_j \bC^{ij}_q \bq \CU_j$ with $\CU_j$ defined in \eqref{eq:h:44}.
In other words, the core Yangian $\mathbb Y$, as the subalgebra
of $\prod_\scH \End(\scH)$, is generated by operators from
the original Yangian $\sY$.

By \cref{lem:tildeSYR}, operators on $F_i$ coming from the original Yangian $\sY$
are coming from the subalgebra $\tilde\sY_R$ of $\sX$.
The same is true for all $\scH$ in \eqref{eq:h:11}, as $\tilde\Phi$ is a Hopf algebra
homomorphism. Since this holds for all $\scH$ simultaneously, $\tilde\sY_R$ is the subalgebra
of $\prod_\scH \End(\scH)$ consisting of operators coming from the original Yangian $\sY$.
Therefore the core Yangian $\mathbb Y$ is the equal to $\tilde\sY_R$
as the subalgebra of $\sX \cong \sX_{\mathrm{MO}}$.

Finally, let us show that the center of $\sX = \sX_{\mathrm{MO}}$ is generated by
$\ch_k(\CW_i)$. It is clear that $\ch_k(\CW_i)$ are central.
\begin{NB}
  $\Delta \ch_k(\CW_i) = \ch_k(\CW_i)\otimes 1 + 1\otimes \ch_k(\CW_i)$, As
  $\CW_i$ decomposes according to the tensor product decomposition. Note
  $\bw = \bw^1 + \bw^2$ in \eqref{eq:h:46}. And $\ch_k(\CW_i)$ on $F_i[u]$ is
  $\ch_k(\CW_i\otimes \shfO(u)) = \ch_k(\CW_i) + \ch_{k-1}(\CW_i) u + \cdots$.
\end{NB}%
 In the associated graded, $\sX = \sX_{\mathrm{MO}}$ degenerates to $\bU(\fg^{\mathrm{ext}}[u])$ by \cref{thm:PBW}.
 By \cite[Lemma~1.7.4]{MR2355506}, the center of $\bU(\fg^{\mathrm{ext}}[u])$ is
 $\bU(\mathfrak z_I[u])$. The latter is generated by coefficients of $\delta_i[u] = \sum_k \ch_k(\CW_i) u^k$
in $u$. Thus the center of $\sX = \sX_{\mathrm{MO}}$ is generated by $\ch_k(\CW_i)$.

The proof of \cref{thm:MO-Yangian = ext-Yangian} is complete.

\subsection{Yangian representations on cohomology of quiver varieties}\label{subsec:Y_rep_coh}

By the identification $\sX_{\mathrm{MO}}\cong \sX$, we have a representation
of $\sX_{\mathrm{MO}}$ on $\scH = F_{i_1}[u_1]\otimes F_{i_2}[u_2]\otimes\dots\otimes F_{i_n}[u_n]$.
The space $\scH$ is identified, via \eqref{eq:h:48}, with $H^*_{\TT_{\varpi_{i_1}}}(\fM(\varpi_{i_1}))\otimes 
H^*_{\TT_{\varpi_{i_2}}}(\fM(\varpi_{i_2}))\otimes\dots\otimes 
  H^*_{\TT_{\varpi_{i_n}}}(\fM(\varpi_{i_n}))$.
By a straight forward generalization of \cref{thm:tensor}, this space is further
identified with $H_*^{\CC^\times_\hbar\times A_\bw}(\fM^{\rho\ge 0}(\bw))$, where $\bw = \sum_k \varpi_{i_k}$ and
$\rho\colon\CC^\times\to A_\bw$ is an appropriate cocharacter.
By the construction of \cref{thm:Varagnolo}, 
pushforward homomorphisms with respect to the inclusion
$\fL(\bw)\hookrightarrow\fM^{\rho\ge 0}(\bw)\hookrightarrow \fM(\bw)$
\begin{equation*}
  H_*^{\TT_\bw}(\fL(\bw)) \to
  H_*^{\TT_\bw}(\fM^{\rho\ge 0}(\bw)) \to
  H_*^{\TT_\bw}(\fM(\bw))
\end{equation*}
are homomorphisms of $\sX$-modules.

This assertion can be also deduced directly from the definition of $\sX_{\mathrm{MO}}$. It is a consequence
of the factorization property of the stable envelopes (see \cref{subsec:induced-polarization} and \cite[\S4.2.4]{MR3951025}).
The $R$-matrix $R_{F[u],H^*_{\TT_\bw}(\fM(\bw))}$ for $F = F_i = H^*_{\CC^\times_\hbar}(\fM(\varpi_i))$
and $H^*_{\TT_\bw}(\fM(\bw))$ is given by the product of $R$-matrices, exactly as in \eqref{eq:h:12}
under the \cref{assum:induced-polarization-1}.
As the expansion of the $R$-matrix at $u=\infty$, coefficients of the $R$-matrix
define operators on \emph{non-localized} equivariant homology groups.
This explains the assertion for $H_*^{\TT_\bw}(\fM(\bw))$.

Moreover, in \cite[\S4 and Remark~4]{tensor2} we view the stable envelope as
a class in $H^{\TT_\bw\times \CC^\times_u}_*(Z_{\mathcal T})$, where
$Z_{\mathcal T}$ is the subvariety in \eqref{eq:h:45} with
$A = \CC^\times_u$, 
$X = \fM(\bw + \varpi_i)$, $X^A = \fM(\bw)\times \fM(\varpi_i)$.
Therefore the $R_{F[u],H^*_{\TT_\bw}(\fM(\bw))} = 
\Stab_{-\mathfrak C,\ve}^{-1}\circ \Stab_{\mathfrak C,\ve}$ is given by
the convolution product class in
$H^{\TT_\bw\times \CC^\times_u}(Z_{\mathcal T}^\tau\circ Z_{\mathcal T})$,
where the superscript $\tau$ in $Z_{\mathcal T}^\tau$ refers to a permutation of factors $X$ and $X^A$.
See \cite[Th.~4.4.1]{MR3951025} for the assertion 
$\Stab_{-\mathfrak C,\ve}^{-1} = \Stab_{\mathfrak C,\ve}^\tau$.

The convolution product involves the pushforward homomorphism with respect to the projection $p_{13}\colon X^A\times X\times X^A \to X^A\times X^A$,
which is not proper, even after we restrict it to $p_{12}^{-1}(Z_{\mathcal T,21})\cap p_{23}^{-1}(Z_{\mathcal T})$.
Therefore $p_{13*}$ is defined by the localization. However, $p_{13}$ becomes
proper after we restrict it the $A$-fixed points. Therefore, coefficients of $R$,
expanded at $u=\infty$, are \emph{non-localized} classes in
$H^{\TT_\bw\times \CC^\times_u}(Z_{\mathcal T}^\tau\circ Z_{\mathcal T})$.
Since $Z_{\mathcal T}^\tau\circ Z_{\mathcal T}\subset X^A\times_{X_0^A} X^A$,
the usual Steinberg type variety, coefficients of $R$ preserves
$H_*^{\TT_\bw}(\fL(\bw))$.

The stable envelopes used above are $A\times G_\bw$-equivariant. Therefore,
we can replace $\TT_\bw$-equivariant homology groups by $\CC^\times_\hbar\times G_\bw$-equivariant
ones.

\subsection{Affine types}\label{subsec:affine}

\cite{2023arXiv230902377A} generalized \cref{fact:R-matrix} to the case when $\fg$ is an affine Lie algebra
except $A_1^{(1)}$ and $A_2^{(2)}$.
The coproduct on the Yangian $\sY$ was defined in \cite{2017arXiv170105288G}.

\cref{thm:tensor} and \cref{lem:alg_stable-envelope} remain true in this case with the same proof.
Therefore the Maulik-Okounkov $R$-matrix in this case is also given by
\cite{2023arXiv230902377A}, once it is normalized so that it acts as the identity on
the tensor product of the $\ell$-highest weight vectors.
Also, the Maulik-Okounkov Yangian $\sX_{\mathrm{MO}}$ is isomorphic to the extended Yangian $\sX$,
which is defined by a straightforward modification of \cref{subsec:ext_yangian}.
It is interesting to understand the core Yangian $\mathbb Y$ in the way similar to \cref{subsec:recoverY}.

\section{Twisted Yangian}\label{sec:twisted-yangian}

After reviewing the stable envelope and the derivation of Yang-Baxter equations,
let us switch to the case of $X = \fM^\sigma(\bv,\bw)$.

We keep the notation $\TT_\bw$ for $\CC^\times_\hbar\times A_\bw$ from the previous section.
For usual quiver varieties, $A_\bw$ is a maximal torus of $G_\bw = \prod_i \GL(W_i)$.
For $\sigma$-quiver varieties, $A_\bw$ is a maximal torus of $(G_\bw^\sigma)^0$,
the group introduced in \cref{def:framing-involution}.

\subsection{Dual representation}\label{subsec:dual_rep}

This subsection is a preparation for the subsequent subsections.

Following \cite[\S2.17]{MR1103907}, we define the left dual representation $\lsp{t}V$ of a representation $V$ by
\begin{equation*}
  (x\cdot f)(v) = f(S(y)\cdot v), \quad x\in\sX, f\in\lsp{t}V, v\in V.
\end{equation*}
We also define the left dual for a representation $V$ of $\sY$ in the same way.

\begin{Proposition}[\protect{\cite[Prop.~3.1, 3.2]{MR1103907}}]\label{prop:dual_rep}
  Let $V$ be an irreducible representation of $\sY$ with Drinfeld polynomial $P_{i,V}(u)$.
  Then the left dual representation $\lsp{t}V$ is also irreducible, and its Drinfeld polynomial $P_{i,\lsp{t}V}$ is
  given by
  \begin{equation*}
    P_{i,\lsp{t}V}(u) = P_{\invast[i],V}(u+\frac{\hbar}4 c_\fg).
  \end{equation*}
\end{Proposition}

\subsection{Reflection equation}\label{subsec:reflection-equation}

Suppose $A$ is two dimensional, and consider the decomposition \eqref{eq:h:33} with
$n=2$. We do not need to assume $\bw^a$ ($a = 1,2$) is $1$-dimensional
here. We suppose $W$ decomposes as
$W^0\oplus W^1\oplus W^{-1}\oplus W^2\oplus W^{-2}$
and $u_a$ acts with weight $1$ (resp.\
$-1$) on $W^a$ (resp.\ $W^{-a}$) and trivially on
other summands. Moreover, pairs $W^1$ and $W^{-1}$, $W^2$ and $W^{-2}$
are dual to each other with respect to the form on $W$, and
other pairs are orthogonal.
The fixed point set is the union of
\begin{equation}\label{eq:h:53}
  \fM(\bv^1,\bw^1)\times\fM(\bv^2,\bw^2)\times\fM^\sigma(\bv^0,\bw^0).
\end{equation}
Root hyperplanes are $u_1 - u_2 = 0$, $u_1 + u_2 = 0$ and $u_a = 0$
($a=1,2$) in this case.
The fixed point sets with respect to the subtori defined by these roots are
\begin{equation}\label{eq:h:54}
  \begin{gathered}
 \fM(\bv^1+\bv^2,\bw^1+\bw^2)\times \fM^\sigma(\bv^0,\bw^0),\quad
 \fM(\bv^1+\sigma(\bv^2),\bw^1+\sigma(\bw^2))\times \fM^\sigma(\bv^0,\bw^0),
 \\\text{and}\quad
 \fM(\bv^b,\bw^b)\times \fM^\sigma(\bv^0 + \bv^a + \sigma(\bv^a),\bw^0 + \bw^a + \sigma(\bw^a))
  \end{gathered}
\end{equation}
respectively, where $\{ a,b\} = \{1, 2\}$.

\begin{Assumption}\label{assum:induced-polarization}
We assume 
\begin{itemize}
\item \eqref{eq:h:54} have the induced polarization in the sense of \cref{subsec:induced-polarization}.
\item the polarizations of \eqref{eq:h:53} in \eqref{eq:h:54} are independent of connected components
  the common factor $\fM^\sigma(\bv^0,\bw^0)$ (in the first two cases), and
  $\fM(\bv^b,\bw^b)$ (in the last case). When $\bv^0$ or $\bv^b$ could change, they are also independent of
  $\bv^0$ or $\bv^b$.
\end{itemize}
\end{Assumption}

It is customary to call the $R$-matrix from the chamber $\{ u_a > 0\}$
to $\{ u_a < 0\}$ \emph{$K$-matrix} and denoted by $K(u_a)$.
It is the $R$-matrix defined by the stable envelope for the subtorus generated by $u_a$.
On the other hand, we denote other $R$-matrices by $R(u_1-u_2)$,
$R(u_1+u_2)$ according to the root hyperplane $u_1 - u_2 = 0$ and
$u_1 + u_2 = 0$. 
This abuse of notation is justified, as these $R$-matrices
are coming from ordinary quiver varieties
$\fM(\bv^1,\bw^1)\times\fM(\bv^2,\bw^2)$ with respect to polarizations
in $\fM(\bv^1+\bv^2,\bw^1+\bw^2)$ given by Assumption~\ref{assum:induced-polarization}.

The analog of the Yang-Baxter equation is the
\emph{reflection equation}:
\begin{equation}\label{eq:h:34}
  K(u_2) R(u_1+u_2) K(u_1) R(u_1-u_2) =
  R(u_1 - u_2) K(u_1) R(u_1 + u_2) K(u_2).
\end{equation}
This is due to Li in \cite[Prop.~5.5.3]{MR3900699}.
By Assumption~\ref{assum:induced-polarization}, $K(u_1)$ acts by the identity on
$H^*_{\TT_{\bw^2}}(\fM(\bv^2,\bw^2))$, and similarly for $K(u_2)$, $R(u_1-u_2)$ and $R(u_1+u_2)$ appropriately.

Let us draw the corresponding chambers:
\begin{equation}\label{eq:h:36}
  \begin{CD}
    \{ u_1 > u_2 > 0\}
    @>{K(u_2)}>> \{ u_1 > 0 > u_2, u_1 + u_2 > 0\} \\
    @V{R(u_1-u_2)}VV @VV{R(u_1+u_2)}V \\
    \{ u_2 > u_1 > 0\} @. \{ u_1 > 0 > u_2, u_1 + u_2 < 0\} \\
    @V{K(u_1)}VV @VV{K(u_1)}V \\
    \{ u_2 > 0 > u_1, u_1 + u_2 > 0\}
    @. \{ 0 > u_1 > u_2\} \\
    @V{R(u_1 + u_2)}VV
    @VV{R(u_1 - u_2)}V \\
    \{ u_2 > 0 > u_1, u_1 + u_2 < 0\}
    @>>{K(u_2)}> \{ 0 > u_2 > u_1 \}\rlap{.}
  \end{CD}
\end{equation}

The reflection equation \eqref{eq:h:34} is an equality in
\(
  \End\bigl(H^*_{\TT_{\bw^1}}(\fM(\bw^1))\otimes \allowbreak
  H^*_{\TT_{\bw^2}}(\fM(\bw^2))\otimes\allowbreak
    H^*_{\TT_{\bw^0}}(\fM^\sigma(\bw^0))\bigr)
  \otimes\operatorname{Frac}H^*_{\TT_{\bw}}(\mathrm{pt}).
\)
However, we often want to understand $K$ and $R$-matrices as
\emph{matrices}, rather than linear maps. See \cref{rem:explicitR}.
It is a nontrivial task as $H^*_\TT(\fM(\bw^1))$, \ldots have no
preferred base.
We will do this computation in two examples in \cref{sec:ex1,sec:ex2}.

As a preliminary step, let us consider on which vector spaces $K$ and
$R$-matrices operate more carefully.
Let us introduce notations:
\begin{equation*}
  F_1[u_1] \defeq H^*_{\TT_{\bw^1}}(\fM(\bw^1)), \quad
  F_2[u_2] \defeq H^*_{\TT_{\bw^2}}(\fM(\bw^2)), \quad
  \fH \defeq H^*_{\TT_{\bw^0}}(\fM^\sigma(\bw^0)).
\end{equation*}
By the discussion around \eqref{eq:h:33}, we have two more natural spaces
\begin{equation*}
  F_1^\sigma[-u_1] \defeq H^*_{\TT_{\bw^1}}(\fM(\sigma(\bw^1))),\quad
  F_2^\sigma[-u_2] \defeq H^*_{\TT_{\bw^2}}(\fM(\sigma(\bw^2))).
\end{equation*}
Although $F_a[u_a]$ and $F_a^\sigma[-u_a]$ are isomorphic by $\sigma$, they
naturally appear when we interpret $R(u_1+u_2)$ as the $R$-matrix
for \emph{usual} quiver varieties.

Indeed, $R(u_1+u_2)$ in the right hand side of \eqref{eq:h:34} is
naturally interpreted as
\begin{equation*}
  R_{F_1,F_2^\sigma}(u_1+u_2)\in \End\left(
    F_1[u_1]\otimes F_2^\sigma[-u_2]\right)\otimes
  \operatorname{Frac}H^*_{\TT_{\bw^1+\bw^2}}(\mathrm{pt}),
\end{equation*}
as the fixed point set with respect to the subtorus
$\{ u_1 + u_2 = 0\}$ is
$\fM(\bv^1 + \sigma(\bv^2), \bw^1 + \sigma(\bw^2))$. This
consideration suggests to interpret $K(u_2)$, the rightmost term in
the right hand side of \eqref{eq:h:34} as
\begin{equation}\label{eq:h:47}
  K_{F_2}(u_2) \in \Hom( F_2[u_2]\otimes\fH, F_2^\sigma[-u_2]\otimes\fH)
  \otimes\operatorname{Frac}H^*_{\TT_{\bw^2}}(\mathrm{pt}).
\end{equation}

A further consideration leads us the following diagram:
\begin{equation}\label{eq:h:37}
  \begin{CD}
    F_1[u_1]\otimes F_2[u_2]\otimes\fH
    @>{K_{F_2}(u_2)_{23}}>> F_1[u_1]\otimes F_2^\sigma[-u_2]\otimes\fH \\
    @V{R_{F_1, F_2}(u_1-u_2)_{12}}VV @VV{R_{F_1, F_2^\sigma}(u_1+u_2)_{12}}V \\
    F_1[u_1]\otimes F_2[u_2]\otimes\fH
    @. F_1[u_1]\otimes F_2^\sigma[-u_2]\otimes\fH \\
    @V{K_{F_1}(u_1)_{13}}VV @VV{K_{F_1}(u_1)_{13}}V \\
    F_1^\sigma[-u_1]\otimes F_2[u_2]\otimes\fH
    @. F_1^\sigma[-u_1]\otimes F_2^\sigma[-u_2]\otimes\fH \\
    @V{R_{F_2, F_1^\sigma}(u_1 + u_2)_{21}}VV
    @VV{R_{F_2^\sigma, F_1^\sigma}(u_1 - u_2)_{21}}V \\
    F_1^\sigma[-u_1]\otimes F_2[u_2]\otimes\fH
    @>>{K_{F_2}(u_2)_{23}}> 
    F_1^\sigma[-u_1]\otimes F_2^\sigma[-u_2]\otimes\fH\rlap{.}
  \end{CD}
\end{equation}
Here subscripts $12$, $13$, $23$ indicate tensor factors in which
operators act.
Let us explain subscripts $21$ in the last row. The $R$-matrix for
$F_1^\sigma[-u_1]\otimes F_2[u_2]$ goes from the chamber
$\{ -u_1 > u_2\}$ to $\{ -u_1 < u_2\}$ by convention for $R$ of usual
quiver varieties in \cref{subsec:geom_R}. 
It goes the opposite direction, compared with \eqref{eq:h:36}.
Also, the argument of the
$R$-matrix is $-u_1 - u_2$. This $R$-matrix is
$R_{F_1^\sigma,F_2}(-u_1-u_2)$, hence we replace it by
$R_{F_2, F_1^\sigma}(u_1 + u_2)_{21}$ by the unitarity \eqref{eq:h:25}.
This explains the entry in the left column. For the right column,
$R_{F_1^\sigma,F_2^\sigma}(-u_1+u_2)$ goes from $\{ -u_1 > -u_2\}$ to
$\{ -u_1 < -u_2\}$. It goes the opposite direction. 
We use the unitarity \eqref{eq:h:25} to change it to
$R_{F_2^\sigma,F_1^\sigma}(u_1-u_2)$.

Thus we arrive at the reflection equation, which says \eqref{eq:h:37}
is commutative:
\begin{equation}\label{eq:h:1}
  \begin{multlined}
  K_{F_2}(u_2) R_{F_2,F_1^\sigma}(u_1+u_2)_{21} K_{F_1}(u_1)
  R_{F_1,F_2}(u_1-u_2) \qquad\\
  \qquad = R_{F_2^\sigma,F_1^\sigma}(u_1-u_2)_{21} K_{F_1}(u_1) R_{F_1,F_2^\sigma}(u_1+u_2) K_{F_2}(u_2).
  \end{multlined}
\end{equation}
Here we omit most of subscripts. The only exception is $21$ which is essential.

This type of the reflection equation appeared in \cite{MR774205,
  MR953215}. See also \cite[\S6.2]{MR1193836} containing \emph{four}
types of $R$-matrices. See \cite{2022arXiv220316503A} for a modern
reference.\footnote{The convention of the spectral parameter in the $R$-matrix
  is different from ours. Their $R(z)$ (trigonometric case) should be
  compared with our $R(-u)$. The reflection equation in \cite[Theorem~5.1.2]{2022arXiv220316503A}
  is different from \eqref{eq:h:1} for spectral parameters in $R$-matrices, and also for the left vs right
  coideals.}

\begin{Remark}\label{rem:shift_of_R}
  Suppose $\sigma' = \id$.
  By \cref{prop:dual_rep}, $F_a^\sigma[-u_a]$ is isomorphic to $(\lsp{t}F_a)[-u_a-\frac{\hbar}4 c_\fg]$
  as a representation of $\sY$. Thus the $R$-matrix $R_{F_1,F_2^\sigma}(u_1+u_2)$
  is identified with $R_{F_1,\lsp{t}F_2}(u_1+u_2+\frac{\hbar}4 c_\fg)$.

  If $\sigma'\neq \id$, $F_a^\sigma[-u_a]$ is $(\lsp{t}F_a)[-u_a-\frac{\hbar}4 c_\fg]$, pull-backed 
  by the automorphism of $\sY$ given by the diagram automorphism $\sigma'$.

  If $F_a$ is $F_i$ for $i\in I$ and $\invast[i]=i$, then $F_a^\sigma = F_i$. This happens
  for example when $Q$ is of type A and $\sigma'\neq\id$. 

  In either way, we can understand the reflection equation in terms of representations of $\sY$,
  without invoking $\sigma$-quiver varieties.
\end{Remark}

\subsection{Unitarity}\label{subsec:unitarity}

The $K$-matrix $K(u)$ satisfies the unitarity $K(u)K(-u) = \id$ as
observed in \cite[Prop.~5.5.2]{MR3900699}.
In our framework of understanding the $K$-matrix as a rational homomorphism
from $F[u]\otimes \mathfrak H$ to $F^\sigma[-u]\otimes \mathfrak H$ as \eqref{eq:h:47},
the unitarity says
\begin{equation*}
  K_{F^\sigma}(-u) K_F(u) = \id.
\end{equation*}

\subsection{Definition}

Recall \eqref{eq:h:11}, which has now a geometric construction via quiver varieties:
\begin{equation*}
  \begin{split}
  \scH &= F_{i_1}[u_1]\otimes F_{i_2}[u_2]\otimes\dots\otimes F_{i_n}[u_n] \\
  &= H^*_{\TT_{\varpi_{i_1}}}(\fM(\varpi_{i_1}))\otimes H^*_{\TT_{\varpi_{i_2}}}(\fM(\varpi_{i_2}))\otimes\dots\otimes 
  H^*_{\TT_{\varpi_{i_n}}}(\fM(\varpi_{i_n})).
  \end{split}
\end{equation*}

Choosing an auxiliary space $F$ from $\{ F_i\}$ and
$\mathfrak H = H^*_{\TT_{\bw^0}}(\fM^\sigma(\bw^0))$ as in \cref{subsec:reflection-equation},
we consider
\begin{equation}\label{eq:h:40}
  \begin{multlined}
      \sS_F(u) \defeq 
       (R_{F_{i_1}, F^\sigma})(u+u_1)_{21} 
      \cdots
      (R_{F_{i_n}, F^\sigma})(u+u_n)_{21}\qquad\qquad\qquad\\
      \qquad\qquad\qquad
      K_F(u)\,
      R_{F, F_{i_n}}(u-u_n)
      \cdots
      R_{F, F_{i_1}}(u-u_1).
  \end{multlined}
\end{equation}
For example, the composite of the vertical arrows in the left column of \eqref{eq:h:37}
is an example of $\sS_F(u)$.
We interpret $K_F(u)$ as a rational operator in 
$\Hom(F[u]\otimes \mathfrak H, F^\sigma[-u]\otimes \mathfrak H)$ as in \cref{subsec:reflection-equation}, and
acting by $\id$ on the factor $\scH$.
Then $\sS_F(u)$ is a rational operator in
\begin{equation*}
 \Hom(F[u]\otimes\scH\otimes \mathfrak H, 
 F^\sigma[-u]\otimes\scH\otimes \mathfrak H). 
\end{equation*}
We expand it around $u=\infty$ as for $\sT_F(u)$ in \eqref{eq:h:12}.
The constant term is $\sigma\otimes\id_{\scH\otimes\mathfrak H}$.
This is the $R$-matrix from the chamber $\{ u > u_1 > u_2 > \cdots > u_n > 0 > -u_n > \cdots
   > -u_2 > -u_1\}$ to the chamber $\{ u_1 > u_2 > \cdots > u_n > 0 > -u_n > \cdots
   > -u_2 > -u_1 > u \}$. 

\begin{Definition}
   We define the \emph{extended twisted Yangian} $\sX^{\mathrm{tw}}$ as the subalgebra of
  \begin{equation*}
    \prod_{\mathfrak H} \prod_{\scH}
    \End_{\Bbbk[u_1,\ldots,u_n]}(\scH\otimes\mathfrak H)
  \end{equation*}
  generated by coefficients of $\hbar u^{-N-1}$ of matrix entries of $\sS_F(u) - \sigma\otimes\id_{\scH\otimes\mathfrak H}$
  for all $N\ge 0$ and all auxiliary spaces $F$.
\end{Definition}

\begin{Remark}
  The definition of the extended twisted Yangian $\sX^{\mathrm{tw}}$ is 
  due to Li \cite[\S5.6]{MR3900699}. We slightly modify it by interpreting
  linear operators in a different way, as explained in \cref{subsec:reflection-equation}.
  We also add the additional factor over $\mathfrak H = H^*_{\TT_{\bw^0}}(\fM^\sigma(\bw^0))$, 
  as it could potentially give interesting representations of $\sX^{\mathrm{tw}}$, as
  indicated in \cref{ex:distinguished}.
\end{Remark}

Comparing \eqref{eq:h:40} with \eqref{eq:h:12}, we get
\begin{equation}\label{eq:h:39}
  \sS_F(u) = \sT_{F^\sigma}(-u)^{-1} K_F(u) \sT_F(u),
\end{equation}
where $\sT_{F^\sigma}(-u)$ is defined by \eqref{eq:h:12} with $F$, $u$ replaced by
$F^\sigma$, $-u$ respectively.

When $\bw^0 = 0$, we have $\mathfrak H = \Bbbk$. Therefore we have a projection
$\sX^{\mathrm{tw}} \to \prod_{\scH} \End(\scH)$.
By \eqref{eq:h:39}, the image of this projection is contained in the image
of the extended Yangian $\sX$. As remarked in \cref{subsec:translation},
the homomorphism of $\sX$ is injective. Therefore we have a homomorphism
\begin{equation}\label{eq:h:41}
  \sX^{\mathrm{tw}} \to \sX.
\end{equation}

\begin{Remark}\label{rem:smaller_ext-twist-Yang}
  (1) 
  Even if $\bw^0 = 0$, $K_F(u) - \sigma \in \frac{\hbar}{u}\Hom(F, F^\sigma)[[u^{-1}]]$ could be nonzero.
  Indeed, we will see that $K_F(u) - \sigma$ is zero when the type is $(+)$,
  but nonzero when the type is $(-)$ in \cref{sec:ex2}.
  In particular, the homomorphism \eqref{eq:h:41} depends on the chosen type.

  (2) Viewing the image of $\sX^{\mathrm{tw}}$ in \eqref{eq:h:41} as a subalgebra of $\sX$,
  the definition has similarity with one in \cite[Def.~3.1]{MR3545488}.
  %
 %
  
  (3)
As in \cref{thm:MO-Yangian = ext-Yangian}(3), we expect the center $Z(\sX^{\mathrm{tw}})$
of $\sX^{\mathrm{tw}}$ is the polynomial ring in variables $\ch_{2k}(\CW_i)$ from
$i$ in the first factor of $G_\bw^\sigma$,
and $\ch_k(\CW_i) + \ch_k(\CW_j^*)$ from $\{ i,j\}$
in the second factor of $G_\bw^\sigma$ in \cref{def:framing-involution}.
Here we only consider even degree parts in the first case, as $\CW_i\cong\CW_i^*$.
We then expect that $\sX^{\mathrm{tw}}$ decomposes as the tensor product
$\sY^{\mathrm{tw}}\otimes Z(\sX^{\mathrm{tw}})$, where $\sY^{\mathrm{tw}}$
is the inverse image of $\tilde\sY_R$ under the homomorphism \eqref{eq:h:41}.

  (4) In view of \cref{rem:smaller_ext-Yang}, we can choose and fix $i\in I$,
  and consider only tensor products of $F_i[u]$ to define the extended twisted Yangian.
  Namely we define $\sX'$ from a single $F_i$ as in \cref{rem:smaller_ext-Yang}, and
  define $\sX^{\prime\mathrm{tw}}$ by the image of $\sT_{F_i^\sigma}(-u)^{-1} K_{F_i}(u) \sT_{F_i}(u)$.
  We can even restrict $\mathfrak H$ only to $\bw^0$. This version of $\sX^{\prime\mathrm{tw}}$
  is closer to the twisted Yangian defined in \cite{MR3545488}.
  See \cref{thm:O_twisted,thm:MR_refl} below.
\end{Remark}

\subsection{Coideal property}

The tensor product $\scH_1\otimes\scH_2$ of two spaces $\scH_1$, $\scH_2$ of form \eqref{eq:h:11} is again
of the same form. The projection
\begin{equation*}
   \prod_{\mathfrak H} \prod_{\scH} \End(\scH\otimes\mathfrak H)
    \to \prod_{\mathfrak H}\prod_{\scH_1,\scH_2} 
    \End(\scH_1\otimes\scH_2\otimes\mathfrak H)
\end{equation*}
gives a coproduct
\begin{equation*}
  \Delta\colon \sX^{\mathrm{tw}}\to \sX\otimes\sX^{\mathrm{tw}}
\end{equation*}
by sending $\sS_F(u)$ to $\sT_{F^\sigma}(-u)^{-1}\sT_F(u)\otimes \sS_F(u)$.
This is compatible with \eqref{eq:h:41}, i.e., \eqref{eq:h:41} sends
this coproduct to the coproduct of the extended Yangian $\sX$.
Although we do not know \eqref{eq:h:41} is injective as we add factors
$\mathfrak H$, the compatibility means that $\sX^{\mathrm{tw}}$ is
a \emph{left coideal subalgebra} of $\sX$.

The coideal property implies that the category of $\sX^{\mathrm{tw}}$-modules
is a \emph{module category} over the monoidal category of $\sX$-modules.

\subsection{Twisted current algebra as the associated graded}

Recall that $\sX$ has an ascending filtration defined in \cref{def:filt}. 
The associated graded algebra is $\bU(\fg^{\mathrm{ext}}[u])$ by \cref{thm:PBW}.

Viewing $\sX$ as $\sX_{\mathrm{MO}}$ by \cref{thm:MO-Yangian = ext-Yangian}, the involution $\sigma$
on $\fM(\bw)$ induces an involution on $\sX$, and also $\fg^{\mathrm{ext}}[u]$. Let us denote all of them by $\sigma$ for simplicity.
\begin{NB}
  We probably need to assume some conditions on polarization to ensure the above claim.
\end{NB}%
Recall further that $\fg^{\mathrm{ext}}$ is $\fg\oplus\mathfrak z_I$, and $\mathfrak z_I$ is spanned
by $\delta_i$, $i\in I$.

\begin{Lemma}\label{lem:sigma_on_g}
  \textup{(1)} The involution $\sigma$ on $\fg^{\mathrm{ext}}$ is given by
  \begin{equation*}
    \sigma(h_i) = -h_{\sigma'(i)}, \quad \sigma(\delta_i) = \delta_{\sigma'(\invast[i])}, \quad
    \sigma(e_i) = \pm f_{\sigma'(i)}, \quad \sigma(f_i) = \pm e_{\sigma'(i)}.
  \end{equation*}

  \textup{(2)} The involution $\sigma$ on $\fg^{\mathrm{ext}}[u]$ is given by
  $\sigma(x u^k) = \sigma(x) (-u)^k$ for $x\in\fg^{\mathrm{ext}}$.
\end{Lemma}

\begin{proof}
  (1) By the formula of $\sigma(\mu)$ in \eqref{eq:h:50}, the formula for $\sigma(h_i)$ follows,
as $h_i$ is $\langle h_i,\mu\rangle = \operatorname{rank}\CU_i$ on $\fM(\bv,\bw)$ with $\mu = \bw - \bC\bv$.
On the other hand, $\delta_i$ counts dimension of $W_i$, hence the formula of $\sigma(\lambda)$ in
\eqref{eq:h:50} implies the formula for $\sigma(\delta_i)$.

Elements $e_i$, $f_i$ are generators of root subspaces of $\fg^{\mathrm{ext}}$, which is $1$-dimensional.
Considering their weights, we have the above formula up to nonzero multiplicative constants. We have
natural integral structures on homology groups, and $e_i$, $f_i$ are integral generators. Therefore
constants must be $\pm 1$. Being $\sigma$ an involution, $\pm$ must be the same for $e_i$ and $f_i$.

(2) It is enough to check the assertion for $x = h_i$, $\delta_i$, $k=1$. Then $h_i u$, $\delta_i u$ are
given by the first Chern classes of $\CU_i$, $\CW_i$ respectively. Since $\CU_i$, $\CW_i$ are
replaced by $-\CU_{\sigma'(i)}^\vee$, $\CW_{\sigma'(\invast[i])}^\vee$ under $\sigma$ by \eqref{eq:h:51},
we have the assertion.
\end{proof}

\begin{Remark}\label{rem:symmetric-pair}
  From the computation of $\sigma$ on $\fg^{\mathrm{ext}}$, we can compute which type of symmetric pair
  $(\fg,\fg^\sigma)$ in \cite[Table~1]{MR3900699} appears from the $\sigma$-quiver variety. More precisely,
  $\Gamma$ and $a$ there are the Dynkin diagram of $\fg$ and $\sigma'\circ\invast$
  respectively.
\end{Remark}

Let us introduce a filtration $\sX^{\mathrm{tw}}_{\le 0}\subset \sX^{\mathrm{tw}}_{\le 1}\subset\cdots$ on $\sX^{\mathrm{tw}}$ from the filtration on $\sX$ via
the homomorphism \eqref{eq:h:41}.
As in \cite[Prop.~3.3]{MR3545488}, we have the following:
\begin{Proposition}
  The image of the associated graded algebra $\gr\sX^{\mathrm{tw}}$ of $\sX^{\mathrm{tw}}$ 
  under \eqref{eq:h:41} is isomorphic to the universal enveloping algebra
  $\bU(\fg^{\mathrm{ext}}[u]^\sigma)$ of the $\sigma$-twisted current algebra
  $\fg^{\mathrm{ext}}[u]^\sigma$.
\end{Proposition}

\begin{Corollary}[\protect{cf.\ \cite[Cor.~3.3]{MR3545488}}]
  There is an isomorphism between $\bU((\fg^{\mathrm{ext}})^\sigma)$
  and the image of $\sX^{\mathrm{tw}}_{\le 0}$ under \eqref{eq:h:41}.
\end{Corollary}

\subsection{Representations on cohomology groups of \texorpdfstring{$\sigma$}{σ}-quiver varieties}

By its definition, $\sX^{\mathrm{tw}}$ has a representation
$\scH\otimes\mathfrak H = H^*_{\TT_{\varpi_{i_1}}}(\fM(\varpi_{i_1}))\otimes
\cdots\otimes H^*_{\TT_{\varpi_{i_n}}}(\fM(\varpi_{i_n}))\otimes
H^*_{\TT_{\bw^0}}(\fM^\sigma(\bw^0))$.
As in \cref{thm:tensor}, it is isomorphic to $H^{\TT_\bw}_*(\fM^\sigma(\bw)^{\rho\ge 0})$,
where $\bw = \sum_k \varpi_{i_k} + \sigma(\varpi_{i_k}) + \bw^0$ and
$\rho\colon \CC^\times\to A_\bw$ is an appropriate character.
As in \cref{subsec:Y_rep_coh}, we have the following:
\begin{Theorem}
  Pushforward homomorphisms with respect to the inclusion
  $\fL^\sigma(\bw)\hookrightarrow \fM^\sigma(\bw)^{\rho\ge 0}\hookrightarrow \fM^\sigma(\bw)$
  \begin{equation*}
    H^{\TT_\bw}_*(\fL^\sigma(\bw)) \to H^{\TT_\bw}_*(\fM^\sigma(\bw)^{\rho\ge 0})\to
    H^{\TT_\bw}_*(\fM^\sigma(\bw)) 
  \end{equation*}
  are homomorphisms of $\sX^{\mathrm{tw}}$-modules. The leftmost and rightmost ones, after
  replacing $\TT_\bw$ by $\CC^\times_\hbar\times G_\bw^\sigma$, are also $\sX^{\mathrm{tw}}$-modules,
  and forgetting homomorphisms $H^{\CC^\times_\hbar\times G_\bw^\sigma}(\fL^\sigma(\bw)\text{ or } \fM^\sigma(\bw)) \to H^{\TT_\bw}_*(\fL^\sigma(\bw)\text{ or } \fM^\sigma(\bw))$
  are homomorphisms of $\sX^{\mathrm{tw}}$-modules.
\end{Theorem}

\subsection{Three additional $RTT$ relations}

The $RTT$ relations for $F_1[u]\otimes F_2^\sigma[-v]\otimes \scH$,
$F_1^\sigma[-u]\otimes F_2[v]\otimes \scH$,
$F_1^\sigma[-u]\otimes F_2^\sigma[-v]\otimes \scH$ are
\begin{align}
  R_{F_1, F_2^{\sigma}}(u+v) \sT_{F_1}(u) \sT_{F_2^\sigma}(-v) &=
    \sT_{F_2^\sigma}(-v) \sT_{F_1}(u) R_{F_1, F_2^{\sigma}}(u+v),\label{eq:h:6}\\
  R_{F_1^{\sigma},F_2}(-u-v) \sT_{F_1^\sigma}(-u) \sT_{F_2}(v) &=
    \sT_{F_2}(v)\sT_{F_1^\sigma}(-u) R_{F_1^{\sigma},F_2}(-u-v),\label{eq:h:7}\\
  R_{F_1^\sigma, F_2^\sigma}(-u+v) \sT_{F_1^\sigma}(-u) \sT_{F_2^\sigma}(-v) &=
    \sT_{F_2^\sigma}(-v)\sT_{F_1^\sigma}(-u) R_{F_1^\sigma, F_2^\sigma}(-u+v).\label{eq:h:8}
\end{align}
\begin{NB}
  Old notation:

Here $R^{\bullet\sigma}(u-v) = R_{F_1,F_2^\sigma}(u-v)$, 
$R^{\sigma\bullet}(u-v)=R_{F_1^\sigma,F_2}(u-v)$, 
$R^{\sigma\sigma}(u-v) = R_{F_1^\sigma,F_2^\sigma}(u-v)$.
\end{NB}%

Indeed, \eqref{eq:h:7} can be deduced from \eqref{eq:h:6} as follows.
Conjugating \eqref{eq:h:6} by $(12)$ and exchanging names of $F_1$, $F_2$,
we get
\begin{NB}
\begin{equation*}
  R^{\bullet\sigma}(u-v)_{21} \sT_2(u)\sT_1^\sigma(v) =
  \sT_1^\sigma(v) \sT_2(u) R^{\bullet\sigma}(u-v)_{21}.
\end{equation*}
\end{NB}%
\begin{equation*}
  R_{F_2, F_1^{\sigma}}(u+v)_{21} \sT_{F_2}(u)\sT_{F_1^\sigma}(-v) =
  \sT_{F_1^\sigma}(-v) \sT_{F_2}(u) R_{F_2, F_1^{\sigma}}(u+v)_{21}.
\end{equation*}
By the unitarity \eqref{eq:h:25}, we have
$R_{F_2, F_1^{\sigma}}(u+v)_{21} = R_{F_1^{\sigma}, F_2}(-u-v)^{-1}$, hence
\begin{NB}
\begin{equation*}
  \sT_2(u)\sT_1^\sigma(v) R^{\sigma\bullet}(v-u) =
  R^{\sigma\bullet}(v-u) \sT_1^\sigma(v) \sT_2(u).
\end{equation*}
\end{NB}%
\begin{equation*}
  \sT_{F_2}(u)\sT_{F_1^\sigma}(-v) R_{F_1^{\sigma}, F_2}(-u-v) =
  R_{F_1^{\sigma}, F_2}(-u-v) \sT_{F_1^\sigma}(-v) \sT_{F_2}(u).
\end{equation*}
By exchanging $u$, $v$, we get \eqref{eq:h:7}.

\subsection{Reflection algebra}
\label{subsec:reflection-algebra}

\begin{NB}
The $K$-matrices $\{ K_{F}(u), K_{F_1}(u_1) \}$ satisfies the
reflection equation. That means that the following diagram commutes:
\begin{equation*}
  \begin{CD}
    F[u]\otimes F_1[u_1] @>{1\otimes K_{F_1}(u_1)}>> F[u]\otimes F_1^\sigma[-u_1] \\
    @V{R_{F, F_1}(u-u_1)}VV @VV{R_{F, F_1^\sigma}(u+u_1)}V \\
    F[u]\otimes F_1[u_1] @. F[u]\otimes F_1^\sigma[-u_1] \\
    @V{K_{F}(u)\otimes 1}VV @VV{K_{F}(u)\otimes 1}V \\
    F^\sigma[-u]\otimes F_1[u_1] @. F^\sigma[-u]\otimes F_1^\sigma[-u_1] \\
    @V{(R_{F_1, F^\sigma})_{21}(u + u_1)}VV @VV{(R_{F_1^\sigma, F^\sigma})_{21}(u - u_1)}V \\
    F^\sigma[-u]\otimes F_1[u_1] @>>{1\otimes K_{F_1}(u_1)}> F^\sigma[-u]\otimes F_1^\sigma[-u_1]\rlap{.}
  \end{CD}
\end{equation*}
We define the operator $\sS_{F}(u)$ as the composition of the three vertical arrows
on the left and right hand sides. The left and right sides are equal up to exchange of $F_1[u_1]$ and $F_1^\sigma[-u_1]$.
Therefore the use of this notation is justified. More formally, we consider $\sS_{F}(u)$ as
collection of operators $F[u]\otimes \scH \to F^\sigma[-u]\otimes \scH$ for all $\scH$
soon below. The above two vertical arrows are $F_1$ and $F_1^\sigma$ components of $\sS_{F}(u)$.

We use bases of $F$ and $F^\sigma$ to regard matrix entries of $\sS_{F}(u)$ as generating functions
of operators on $F_1[u_1]$.
Then we interpret the above diagram as
$K_{F_1}(u_1)\colon F_1[u_1]\to F_1^\sigma[-u_1]$ intertwines operators $\sS_{F}(u)$.

A bit more generally, let us increase the number of factors in $\scH$.

We move from the chamber $\{ u > u_1 > u_2 > \cdots > u_n > 0 > -u_n > \cdots
   > -u_2 > -u_1\}$ to the chamber $\{ u_1 > u_2 > \cdots > u_n > 0 > -u_n > \cdots
   > -u_2 > -u_1 > u \}$. A special case is $\{ u > 0\}$ to $\{ 0 > u\}$.
   It happens on the module $F[u]\otimes\scH$, where $\scH$ is
   $F_{i_1}[u_1]\otimes F_{i_2}[u_2]\otimes\dots\otimes F_{i_n}[u_n]$.
   Say $n=0$, we have
    \begin{equation*}
      \begin{CD}
        F[u] @>x>> F[u] \\
        @V{K_F(u)}VV @VV{K_F(u)}V \\
        F^\sigma[-u] @>>x> F^\sigma[-u]
      \end{CD}
    \end{equation*}
    for $x\in\sX^{\mathrm{tw}}$, which is a matrix coefficient of $\sS_{F'}(u)$
    above with $F$, $F_1$ are replaced by $F'$, $F$ respectively.
    
    For $n=1$, we have
    \begin{equation*}
      \begin{CD}
        F[u]\otimes F_1[u_1] @>{\Delta(x)=\sum x_{(1)}\otimes x_{(2)}}>> F[u]\otimes F_1[u] \\
        @V{(12)R_{F, F_1}(u-u_1)}VV @VV{(12)R_{F, F_1}(u-u_1)}V \\
        F_1[u_1]\otimes F[u] @>>{\sum x_{(1)}\otimes x_{(2)}}> F_1[u]\otimes F[u] \\
        @V{1\otimes K_F(u)}VV @VV{1\otimes K_F(u)}V \\
        F_1[u_1]\otimes F^\sigma[-u] @>>{\sum x_{(1)}\otimes x_{(2)}}> F_1[u_1]\otimes F^\sigma[-u] \\
        @V{(12) R_{F_1, F}(u_1+u)}VV @VV{(12) R_{F_1, F}(u_1+u)}V \\
        F^\sigma[-u]\otimes F_1[u_1] @>>{\sum x_{(1)}\otimes x_{(2)}}> F^\sigma[-u]\otimes F_1[u_1]\rlap{.}
      \end{CD}
    \end{equation*}
    We now stop exchange of factors:
    \begin{equation*}
      \begin{CD}
        F[u]\otimes F_1[u_1] @>{\Delta(x)=\sum x_{(1)}\otimes x_{(2)}}>> F[u]\otimes F_1[u] \\
        @V{R_{F, F_1}(u-u_1)}VV @VV{(12)R_{F F_1}(u-u_1)}V \\
        F[u]\otimes F_1[u_1]@>>{\Delta^{\mathrm{op}}(x) = \sum x_{(2)}\otimes x_{(1)}}> F[u]\otimes F_1[u] \\
        @V{K_F(u)\otimes 1}VV @VV{K_F(u)\otimes 1}V \\
        F^\sigma[-u]\otimes F_1[u_1] @>>{\Delta^{\mathrm{op}}(x) = \sum x_{(2)}\otimes x_{(1)}}> F^\sigma[-u]\otimes F_1[u_1] \\
        @V{(R_{F_1, F^\sigma})_{21}(u + u_1)}VV @VV{(R_{F_1, F^\sigma})_{21}(u + u_1)}V \\
        F^\sigma[-u]\otimes F_1[u_1] @>>{\Delta(x)=\sum x_{(1)}\otimes x_{(2)}}> F^\sigma[-u]\otimes F_1[u_1]\rlap{.}
      \end{CD}
    \end{equation*}
    Here we have used $(12)R_{F^\sigma F_1}(u+u_1)(12)
    = (R_{F_1 F^\sigma})_{21}(u+u_1)$ by the definition.

    Now inductively, we consider
    \begin{equation*}
      \sS_F(u) \defeq 
       (R_{F_{i_1} F^\sigma})_{21}(u+u_1) 
      \cdots
      (R_{F_{i_n} F^\sigma})_{21}(u+u_n)\,K_F(u)\,
      R_{F, F_{i_n}}(u-u_n)
      \cdots
      R_{F, F_{i_1}}(u-u_1).
    \end{equation*}
    As for $\sT_F(u)$, we consider operators for all possible choices of $\scH$ simultaneously.
    Note that $K_F(u)$ does not depend on the choice of $\scH$. Therefore we
    get $\sT_{F^\sigma}(-u)^{-1} K_F(u) \sT_F(u) - \id \in \frac{\hbar}u 
    \Hom_{\Bbbk}(F, F^\sigma)\otimes\prod_{\scH}\End_{\Bbbk[u_1,u_2,\dots, u_n]}(\scH)[[u^{-1}]]$.
    Here $\sT_F^\sigma(-u)$ is defined as in \eqref{eq:h:12} with $F$ replaced by $F^\sigma$:
    \begin{equation*}
      \begin{split}
       \sT_{F^\sigma}(-u) & \defeq 
      R_{F^\sigma,F_{i_n}}(-u-u_n) R_{F^\sigma, F_{i_{n-1}}}(-u-u_{n-1})\cdots
      R_{F^\sigma,F_{i_1}}(-u-u_1) \\
      &= \left[(R_{F_{i_1},F^\sigma})_{21}(u+u_1) (R_{F_{i_2},F^\sigma})_{21}(u+u_2)\cdots
      (R_{F_{i_n},F^\sigma})_{21}(u+u_n)\right]^{-1}.
      \end{split}
    \end{equation*}

    More generally, we could have a $\sigma$-quiver variety $\fM^\sigma_\zeta(\bw)$
    with $\CC^\times$-action, where $\CC^\times$ is given by a cocharacter in $G^\sigma_\bw$.
    Then the fixed point set is of form $\fM^\sigma_\zeta(\bw^0)\times \fM_\zeta(\bw^1)$
    with $\bw = \bw^0 + \bw^1 + \sigma(\bw^1)$. We have a $K$-matrix
    \begin{equation*}
      K(u) \in \End(H^*_{\CC^\times_\hbar}(\fM^\sigma_\zeta(\bw^0))\otimes
      H^*_{\CC^\times_\hbar}(\fM_\zeta(\bw^1)))(u).
    \end{equation*}
    It could be possible that the first factor have at most only $\grpO(1)^N$-action,
    but still nontrivial. An example should be written *****.
\end{NB}%

By the construction of the $R$-matrix from the stable envelope, $\sS(u)$ apparently
satisfies the reflection equation
\begin{NB}
\begin{NB2}
\begin{equation}
  R_{21}^{\sigma\sigma}(u-v)\sS_2(v) R^{\sigma\bullet}(-u-v) \sS_1(u)
  = \sS_1(u) R^{\bullet\sigma}_{21}(-u-v) \sS_2(v) R(u-v),
\end{equation}
\end{NB2}%
\begin{equation}
  R_{21}^{\sigma\sigma}(u-v)\sS_1(u) R^{\bullet\sigma}(u+v) \sS_2(v)
  = \sS_2(v) R^{\bullet\sigma}_{21}(u+v) \sS_1(u) R(u-v),
\end{equation}
\end{NB}%
\begin{equation}\label{eq:h:42}
  \begin{multlined}
  \sS_{F_2}(u_2) R_{F_2,F_1^{\sigma}}(u_1+u_2)_{21} \sS_{F_1}(u_1)
   R_{F_1, F_2}(u_1-u_2) \qquad\qquad\qquad\\
 \qquad\qquad\qquad
  = R_{F_2^\sigma,F_1^\sigma}(u_1-u_2)_{21}\sS_{F_1}(u_1) 
  R_{F_1, F_2^{\sigma}}(u_1+u_2) \sS_{F_2}(u_2),    
  \end{multlined}
\end{equation}
as (rational) operator from $F_1[u_1]\otimes F_2[u_2]\otimes \scH\otimes\mathfrak H$
to $F_1^\sigma(-u_1)\otimes F_2^\sigma(-u_2)\otimes\scH\otimes\mathfrak H$.

\begin{NB}
The right hand side of \eqref{eq:h:1} gives an isomorphism
$F_1[u]\otimes F_2[v]\otimes \scH\to F_1^\sigma[-u]\otimes
F_2^\sigma[-v]\otimes \scH$ intertwines $\sY^\sigma$ actions as
\begin{equation*}
  \begin{CD}
    F_1[u] \otimes F_2[v] \otimes \scH @>{(\Delta\otimes\id)\Delta}>>
    F_1[u] \otimes F_2[v] \otimes \scH\\
    @V{R_{F_1F_2}(u-v)}VV @VV{R_{F_1F_2}(u-v)}V  \\
    F_1[u]\otimes F_2[v] \otimes \scH @>{(\Delta^{\mathrm{op}}\otimes\id)\Delta}>>
    F_1[u]\otimes F_2[v] \otimes \scH\\
    @V{\sS_{F_1}(u)_{13}}VV @VV{\sS_{F_1}(u)_{13}}V \\
    F_1^\sigma[-u]\otimes F_2[v] \otimes \scH @>{(\Delta^{\mathrm{op}}\otimes\id)\Delta}>>
    F_1^\sigma[-u]\otimes F_2[v] \otimes \scH\\
    @V{R^{\bullet\sigma}_{21}(u+v)}VV @VV{R^{\bullet\sigma}_{21}(u+v)}V  \\
    F_1^\sigma[-u] \otimes F_2[v] \otimes \scH @>{(\Delta\otimes\id)\Delta}>>
    F_1^\sigma[-u] \otimes F_2[v] \otimes \scH\\
    @V{\sS_{F_2}(v)_{23}}VV @VV{\sS_{F_2}(v)_{23}}V \\
    F_1^\sigma[-u] \otimes F_2^\sigma[-v] \otimes \scH @>{(\Delta\otimes\id)\Delta}>>
    F_1^\sigma[-u] \otimes F_2^\sigma[-v] \otimes \scH
    \rlap{.}
  \end{CD}
\end{equation*}
Note that $\sY^\sigma$ is a left coideal subalgebra with respect to
the coproduct $\Delta$, hence $(\Delta^{\mathrm{op}}\otimes\id)\Delta$ in the second
horizontal arrow is in $\sY\otimes\sY^\sigma$, when applied to $F_2[v]\otimes (F_1[u]\otimes \scH)$
after conjugating by $(12)$. Therefore it commutes
with $\sS_{F_1}(u)_{13} = (12)(\id\otimes\sS_{F_1}(u))(12)$. The same remark applies, more obviously, to the bottom
commutative square.

Similarly the left hand side gives an intertwiner
\begin{equation*}
  \begin{CD}
    F_1[u] \otimes F_2[v] \otimes \scH @>{\Delta}>>
    F_1[u] \otimes F_2[v] \otimes \scH\\
    @V{\sS_{F_2}(v)_{23}}VV @VV{\sS_{F_2}(v)_{23}}V  \\
    F_1[u]\otimes F_2^\sigma[-v] \otimes \scH @>{\Delta}>>
    F_1[u]\otimes F_2^\sigma[-v] \otimes \scH \\
    @V{R^{\bullet\sigma}(u+v)}VV @VV{R^{\bullet\sigma}(u+v)}V \\
    F_1[u]\otimes F_2^\sigma[-v] \otimes \scH @>{\Delta^{\mathrm{op}}}>>
    F_1[u]\otimes F_2^\sigma[-v] \otimes \scH \\
    @V{\sS_{F_1}(u)_{13}}VV @VV{\sS_{F_1}(u)_{13}}V  \\
    F_1^\sigma[-u] \otimes F_2^\sigma[-v] \otimes \scH
    @>{\Delta^{\mathrm{op}}}>> F_1^\sigma[-u] \otimes F_2^\sigma[-v]\otimes \scH \\
    @V{R^{\sigma\sigma}_{21}(u-v)}VV @VV{R^{\sigma\sigma}_{21}(u-v)}V \\
    F_1^\sigma[-u] \otimes F_2^\sigma[-v] \otimes \scH @>{\Delta}>>
    F_1^\sigma[-u] \otimes F_2^\sigma[-v] \otimes \scH
    \rlap{.}
  \end{CD}
\end{equation*}
The reflection equation \eqref{eq:h:1} means that two intertwiners
coincide.
\end{NB}%

Let us formally derive the reflection equation
from the basic one \eqref{eq:h:1} and the $RTT$ relations, 
though the argument is well-known.
Namely we show the following lemma, then \eqref{eq:h:42} follows by induction.

\begin{Lemma}\label{lem:embedding}
  If $\{ \sS_F(u)\}$ is a solution of the reflection equation \eqref{eq:h:42}, then
\begin{equation*}
  \begin{NB}
    \text{old version} 
    \sT^\sigma(-u)\sS(u)\sT(u)^{-1}\qquad\end{NB}%
    \sT_{F^\sigma}(-u)^{-1}\sS_F(u)\sT_F(u)
\end{equation*}
is also a solution of the reflection equation.
\end{Lemma}


\begin{proof}
  Let us introduce short-hands notations:
  \begin{gather*}
  R^{\bullet\sigma}(u+v) = R_{F_1,F_2^\sigma}(u+v), 
  R^{\sigma\bullet}_{21}(u+v)=R_{F_2,F_1^\sigma}(u+v)_{21},
  \\
  R^{\sigma\sigma}_{21}(u-v) = R_{F_2^\sigma,F_1^\sigma}(u-v)_{21},
  \\
  \sS_a(u) = \sS_{F_a}(u),
  \sT_a^\sigma(-u) = \sT_{F_a}(u), \sT_a(u) = \sT_{F_a}(u)
  \quad (a=1,2).
  \end{gather*}
\begin{NB}
  We need to show 
  \begin{equation}\label{eq:h:2}
    \begin{multlined}
      R_{21}^{\sigma\sigma}(u-v)\sT_2^\sigma(-v)\sS_2(v)\sT_2(v)^{-1}
      R^{\sigma\bullet}(-u-v)
      \sT_1^\sigma(-u)\sS_1(u) \sT_1(u)^{-1}\\
      = \sT_1^\sigma(-u) \sS_1(u) \sT_1(u)^{-1} R^{\bullet\sigma}_{21}(-u-v)
      \sT_2^\sigma(-v) \sS_2(v)\sT_2(v)^{-1}R(u-v).
    \end{multlined}
  \end{equation}
  We change the left hand side by first applying \eqref{eq:h:7}
  \begin{NB2}
  \(
    \sT_2(v)^{-1} R^{\sigma\bullet}(-u-v) \sT_1^\sigma(-u)
    = \sT_1^\sigma(-u) R^{\sigma\bullet}(-u-v) \sT_2(v)^{-1}
  \)
  \end{NB2}%
  to the middle part $\sT_2(v)^{-1} R^{\sigma\bullet}(-u-v) \sT_1^\sigma(-u)$,
  second using
  $\sS_2(v)\sT_1^\sigma(-u) = \sT_1^\sigma(-u)\sS_2(v)$,
  $\sT_2(v)^{-1}\sS_1(u) = \sS_1(u)\sT_2(v)^{-1}$%
  \begin{NB2}
    \begin{equation*}
      R^{\sigma\sigma}_{21}(u-v)\sT_2^\sigma(-v)
      \sT_1^\sigma(-u)\sS_2(v) R^{\sigma\bullet}(-u-v) \sS_1(u)
      \sT_2(v)^{-1}\sT_1(u)^{-1}
    \end{equation*}
  \end{NB2}%
  , then applying $RTT$ relation \eqref{eq:h:8} to the first three
  factors. The result is
  \begin{equation*}
    \sT_1^\sigma(-u)\sT_2^\sigma(-v) R^{\sigma\sigma}_{21}(u-v)
    \sS_2(v) R^{\sigma\bullet}(-u-v) \sS_1(u)
    \sT_2(v)^{-1}\sT_1(u)^{-1}.
  \end{equation*}
  We apply \eqref{eq:h:1} to the middle to get
  \begin{equation*}
    \sT_1^\sigma(-u)\sT_2^\sigma(-v) \sS_1(v) R^{\bullet\sigma}_{21}(-u-v)
    \sS_2(v) R(u-v) \sT_2(v)^{-1}\sT_1(u)^{-1}.
  \end{equation*}
  We apply the $RTT$ relation \eqref{eq:h:5} to the last three factors
  $R(u-v)\sT_2(v)^{-1}\sT_1(u)^{-1}$, and then use
  $\sT_2^\sigma(-v)\sS_1(v) = \sS_1(v) \sT_2^\sigma(-v)$. The result is
  \begin{equation*}
    \sT_1^\sigma(-u) \sS_1(v) \sT_2^\sigma(-v) R^{\bullet\sigma}_{21}(-u-v)
    \sS_2(v) \sT_1(u)^{-1} \sT_2(v)^{-1} R(u-v).
  \end{equation*}
  We apply $\sS_2(v) \sT_1(u)^{-1} = \sT_1(u)^{-1} \sS_2(v)$ and \eqref{eq:h:7}
  \begin{NB2}
  \begin{equation*}
    \sT_2^\sigma(-v)R_{21}^{\bullet\sigma}(-u-v) \sT_1(u)^{-1}
    = \sT_1(u)^{-1} R^{\bullet\sigma}_{21}(-u-v) \sT_2^\sigma(-v)
  \end{equation*}
  \end{NB2}%
  to this expression. The final result is the right hand side of
  \eqref{eq:h:2}.  
  \end{NB}

  We need to show 
  \begin{equation}\label{eq:h:2e}
    \begin{multlined}
      R_{21}^{\sigma\sigma}(u-v)\sT_1^\sigma(-u)^{-1}\sS_1(u)\sT_1(u)
      R^{\bullet\sigma}(u+v)
      \sT_2^\sigma(-v)^{-1} \sS_2(v) \sT_2(v)\\
      = \sT_2^\sigma(-v)^{-1} \sS_2(v) \sT_2(v) R^{\sigma\bullet}_{21}(u+v)
      \sT_1^\sigma(-u)^{-1} \sS_1(u)\sT_1(u)R(u-v).
    \end{multlined}
  \end{equation}
  We change the left hand side by first applying \eqref{eq:h:6}
  \begin{NB}
  \(
    \sT_1(u) R^{\bullet\sigma}(u+v) \sT_2^\sigma(-v)^{-1}
    = \sT_2^\sigma(-v)^{-1} R^{\bullet\sigma}(u+v) \sT_1(u)
  \)
  \end{NB}%
  to the middle part $\sT_1(u) R^{\bullet\sigma}(u+v) \sT_2^\sigma(-v)^{-1}$,
  second using
  $\sS_1(u)\sT_2^\sigma(-v)^{-1} = \sT_2^\sigma(-v)^{-1}\sS_1(u)$,
  $\sT_1(u)\sS_2(v) = \sS_2(v)\sT_1(u)$%
  \begin{NB}
    \begin{equation*}
      R^{\sigma\sigma}_{21}(u-v)\sT_1^\sigma(-u)^{-1}
      \sT_2^\sigma(-v)^{-1}\sS_1(u) R^{\bullet\sigma}(u+v) \sS_1(v)
      \sT_1(u)\sT_2(v)
    \end{equation*}
  \end{NB}%
  , then using $RTT$ relation \eqref{eq:h:8}. The result is
  \begin{equation*}
    \sT_2^\sigma(-v)^{-1}\sT_1^\sigma(-u)^{-1}R^{\sigma\sigma}_{21}(u-v)
    \sS_1(u) R^{\bullet\sigma}(u+v) \sS_2(v)
    \sT_1(u)\sT_2(v).
  \end{equation*}
  We apply \eqref{eq:h:42} to the middle to get
  \begin{equation*}
    \sT_2^\sigma(-v)^{-1}\sT_1^\sigma(-u)^{-1} \sS_2(v) R^{\bullet\sigma}_{21}(u+v) \sS_1(u) R(u-v)
    \sT_1(u)\sT_2(v).
  \end{equation*}
  We apply the $RTT$ relation \eqref{eq:h:5} to the last part
  $R(u-v)\sT_1(u)\sT_2(v)$, and then use
  $\sT_1^\sigma(-u)^{-1} \sS_2(v) = \sS_2(v) \sT_1^\sigma(-u)^{-1}$. We get
  \begin{equation*}
    \sT_2^\sigma(-v)^{-1} \sS_2(v) \sT_1^\sigma(-u)^{-1} R^{\bullet\sigma}_{21}(u+v)
    \sS_1(u) 
    \sT_2(v) \sT_1(u) R(u-v).
  \end{equation*}
  We apply $\sS_1(u) \sT_2(v) = \sT_2(v) \sS_1(u)$ and \eqref{eq:h:7}
  \begin{NB}
  \begin{equation*}
    \sT_1^\sigma(-u)^{-1} R_{21}^{\sigma\bullet}(u+v) \sT_2(v)
    = \sT_2(v) R^{\bullet\sigma}_{21}(u+v) \sT_1^\sigma(-u)^{-1}
  \end{equation*}
  \end{NB}%
  to this expression. The final result is the right hand side of
  \eqref{eq:h:2e}.
\end{proof}

Following \cite{MR3545488}, we define the extended reflection algebra.
\begin{Definition}
  The \emph{extended reflection algebra} $\sX^{\mathrm{ref}}$ is the algebra
  generated by $\{ s^{(r)}_{i;ab}\mid i\in I, a,b = 1,\dots, \dim F_i,
  r\ge 1\}$, subject to the relations \eqref{eq:h:42}. Here
  \(
    S_{F_i}(u) = \sigma + \left(\sum_{r\ge 1} s^{(r)}_{i;ab} u^{-r-1} \right)_{a,b}
  \)
  is regarded as an element of $\Hom(F_i, F_i^\sigma)\otimes \sX^{\mathrm{ref}}$ under
  the bases of $F_i$, $F_i^\sigma$.
\end{Definition}

When $\fg = \mathfrak{gl}_N$, this is introduced in \cite[\S2.13]{MR2355506} under
the name of \emph{extended twisted Yangian}. The terminology is
not consistent with ours, as our extended twisted Yangian is called just
twisted Yangian in \cite{MR2355506}. 

Since we have checked the reflection equation \eqref{eq:h:42} both geometric and
formal ways, we have a surjective homomorphism $\sX^{\mathrm{ref}}\twoheadrightarrow\sX^{\mathrm{tw}}$.

The kernel of $\sX^{\mathrm{ref}}\twoheadrightarrow\sX^{\mathrm{tw}}$ has been
understood in many examples, at least when $\sX^{\mathrm{tw}}$ is a version
in \cref{rem:smaller_ext-twist-Yang}(4), e.g., \cite[Th.~2.13.4]{MR2355506} for Olshanski
twisted Yangian, 
\cite[\S2.16, Examples 4-6]{MR2355506} for reflection equation algebras,\footnote{ 
$\sX^{\mathrm{ref}}$, $\sX^{\mathrm{tw}}$ are denoted by
$\mathrm{X}(\fg_N)$, $\mathrm{Y}(\fg_N)$ in \cite[Th.~2.13.4]{MR2355506} 
$\tilde{\mathcal B}(N,l)$, $\mathcal B(N,l)$ in \cite[\S2.16, Examples 4-6]{MR2355506}
respectively.
}
and
\cite[Th.~5.2]{MR3545488} for twisted Yangian
for symmetric pairs of type B,C,D.
But we do not pursue this direction here.

\begin{NB}
\begin{Example}
  The reflection equation for the twisted Yangian $\sY(\fg_N)$ (see
  \cite[(2.11)]{MR2355506}) is
  \begin{equation*}
    R(u-v) \sS_1(u) R^t(-u-v) \sS_2(v) = \sS_2(v) R^t(-u-v) \sS_1(u) R(u-v),
  \end{equation*}
  where $R^t$ is the transpose of $R$ with respect to the first or
  second factor. 
  The \eqref{eq:h:7} requires
\begin{equation*}
  R^{\sigma\bullet}(u-v) \sT_1^\sigma(u) \sT_2(v) = \sT_2(v)\sT_1^\sigma(u) R^{\sigma\bullet}(u-v).
\end{equation*}
If we look for such an equation in \cite{MR2355506}, the closest one is \cite[(1.30)]{MR2355506}:
\begin{equation*}
  \sT_2(v) R^t(u-v) \sT_1^t(u) = \sT_1^t(u) R^t(u-v) \sT_2(v).
\end{equation*}
Therefore it is natural to identify
\begin{equation}\label{eq:h:10}
  R^{\sigma\bullet}(u-v) = (R^t(u-v))^{-1},\quad
  \sT^\sigma(u) = \left(\sT^t(u)\right)^{-1}.
\end{equation}



%

Now we substitute \eqref{eq:h:10}, 
via $R^{\bullet\sigma}(u-v)= R^{\sigma\bullet}_{21}(v-u)^{-1}
= R^t(v-u)$,
into \eqref{eq:h:1}. We obtain
the defining equation of the coideal subalgebra in \cite[(2.11)]{MR2355506}, written above.

\begin{NB2}
  More directly, start from $R(u-v) \sT_1(u) \sT_2(v) = \sT_2(v) \sT_1(u) R(u-v)$, we
  conjugate by $(12)$ to obtain $R_{21}(u-v) \sT_2(u) \sT_1(v) = \sT_1(v) \sT_2(u) R_{21}(u-v)$.
  Take the transpose with respect to the second factor: $\sT_2^t(u) R_{21}(u-v)^t \sT_1(v) = \sT_1(v) R_{21}(u-v)^t \sT_2^t(u)$.
  Therefore $R_{21}(u-v)^t \sT_1(v) (\sT_2^t(u))^{-1} = (\sT_2^t(u))^{-1} \sT_1(v) R_{21}(u-v)^t$.
  We use $R_{21}(u-v)^t = R(u-v)^t$, hence we get \eqref{eq:h:6}.
\end{NB2}%

\cref{lem:embedding} below gives an embedding by
$\sT^\sigma(-u)^{-1}\sS(u)\sT(u)$. This coincides with \cite[Th.~2.4.3]{MR2355506} up to the application of the automorphism $\sT(u)\mapsto \sT^t(-u)$.
\end{Example}
\end{NB}%

\section{Examples -- instanton moduli spaces}\label{sec:ex1}

In this section we give examples of involutions on quiver varieties,
which give moduli spaces of $\SO$ and $\grpSp$ instantons on ALE spaces of type $\mathrm{A}_{\ell-1}$.
We also discuss type D and $\mathrm{E}_6$ briefly.
We compute the $K$-matrices of these examples. From the computation,
we can see that the twisted Yangian is the 
twisted Yangian $\sY(\mathfrak{o}_\ell)$ of Olshanski for $\SO$-instantons.
(See \cite[Chapter~2]{MR2355506} for the definition.)
On the other hand, we cannot choose an induced polarization for $\grpSp$-instantons,
except the case $\ell=2$. 

As a byproduct, we compute their Betti numbers of $\fM_\zeta^\sigma$ in terms of Young tableaux
in these examples.

\subsection{\texorpdfstring{$\grpSp$}{Sp}-instanton moduli spaces}

We consider type $\mathrm{A}_{\ell-1}$ quiver variety and name linear maps as follows:
\begin{equation}\label{eq:31}
  \begin{tikzcd}[execute at end picture={
      \draw[dashed,thin]
      ([yshift=2ex]$(\tikzcdmatrixname-1-3.north)!0.5!
      (\tikzcdmatrixname-1-4.north)$)
      --
      ([yshift=-2ex]$(\tikzcdmatrixname-2-3.north)!0.5!
      (\tikzcdmatrixname-2-4.north)$);
      }]
      V_1 \arrow[r,yshift=.5ex,"C_{\frac32}"] 
      \arrow[d,xshift=.5ex,near end,"D_{\frac12}"]
      &
      V_2 \arrow[l,yshift=-.5ex,"D_{\frac32}"]
      \arrow[r,yshift=.5ex,"C_{\frac52}"]
      &
      \phantom{V}
      \arrow[r,no head,very thick,dotted]
      \arrow[l,yshift=-.5ex,"D_{\frac52}"]
      &
      \phantom{V}
      \arrow[r,yshift=.5ex,"C_{\ell-\frac52}"] &
      V_{\ell-2} \arrow[l,yshift=-.5ex,"D_{\ell-\frac52}"]
      \arrow[r,yshift=.5ex,"C_{\ell-\frac32}"] & V_{\ell-1}
      \arrow[l,yshift=-.5ex,"D_{\ell-\frac32}"]
      \arrow[d,xshift=.5ex,"-C_{\ell-\frac12}"]
      \\
      W_1 
      \arrow[u,xshift=-.5ex, "C_{\frac12}"]
      &&\phantom{V}&\phantom{V}&& 
      W_{\ell-1}
      \arrow[u,xshift=-.5ex, near start, "D_{\ell-\frac12}"]
  \end{tikzcd}
\end{equation}
We also set $V_0 = W_1$, $V_\ell = W_{\ell-1}$ for notational convention.
The involution $\invast$ is given by $\invast[i] = \ell-i$, as indicated in
the middle broken line. We take $\sigma' = \id$.
We choose the linear orientation so that $C_{3/2}$, $C_{5/2}$, \dots,
$C_{\ell-3/2}$ are in the orientation. We put the $(-)$ sign on $C_{\ell-1/2}$
so that the defining equation has a uniform form.

We assume $\dim V_0 = \dim V_1 = \dots = \dim V_\ell$. We set this
number as $\bw_1$. Let
\begin{equation*}
  \bM 
  \defeq \bigoplus_{i=1}^\ell \Hom(V_{i-1},V_i)\oplus \Hom(V_i,V_{i-1}),
  \qquad
  G 
  \defeq \prod_{i=1}^{\ell-1} \GL(V_i).
\end{equation*}
Thus the quiver variety $\fM_\zeta$ is the symplectic reduction of $\bM$ by $G$ 
at the moment map level $\zeta_\CC$ with respect to the stability condition $\zeta_\RR$.
It is not so important, but we use the following convention for $\zeta$. The moment map equations are
\begin{equation*}
  C_{i-\frac12} D_{i-\frac12} - \zeta_{i-\frac12,\CC}
  = D_{i+\frac12} C_{i+\frac12} - \zeta_{i+\frac12,\CC}
  \quad\text{for $i=1,\dots,\ell-1$}.
\end{equation*}
Therefore we have $\ell$ parameters $\zeta_{\frac12,\CC}$, \dots, $\zeta_{\ell-\frac12,\CC}$.
However, only differences $\zeta_{\frac12,\CC} - \zeta_{\frac32,\CC}$, \dots,
$\zeta_{\ell-\frac32,\CC} - \zeta_{\ell-\frac12,\CC}$ matter. Thus we only have $\ell-1$ independent parameters.
In this convention, the reflection functor $S_i$ for a simple reflection at $i$ exchanges
$\zeta_{i-\frac12}$ and $\zeta_{i+\frac12}$, and leaves other parameters fixed.

As in \cref{sec:inv-quiver}, we define an involution $\sigma = S_{w_0}\circ\invast\circ t
$ on $\bM$.
Since $\bw - \bC\bv = 0$ for our choice of dimension vectors, we have $w_0\ast \bv^* = \bv$, hence
the target space of $\sigma$ has the same dimension vectors as the source space.
Next we choose $(-)$ as a type, hence we put a symplectic form $(\ ,\ )$ on
$W_1\oplus W_{\ell-1}$ such that
$W_1$ and $W_{\ell-1}$ are lagrangian subspaces. (Hence $W_1$ to $W_{\ell-1}$ are dual to each other.)
Then $W_{\ell-1}^*\oplus W_1^* \cong W_1\oplus W_{\ell-1}$. Therefore we can identify
the source and target spaces of $\sigma$ as a variety. Therefore $\sigma$ is regarded as an involution
on $\fM_\zeta(\bv,\bw)$.

Concretely, $\sigma$ changes, in the categorical quotient
$\fM_0$, the linear maps as
\begin{equation*}
  \begin{tikzcd}
      W_{\ell-1}^* \arrow[r,yshift=.5ex,"\lsp{t}C_{\ell-\frac12}"] &
      V_{\ell-1}^* \arrow[l,yshift=-.5ex,"\lsp{t}D_{\ell-\frac12}"]
      \arrow[r,yshift=.5ex,"\lsp{t}C_{\ell-\frac32}"]
      &
      \phantom{V}
      \arrow[r,no head,very thick,dotted]
      \arrow[l,yshift=-.5ex,"\lsp{t}D_{\ell-\frac32}"]
      &
      \phantom{V}
      \arrow[r,yshift=.5ex,"\lsp{t}C_{\frac32}"] &
      V_{1}^* \arrow[l,yshift=-.5ex,"\lsp{t}D_{\frac32}"]
      \arrow[r,yshift=.5ex,"\lsp{t}C_{\frac12}"] & 
      W_1^*
      \arrow[l,yshift=-.5ex,"-\lsp{t}D_{\frac12}"]
      \\
      W_1 \arrow[u,"\varphi_{1}","\cong"'] &
      &&&
      &
      W_{\ell-1}\rlap{,} \arrow[near start, u,"\varphi_{{\ell-1}}","\cong"']
  \end{tikzcd}
\end{equation*}
where $\varphi_{1}$, $\varphi_{{\ell-1}}$ are isomorphisms induced by the symplectic form on $W_1\oplus W_{\ell-1}$.
Hence they satisfy $\lsp{t}{\varphi}_{1} = -\varphi_{{\ell-1}}$.

The $\sigma$-quiver variety $\fM^\sigma_\zeta$ is the
moduli space of framed $\grpSp(2\bw_1)$-instantons on an ALE space of
type $\mathrm{A}_{\ell-1}$ with appropriate Chern class and boundary
condition, as reviewed in \cref{sec:ADHM}.

\begin{Example}\label{ex:ALE}
  Suppose $\bw_1 = 1$. Then $\fM_\zeta$ is an ALE space of type
  $\mathrm{A}_{\ell-1}$ with period $\zeta$. The involution $\sigma$ is trivial
  in this case. Let us give two proofs of this statement.

  The first one is computational. We have
  $\pi\colon\fM_\zeta\to\fM_{(0,\zeta)}$, which is a resolution of
  singularities \cite{Kr}. It is enough to check the assertion for
  $\fM_{(0,\zeta)}$. The coordinate ring of $\fM_{(0,\zeta)}$ is
  generated by
  \[
    x \defeq C_{\ell-\frac12}\dots C_{\frac12}, \quad
    y \defeq D_{\frac12} \dots D_{\ell-\frac12}, \quad
    z \defeq D_{\frac12} C_{\frac12}
  \]
  with relation
  $x y = (z + \zeta_{\frac12,\CC})\dots (z +
  \zeta_{\ell-\frac12,\CC})$. Note that $x$, $y$, $z$ are mapped to
  $x$, $y$,
  $C_{\ell-\frac12} D_{\ell-\frac12} = z + \zeta_{\ell-\frac12,\CC} -
  \zeta_{\frac12,\CC}$.
  In view of the definition \cite[\S1.1]{MR1775358}, these are sent to
  $x$, $y$, $z$ by the reflection functor.
  \begin{NB}
    In the symmetric case, $x$, $y$, $z$ are mapped to $-x$, $-y$, $z$.
  \end{NB}%

  We have an alternative, more conceptual explanation as follows: By
  \cite{KN}, $\fM_\zeta$ is the moduli space of framed
  $\U(2)$-instantons with fixed Chern class and boundary condition at
  $\infty$. In this example, the first Chern class $c_1$ is zero, and
  a self-dual boundary condition. Therefore an instanton is
  automatically isomorphic to its dual. 
  %
  $\SU(2)$-instantons are automatically
  $\grpSp(2)$-instantons as well, since $\SU(2)\cong\grpSp(2)$.

  Thus we have $\fM_\zeta^\sigma = \fM_\zeta$ is an ALE space of type $\mathrm{A}_{\ell-1}$ with period $\zeta$ in this case.
\end{Example}

\begin{Remark}[Relation to Coulomb branches]\label{rem:coulombC}
  Consider the following quiver gauge theory of type $\mathrm{C}_{\bw_1}$:
  \begin{equation}\label{eq:h:24}
    \begin{tikzpicture}[baseline=(current bounding box.center),decoration={brace,amplitude=7},scale=.5,
    dot/.style = {draw, circle, thick, minimum size=#1,
              inner sep=0pt, outer sep=0pt},
    dot/.default = 6pt  
                    ] 
    \node[dot=20pt] at (0,0) (v1) {$1$};
    \node[dot=20pt] at (4,0) (v2) {$2$};
    \draw[-] (v1.east) -- (v2.west);
    \draw[-] (v2.east) -- (6,0);
    \draw[-,very thick,dotted] (6.5,0) -- (8.5,0);
    \node[dot=20pt] at (11,0) (v3) {$\scriptstyle\bw_1\!-\!1$};
    \draw[-] (9,0) -- (v3.west);
    \node[dot=20pt] at (15,0) (v4) {$\bw_1$};
    \draw[double,double distance=.3ex,->] (v4.west) -- (v3.east);
    \node[draw,regular polygon,regular polygon sides=4,thick,
    inner sep=0pt, outer sep=0pt,minimum size=24pt]
    at (15,-3) (w4) {$\ell$};
    \draw[-] (v4.south) -- (w4.north);
    \end{tikzpicture}
  \end{equation}
  By \cite{Henrique}, the Coulomb branch of this theory is a symmetric bow variety. By
  the choice of the above dimension vectors, it is nothing but the
  categorical quotient $\bM^\sigma\tslash G^\sigma$ in this case.
  Here
  \begin{equation*}
    G^\sigma = \prod_{a=1}^{\ell/2-1} \GL(V_i)\times\grpO(V_{\ell/2})
    \quad\text{if $\ell$ is even},\qquad
    \prod_{a=1}^{(\ell-1)/2} \GL(V_i)
    \quad\text{if $\ell$ is odd}.
  \end{equation*}
  Indeed, the bow diagram is
  \begin{equation*}
    \begin{tikzpicture}[baseline=0pt,decoration={brace,amplitude=7},scale=.8]
    \draw[-,thick] (0,0) -- (2.5,0);
    \node[label=above:$0$] at (0,0) {};
    \node at (0.5,0) (g1) {$\vphantom{j^X}\boldsymbol\times$};
    \node[label=above:$1$] at (1,0) {};
    \node at (1.5,0) (g2) {$\vphantom{j^X}\boldsymbol\times$};
    \node[label=above:$2$] at (2,0) {};
    \draw[-,thick,dotted] (2.5,0) -- (4.5,0);
    \draw[-,thick] (4.5,0) -- (8.5,0);
    \node[label=above:$\bw_1\!-\!1$] at (4.8,0) {};
    \node at (5.5,0) (g3) {$\vphantom{j^X}\boldsymbol\times$};
    \node[label=above:$\bw_1$] at (6,0) {};
    \node at (6.5,0) (g4) {$\vphantom{j^X}\boldsymbol\medcirc$};
    \node[label=above:$\bw_1$] at (7,0) {};
    \node at (7.5,0) (g5) {$\vphantom{j^X}\boldsymbol\medcirc$};
    \node[label=above:$\bw_1$] at (8,0) {};
    \draw[-,thick,dotted] (8.5,0) -- (9.5,0);
    \draw[dashed,thin] (9.5,-.5) -- (9.5,.5);
    \end{tikzpicture}
  \end{equation*}
  where we only draw the left side of the symmetry axis. The triangle
  parts give $T^*\GL(\bw_1)$, hence it means that we only consider the
  two-way part, and do not take the reduction by $\GL(\bw_1)$. It is
  nothing but \eqref{eq:31}, and the involution coincides with one
  used to define $\bM^\sigma$.

  We can put deformation parameters $\zeta_\CC$ for $\GL$-factors, but the number of parameters
  obtained in this way is at most $\lfloor (\ell-1)/2 \rfloor$. On the other hand,
  the number of parameters given by the flavor symmetry \cite[3(viii)]{2016arXiv160103586B} is
  $\ell-1$. This number of parameters is achieved in the definition of $\fM_\zeta^\sigma$ above.
  Indeed, we will show that $\fM_\zeta^\sigma$ is the deformed Coulomb branch in a separate publication.

  Being the deformed Coulomb branch of a quiver gauge theory of type $\mathrm{C}_{\bw_1}$,
  it is Beilinson-Drinfeld affine Grassmannian slice of type $\mathrm{C}_{\bw_1}$ by
  \cite[\S3(v)]{2016arXiv160403625B}. Its cohomology is related to a representation
  of the Langlands dual Lie algebra, i.e., $\mathfrak{so}(2\bw_1+1)$ by geometric Satake correspondence.
  Since we have $\ell$-dimensional framing at the rightmost vertex, the representation in question
  is the tensor product of $\ell$-copies of the spin representation.
  We expect that the twisted Yangian and $\mathfrak{so}(2\bw_1+1)$ commute,
  like in orthogonal skew Howe duality 
  \cite{2020arXiv200511299W,2022arXiv220809773A,2024arXiv240508126B}.
  This is because $\mathfrak{so}(2\bw_1+1)$ should be defined on
  $\boldsymbol\times$ parts in the bow diagram as in \cite{2018arXiv181004293N}
  while $\sX^{\mathrm{tw}}$ is defined on $\boldsymbol\medcirc$ parts.
\end{Remark}

This remark is a motivation to study this example of $\sigma$-quiver varieties. Indeed, the computation in
the next subsection was done some years ago.

\begin{Remark}\label{rem:SdualC}
  The categorical quotient for \eqref{eq:31} is \eqref{eq:h:23} with $+$, $-$ swapped
  when $\ell$ is even,
  or \eqref{eq:h:22} with $\Wedge^2$ replaced by $\Sym^2$ when $\ell$ is odd.
  Since the $S$-duality
  is expected to be the \emph{duality}, namely the $S$-dual of the $S$-dual is the
  original space, we expect the $S$-dual of (\ref{eq:h:23}, \ref{eq:h:22}) as
  the rightmost part of \eqref{eq:h:24}. However, this is a delicate point,
  as there is no consensus about the definition of Higgs branch of a \emph{non-symmetric}
  quiver gauge theory, as in \eqref{eq:h:24}. See \cite{Examples_of_S-dual}.
\end{Remark}

\subsection{Torus action}\label{subsec:torus_action}

We have an induced action of $G_\bw^\sigma = \GL(W_1)$ on $\fM_\zeta^\sigma$
by \cref{rem:framing-symmetry}.
Note that $g\in\GL(W_1)$ acts by $\lsp{t}g^{-1}$ on $W_{\ell-1}$ via the isomorphism
$W_{\ell-1}\cong W_1^*$.
\begin{NB}
  $C_{1/2}$, $D_{1/2}$ are changed to $C_{1/2} g_0^{-1}$, $g_0
  D_{1/2}$, and $C_{\ell-1/2}$, $D_{\ell-1/2}$ to
  $g_\ell C_{\ell-1/2}$, $D_{\ell-1/2} g_{\ell}^{-1}$.
  \begin{equation*}
    \lsp{t}C_{\ell-\frac12} \lsp{t}g_\ell = \lsp{t}C_{\ell-\frac12} g_0^{-1},
    \quad\text{etc}.
  \end{equation*}
  Therefore $g_\ell = \lsp{t} g_0^{-1}$.
\end{NB}%
It preserves the symplectic form, and the moment map is given by
\begin{equation*}
  \mu([C,D]) = - D_{\frac12} C_{\frac12} + \zeta_{\frac12,\CC}.
\end{equation*}
\begin{NB}
Here we introduce a shift $\zeta_{\frac12,\CC}$ so that it gives a
deformation when we glue this two way part to other parts and take
hamiltonian reduction to introduce a symmetric bow variety.
\end{NB}%
The introduction of the shift $\zeta_{\frac12,\CC}$ is a matter of the convention to match
with other parts of moment maps.

Let us fix a decomposition $W_1 = W_1(1)\oplus\dots \oplus W_1(\bw_1)$
into $1$-dimension subspaces. We have the corresponding decomposition
$W_{\ell-1} = W_{\ell-1}(-1)\oplus\dots\oplus W_{\ell-1}(-\bw_1)$ so that
$W_{\ell-1}(k) \cong W_1(-k)^*$ under $W_{\ell-1} \cong W_1^*$.
We have the maximal torus $A_{\bw_1}$ of $\GL(W_1)$ preserving the
above decomposition.

Let us consider a fixed point $x$ with respect to the
$A_{\bw_1}$-action. Analysis of fixed points in $\fM_\zeta$ is
explained in \cref{subsec:torus-fixed-points}.
If $W_1(k)$ is the eigenspace for an eigenvalue $t_k$, $W_{\ell-1}(-k)$ is
for $t_k^{-1}$. In particular, all eigenvalues in $W_0\oplus W_{\ell-1}$
are distinct, as $t_k\neq \pm 1$ for a generic element in $A_{\bw_1}$.
Therefore, at a fixed point $x = [(C,D)]$, vector spaces
$V_i$ ($i=1,\dots,\ell-1)$ decomposes as
\begin{equation*}
  V_i = V_i(1)\oplus \dots \oplus V_i(\bw_1)
  \oplus V_i(-1)\oplus \dots \oplus V_i(-\bw_1)
\end{equation*}
so that $C$, $D$ sends $V_i(k)$ to $V_{i\pm 1}(k)$. Let us denote
components for $k$ by $C(k)$, $D(k)$.
Then $x_k\defeq [(C(k), D(k))]$ gives a point of the quiver variety
with dimension vectors $(\dim V_1(k),\dots,\dim V_{\ell-1}(k))$,
$(\dim W_1(k),0,\dots,0,\dim W_{\ell-1}(k))$.
For example, if $k > 0$, the data look like
\begin{equation*}
  \begin{tikzcd}
    W_1(k) \arrow[r,yshift=.5ex,"C_{\frac12}(k)"] & V_1(k)
    \arrow[l,yshift=-.5ex,"D_{\frac12}(k)"]
    \arrow[r,yshift=.5ex,"C_{\frac32}(k)"] & \phantom{V} \arrow[r,no
    head,very thick,dotted] \arrow[l,yshift=-.5ex,"D_{\frac32}(k)"] &
    \phantom{V} \arrow[r,yshift=.5ex,"C_{\ell-\frac32}(k)"] &
    V_{\ell-1}(k) \arrow[l,yshift=-.5ex,"D_{\ell-\frac32}(k)"]
    \arrow[r,yshift=.5ex,"C_{\ell-\frac12}(k)=0"] & 0,
    \arrow[l,yshift=-.5ex,"D_{\ell-\frac12}(k)=0"]
  \end{tikzcd}
\end{equation*}
as $W_{\ell-1}(k) = 0$. This is a point in the quiver variety whose
framing dimension vector is $(1,0,\dots,0)$. It is well-known 
that the corresponding quiver variety is nonempty if and only if
\begin{equation}\label{eq:24}
  \dim V_1(k) = \dots = \dim V_i(k) = 1,\quad
  \dim V_{i+1}(k) = \dots = \dim V_{\ell-1}(k) = 0
\end{equation}
for some $i = 0,1,\dots,\ell-1$. Moreover it is a single point if it
is nonempty.
See e.g., \cite[Lemma~5.1]{MR1285530} for a geometric proof. 
In terms of a $\U(2\bw_1)$-instanton moduli space of $\fM_\zeta$, a
fixed point corresponds to a direct sum of line bundles
($\U(1)$-instantons) compatible with the framing.
Alternatively, 
we note that this set of quiver varieties realizes the vector representation $\CC^\ell$ of $\algsl_\ell$
so that the quiver variety for \eqref{eq:24} gives a weight space. Since all weights
are in Weyl group orbits of the highest weight, the corresponding quiver varieties
are points.
  
If the fixed point $x$ is in $\fM_\zeta^\sigma$, the decomposition is
preserved under $\sigma$. More precisely $x$ decomposes as $x_k$
($k= -\bw_1, \dots, -1, 1, \dots, \bw_1$), and the summand $x_k$ is
mapped to $x_{-k}$ under $\sigma$. Therefore if the summand for
$k > 0$ satisfies \eqref{eq:24} for $i$, the corresponding summand for
$-k$ satisfies
\begin{equation}\label{eq:25}
  \begin{multlined}
    \dim V_1(-k) = \dots = \dim V_i(-k) = 0,\\
    \dim V_{i+1}(-k) = \dots = \dim V_{\ell-1}(-k) = 1.
  \end{multlined}
\end{equation}
If these are satisfied, $x_k$ is sent to $x_{-k}$ indeed. In fact,
taking transpose, the diagram automorphism and the reflection functor
are defined more generally for quiver varieties with generic
parameter. Since relevant quiver varieties are single points, we only
need to check that dimension vectors above are exchanged by
$\sigma$. This check is trivial.
\begin{NB}
  Note that we already fix isomorphisms $V_0\to V_\ell^*$,
  $V_\ell\to V_0^*$. Given a line bundle $L$. Take the direct sum of
  its dual $E\defeq L\oplus L^*$. Then we have
  $\varPhi\colon E\to E^* = L^*\oplus L$ by
  $\varPhi = \begin{pmatrix} 0 & \id_{L^*} \\ -\id_L & 0
  \end{pmatrix}$.  Then
  $\varPhi^* = \begin{pmatrix} 0 & -\id_{L^*} \\ \id_L & 0
  \end{pmatrix} = - \varPhi$. Hence $E$ is a symplectic vector bundle.
\end{NB}%
  
In this way, a fixed point assigns an integer $i=0,\dots,\ell-1$ for
each $k=1,\dots, \bw_1$. Therefore we have $\ell^{\bw_1}$ fixed points.

Following \cite{MR1285530}, we assign row-increasing Young tableau $T$
to a fixed point. 
An each row corresponds to the summand $x_k$.
In order to avoid confusion, let us call the row for $x_k$ \emph{the
  $k$-th row}. Then we put $(-\bw_1)$-th row at the top, $(-\bw_1+1)$-th
row one below, and so on. After putting $(-1)$-st row, we put the $1$-st
row next below, $2$-nd row next, and so on. So the numbering of rows is
increasing if we go below.
The length of a row is equal to the numbering of the framed vector
space, which is $1$-dimensional.
Thus the $k$-th row with $k > 0$ (resp.\ $k < 0$) has length
$1$ (resp.\ $\ell-1$).
Each node of $T$ is numbered by integers between $1$ and $\ell$.
In each row, the sequence of numbers is strictly increasing.
A number $i$ appears in the $k$-th row if and only if
\begin{equation*}
  \begin{cases}
    \dim V_{i-1}(k) - \dim V_i(k) = 0 &
    \text{if $i \le $ the length of the row},\\
    \dim V_{i-1}(k) - \dim V_i(k) = 1 & \text{if $i > $ the length of the row}.
  \end{cases}
\end{equation*}
See \cite[p.120]{MR1285530}. In our case, for $k > 0$, the unique node
is numbered by $(i+1)$ when $i$ is given by \eqref{eq:24} for $V(k)$.
On the other hand, the $(-k)$-th row have nodes $1$, \dots,
$\ell$, excluding $(i+1)$ when
$i$ is given by \eqref{eq:25} for $V(-k)$.
The sum $\sum_k \dim V_i(k)$ is $\dim V_i$, hence is fixed. This
condition appears as a constraint
\begin{equation}\label{eq:26}
  \text{the cardinarity of the nodes with $i$ is $\bw_1$.}
\end{equation}
Row increasing Young tableaux $T$ with \eqref{eq:26} is bijective to
torus fixed points in $\fM_\zeta$ (\cite[Lem.~5.9]{MR1285530}).

Since $i$ for \eqref{eq:24} and \eqref{eq:25} are the same when a
fixed point is in $\fM_\zeta^\sigma$, we have an additional constraint
\begin{equation}\label{eq:27}
  \begin{minipage}{.8\linewidth}
    integers appearing the $k$-th row are those which do \emph{not}
  appear in $(-k)$-th row.
  \end{minipage}
\end{equation}
Here is an example:
\begin{equation*}
  \begin{ytableau}
    \none[-2] & \none & 2 & 3 & 4 & 5 \\
    \none[-1] & \none & 1 & 2 & 4 & 5 \\
    \none[1] & \none & 3 \\ \none[2] & \none & 1
  \end{ytableau}
\end{equation*}
with $\bw_1 = 2$, $\ell = 5$. The numbering of rows is written in the
left of the first row.

It is clear that these tableaux are determined by entries of $k$-th
rows with $k > 0$. Therefore we have $\ell^{\bw_1}$ fixed points.
We will record entries for $k < 0$ in order to express the formula of
Betti numbers by using results in \cite{MR1285530}.

Let $T(k,a)$ be the number assigned to the node at the $k$-th row,
the $a$-th column.
For a given pair $k$, $a$, we consider the following conditions on $l\neq k$
(\cite[Def.~5.10]{MR1285530}):
\begin{equation}\label{eq:28}
  \begin{minipage}{.8\linewidth}
    \begin{itemize}
    \item $T(l,a) < T(k,a)$ when $l < k$.
    \item $T(l,a+1) < T(k,a)$ when $l > k$, where we understand
      this inequality is not satisfied if there is no node at $(l,a+1)$.
    \item $T(k,a)$ does not appear in the $l$-th row in either cases.
    \end{itemize}
  \end{minipage}
\end{equation}
For the above tableau, the node $T(1,1)=3$ satisfies this condition
for $l=-1$, but not other $l=-2$, $2$.

In practice, if $k > 0 > l$, this condition is satisfied if and only
if $T(k,1)$ does not appear in the $l$-th row and $T(k,1)\neq 1$.
\begin{NB}
  If $T(k,1)$ does not appear in the $l$-th row, all other numbers between $1$ and $\ell$ appear.
  In particular, $1$ appears as $T(k,1) \neq 1$. Then $T(l,1) = 1 < T(k,1)$.
\end{NB}%

Let us define the \emph{charge} $l(T)$ by
\begin{equation*}
  l(T) = \# \{ (k,l,a) \mid \text{\eqref{eq:28} is satisfied}\}.
\end{equation*}
The Poincar\'e polynomial of $\fM_\zeta$ is equal to
$\sum_T t^{2 l(T)}$ by \cite[Th.~5.15]{MR1285530}. Here the sum is
over row increasing Young tableaux $T$ with \eqref{eq:26}.

For a tableau $T$ with \eqref{eq:27} we define $l_{\grpSp}(T)$ by
\begin{equation}\label{eq:34}
  \begin{aligned}
    l_{\grpSp}(T) = \# \{ &(k,-k,a) \mid \text{\eqref{eq:28} is satisfied}\}
    \\
    &+
    \frac12 \# \{ (k,l,a) \mid \text{$k\neq -l$ and \eqref{eq:28} is satisfied}\}.
  \end{aligned}
\end{equation}
Note that \eqref{eq:28} is satisfied for $(k,l,a)$ if and only if
it is so for $(-l,-k,a)$. Therefore this is an integer.

\begin{Theorem}\label{thm:fixed-point-set}
  \textup{(1)} The $\sigma$-quiver variety $\fM_\zeta^\sigma$ has dimension $\bw_1(\bw_1+1)$.

  \textup{(2)} The Poincar\'e polynomial of $\fM_\zeta^\sigma$ is
  \begin{equation*}
    \sum_{T: \eqref{eq:27}} t^{2l_{\grpSp}(T)}.
  \end{equation*}
  In particular, $\fM_\zeta^\sigma$ is connected.
\end{Theorem}

\begin{proof}
%
  By the hyper-K\"ahler rotation, the dimension and homology groups
  are independent of $\zeta$. Therefore we may assume $\zeta$ is of a
  form $(\zeta_{\RR},0)$.

  Let us first introduce a $\CC^\times$-action on $\fM_\zeta$ given by
  multiplication of $t^2$ on $C_{i+\frac12}$ for $i\neq 0$, $\ell-1$ and
  of $t$ on $C_{\frac12}$, $D_{\frac12}$, $C_{\ell-\frac12}$, $D_{\ell-\frac12}$
  for $t\in\CC^\times$. 
  This is the choice of the $\CC^\times$-action in \cite[\S5]{MR1285530}
  with a modification by $t^{-1}$ on all $\GL(V_i)$ ($i\neq 0,\ell$).
  The modification does not change the action on $\fM_\zeta$ as it
  is in $G$.
  It commutes with $\sigma$ on the categorical quotient $\fM_0$.
  Since $\pi$ is a birational map onto its image, the same is true
  for $\fM_\zeta$.
  Hence we have an induced $\CC^\times$-action on
  $\fM_\zeta^\sigma$. We perturb it by an action of a maximal torus
  $A_{\bw_1}$ in $\GL(V_0)$ consisting of diagonal matrices as
  \begin{equation*}
    \rho\colon \CC^\times \ni t \mapsto
    (t^N, t^{\ve_1},\dots, t^{\ve_{\bw_1}})\in
    \CC^\times \times A_{\bw_1}
  \end{equation*}
  with $N\gg \ve_{\bw_1} > \dots > \ve_{1} > 0$. Then
  $(\fM_\zeta^\sigma)^{\rho(\CC^\times)} =
  (\fM_\zeta^\sigma)^{\CC^\times\times A_{\bw_1}}$.
  This action is the restriction of the action used in \cite[\S5]{MR1285530}.
  Moreover $\lambda(t)[(C,D)]$ stays in a compact subset of $\fM_\zeta^\sigma$
  as $t\to 0$, hence converges to a fixed point. Therefore any
  connected component of $\fM_\zeta^\sigma$ contains a fixed
  point. Thus it is enough to show that the dimension of the tangent
  space at a fixed point is $\bw_1(\bw_1+1)$.
  
  Let us consider the three term complex \cite[(3.11)]{Na-alg},
  describing the tangent space at $[(C,D)]$. (The $0$-th and second
  cohomology vanish.)
  According to the decomposition $(C,D) = \bigoplus (C(k),D(k))$, the
  complex decomposes as $\bigoplus_{k,l} \mathcal C(k,l)$
  (cf.~\cite[(6.1)]{Na-alg}). The summand $\mathcal C(k,l)$ computes
  $\Ext^1(x_k,x_l)$,
  the extension between $x_k$ and $x_l$ for framed vector bundles. 
  See \cite[Remark~3.16]{Na-Tensor}.
  The involution $\sigma$ maps
  $\mathcal C(k,l)$ to $\mathcal C(-l,-k)$. In fact, $x_k$ is mapped
  to $x_{-k}$, and the order of the extension is exchanged as $\sigma$
  involves transpose.

  Let us give the dimension of the first cohomology of
  $\mathcal C(k,l)$. When $k > 0$, let $i(k)$ be the number $i$ in
  \eqref{eq:24}. Then
  \begin{equation*}
    \dim H^1(\mathcal C(k,l)) =
    \begin{cases}
      0 & \text{if $kl > 0$ and $i(|k|) = i(|l|)$}, \\
      1 & \text{if $kl > 0$ and $i(|k|) \neq i(|l|)$}, \\
      1 & \text{if $kl < 0$ and $i(|k|) = i(|l|)$}, \\
      0 & \text{if $kl < 0$ and $i(|k|) \neq i(|l|)$}.
    \end{cases}
  \end{equation*}
  This is checked by a direct calculation for the complex.
  Alternatively, we observe that $\dim H^1(\mathcal C(k,l)) = \dim H^1(\mathcal C(l,k))$
  by the duality given by the symplectic form as in \cref{rem:symplectic}(2) below,
  hence $\dim H^1(\mathcal C(k,l))$ is half of the dimension of the quiver variety
  containing the direct sum $x_k\oplus x_l$. The dimension of the relevant quiver variety
  can be computed as $0$ or $2$ by the formula
  $\bv\cdot (2\bw - \bC\bv)$ mentioned in \cref{subsec:quiver_variety}.

  Note that $\mathcal C(k,-k)$ is mapped to itself by $\sigma$,
  Therefore we need to determine the action of $\sigma$ on
  $H^1(\mathcal C(k,-k))$. Note that
  $H^1(\mathcal C(k,-k))\oplus H^1(\mathcal C(-k,k))$ gives the
  tangent space for the special case $\bw_1 = 1$, corresponding to a
  point $x_k\oplus x_{-k}$. We know that the involution is trivial in
  the case $\bw_1 = 1$ by \cref{ex:ALE}. Therefore $\sigma =
  \id$. Thus $\dim H^1(\mathcal C(k,-k))^\sigma = 1$.
  For other cases, $\mathcal C(k,l)$ and
  $\mathcal C(-l,-k)$ are different components
  in the tangent space $\bigoplus_{k,l} H^1(\mathcal C(k,l))$. 
  Hence $\left(H^1(\mathcal C(k,l))\oplus H^1(\mathcal C(-l,-k))\right)^\sigma
  \cong H^1(\mathcal C(k,l))$.
  From these, as well as the formula above, we see that
  $\bigoplus_{k,l} H^1(\mathcal C(k,l))^\sigma$ is
  $\bw_1(\bw_1+1)$-dimensional, as claimed.
  Indeed, the total dimension $\sum H^1(\mathcal C(k,l))$ is $2\bw_1^2$,
  which is $\dim \fM_\zeta$. Among pairs $(k,l)$, we have $\dim H^1(\mathcal C(k,-k))
  = 1$ for all $k$, which gives $2\bw_1$-dimensional contributions in total.
  Other pairs come with $(k,l)$ and $(-l,-k)$ ($k\neq l$), hence contribute
  as $(2\bw_1^2 - 2\bw_1)/2 = \bw_1^2 - \bw_1$. Thus we get
  $2\bw_1 + \bw_1^2 - \bw_1 = \bw_1(\bw_1+1)$.

  Let us turn to (3). The charge $l(T)$ is the Morse index of the
  moment map associated with the $\CC^\times$-action above at the
  fixed point corresponding to $T$ (\cite[Lemma~5.11]{MR1285530}). It
  is also equal to the sum of dimensions of negative weight subspaces
  of the tangent space at the fixed point. The tangent space
  decomposes as $\bigoplus_{k,l} H^1(\mathcal C(k,l))$ as above, and
  the contribution of $\mathcal C(k,l)$ to the index is equal to the
  number of $a$ such that \eqref{eq:28} is satisfied.

  If $k+l=0$, $\sigma$ is $\id$ on $H^1(\mathcal C(k,l))$. Therefore
  the $H^1(\mathcal C(k,l))^\sigma = H^1(\mathcal C(k,l))$, and the
  contribution to the index is the same as for $l(T)$. If $k+l\neq 0$,
  the contribution is half, as $\mathcal C(k,l)$ and
  $\mathcal C(-l,-k)$ are swapped by $\sigma$. Thus we have the
  formula for $l_{\grpSp}(T)$.

  Let us note that $T$ has $l_{\grpSp}(T) = 0$ if and only if the $k$-th row
  with $k>0$ (resp.\ $k < 0$) has node $1$ (resp.\ $2$, $3$, \dots,
  $\ell$). Since $T$ is unique, $\fM_\zeta^\sigma$ is connected.
\end{proof}

\begin{Remark}\label{rem:symplectic}
  (1)
  The above proof used \cref{ex:ALE} to show that $\sigma$ is identity
  on $H^1(\mathcal C(k,-k))$. Since the second argument in \cref{ex:ALE}
  just used $\grpSp(2)\cong\SL(2)$, the formula \eqref{eq:34} remains valid
  for tableaux of more general shapes, in other words, for arbitrary
  framing dimension vectors $\bw$.

  (2) Under the symplectic structure on the tangent space $\bigoplus H^1(\mathcal C(k,l))$,
  subspaces $H^1(\mathcal C(k,l))$ and $H^1(\mathcal C(l,k))$ are dual to each other, as
  they are $\Ext^1(x_k,x_l)$ and $\Ext^1(x_l,x_k)$. 
  (Small remark: $H^1(\mathcal C(k,k)) = 0$, hence we may assume $k\neq l$, so
  $H^1(\mathcal C(k,l))$ and $H^1(\mathcal C(l,k))$ are different, possibly zero, 
  components.))
  The isomorphism
  $(H^1(\mathcal C(k,l))\oplus H^1(\mathcal C(-l,-k)))^\sigma\cong H^1(\mathcal C(k,l))$
  appeared during the proof come with the `partner',
  $(H^1(\mathcal C(l,k))\oplus H^1(\mathcal C(-k,-l)))^\sigma\cong H^1(\mathcal C(l,k))$,
  hence the $\sigma$-fixed subspace is a symplectic subspace. This checks that
  $\fM_\zeta^\sigma$ is symplectic in a naive approach.
\end{Remark}

\subsection{\texorpdfstring{$K$}{K}-matrix}\label{subsec:K-matrix}

We assume $\zeta_\CC = 0$ and $\sqrt{-1}\zeta_{i,\RR} > 0$ for all $i$. We drop $\zeta$ from the notation
in quiver varieties.

Let us take $\bw_1 = 2$ and two dimensional torus $A = A_{\bw_1}$ in $\GL(W_1)$. Let us denote the
coordinates of the Lie algebra of the torus by $u_1$, $u_2$.

\subsubsection{Wall \texorpdfstring{$u_1=u_2$}{u_1=u_2}}\label{subsub:A}

We first consider the $R$-matrix, moving from the chamber $\{ u_1 > u_2\}$ to
$\{ u_1 < u_2\}$.

We may first take the fixed point set with respect to the diagonal $\Delta\CC^\times$ given by $u_1 = u_2$.
In this case, $(\fM^\sigma)^{\Delta\CC^\times}$ is a usual quiver variety of type $\mathrm{A}_{\ell-1}$ with
the framing dimension vector $(2,0,\dots,0)$. The dimension vector $\bv$ is of the form
$(2,2,\dots,2,1,\dots,1,0,\dots, 0)$, i.e., it is empty otherwise. Moreover, it is the cotangent bundle of the $\ell$-step flags in $\CC^2$, with
dimensions of subspaces specified in this manner. Therefore 
$T^*\proj^1$ if $1$ actually appears in the middle, and a point otherwise.
In order to parametrize fixed points $(\fM^\sigma)^A$, we can concentrate on the lower part of
the Young tableau:
\begin{equation}\label{eq:h:16}
\begin{ytableau}
  \none[1] & \none & s \\ \none[2] & \none & t
\end{ytableau}
\end{equation}
By the above discussion, the corresponding fixed point in contained in a single point component of
$(\fM^\sigma)^{\Delta\CC^\times}$ if and only if $s = t$.
It is contained in a component of the form $T^*\proj^1$ if and only if $s\neq t$. In that case,
there is another fixed point with the same tableau, but with $s$ and $t$ exchanged.

The normal bundle of $(\fM^\sigma)^A$ in $(\fM^\sigma)^{\Delta\CC^\times}$ is
$H^1(\mathcal C(1,2))\oplus H^1(\mathcal C(2,1)) \cong H^1(\mathcal C(-2,-1))\oplus H^1(\mathcal C(-1,-2))$,
which is $2$-dimensional if $s\neq t$, and $0$-dimensional if $s=t$.
We choose the polarization in the fiber direction of $T^*\proj^1$ for $s\neq t$. Concretely,
it is $H^1(\mathcal C(1,2))$ if $s < t$, and $H^1(\mathcal C(2,1))$ if $s > t$.
The normal bundle of $(\fM^\sigma)^{\Delta\CC^\times}$ in $\fM^\sigma$ is
$H^1(\mathcal C(1,-1))\oplus H^1(\mathcal C(-1,1)) \oplus H^1(\mathcal C(2,-2))\oplus H^1(\mathcal C(-2,2))
\oplus H^1(\mathcal C(1,-2))\oplus H^1(\mathcal C(-2,1))$. Summands $H^1(\mathcal C(1,-2))$,
$H^1(\mathcal C(-2,1))$ vanish if $s\neq t$, and are $1$-dimensional if $s=t$.
Other four summands are always $1$-dimensional.
\begin{NB}
  We need to fix the polarization.
\end{NB}%
We do not fix the polarization of this normal bundle at this stage. We assume that it is
an induced one. Therefore the choice must be consistent for the points corresponding
to $(s,t)$ and $(t,s)$.

The computation of the $R$-matrix was done in \cite[Ex.~4.1.2]{MR3951025}. It gives Yang's $R$-matrix
$R(u) = 1 - Pu^{-1}$ with $P = \sum_{ij} e_{ij}\otimes e_{ji}$ up to a normalization factor.
For example, if $\ell=2$, and under the base
\ytableausetup{smalltableaux}
\begin{ytableau}
    1 \\ 1
\end{ytableau},
\begin{ytableau}
    1 \\ 2
\end{ytableau},
\begin{ytableau}
    2 \\ 1
\end{ytableau},
\begin{ytableau}
    2 \\ 2
\end{ytableau},
\ytableausetup{nosmalltableaux}
the $R$-matrix is given by
\begin{equation*}
  R(u_1-u_2) = \begin{pmatrix} 1 & 0 & 0 & 0 \\ 0 & \frac{u_1-u_2}{u_1-u_2 + \hbar} & \frac{\hbar}{u_1-u_2 + \hbar} & 0 \\ 0 & \frac{\hbar}{u_1-u_2 + \hbar} &  \frac{u_1-u_2}{u_1-u_2 + \hbar} & 0 \\ 0 & 0 & 0 & 1\end{pmatrix}.
\end{equation*}
The sign of $\hbar$ is opposite, as our action scales
the cotangent fiber by $\bq^2 = \exp(\hbar)$ instead of $\exp(-\hbar)$.

\begin{NB}
In \cite[(1.12)]{MR2355506}, one uses $R(u) = 1 - Pu^{-1}$. Then
\begin{equation*}
 R(u_1-u_2) = \begin{pmatrix} 1 - \frac1{u_1-u_2} & 0 & 0 & 0 \\ 0 & 1 & -\frac{1}{u_1-u_2} & 0 \\ 0 & -\frac{1}{u_1-u_2} &  1 & 0 \\ 0 & 0 & 0 & 1 - \frac1{u_1-u_2}\end{pmatrix}. 
\end{equation*}
Multiplying $(1 - \frac1{u_1-u_2})^{-1} = \frac{u_1 - u_2}{u_1 - u_2 - 1}$, we get the above after
substitution $\hbar = -1$.
\end{NB}%

This computation is compatible with the identification of the Maulik-Okounkov Yangian with $\sY(\gl_\ell)$ in this case,
hence the $R$-matrix is the intertwiner for $\CC^\ell\otimes\CC^\ell$ with the coproduct $\Delta$ and the opposite $\Delta^{\mathrm{op}}$.

\begin{NB}
  Fixed points. we add $-2$, $-1$-st rows:
  \ytableausetup{smalltableaux}
$q_1 = \begin{ytableau}
    2 \\ 2 \\ 1 \\ 1
\end{ytableau}$,
$q_2 = \begin{ytableau}
    1 \\ 2 \\ 1 \\ 2
\end{ytableau}$,
$q_3 = \begin{ytableau}
    2 \\ 1 \\ 2 \\ 1
\end{ytableau}$,
$q_4 = \begin{ytableau}
    1 \\ 1  \\ 2 \\ 2
\end{ytableau}$.
\ytableausetup{nosmalltableaux}
$H^1(\mathcal C(1,-1))\oplus H^1(\mathcal C(-1,1))$ and
$H^1(\mathcal C(2,-2))\oplus H^1(\mathcal C(-2,2))$ are always
$2$-dimensional.
\begin{equation*}
\begin{split}
 &\dim \left(
  H^1(\mathcal C(1,2))\oplus H^1(\mathcal C(2,1))
  \xrightarrow[\sigma]{\cong} H^1(\mathcal C(-2,-1))\oplus H^1(\mathcal C(-1,-2))
  \right) = \begin{cases}
  2 & \text{for $q_2$, $q_3$} \\ 0 & \text{for $q_1$, $q_4$}  
  \end{cases}
  ,\\
&\dim \left(
  H^1(\mathcal C(1,-2))\oplus H^1(\mathcal C(-2,1))
  \xrightarrow[\sigma]{\cong} H^1(\mathcal C(2,-1))\oplus H^1(\mathcal C(-1,2))
  \right) = \begin{cases}
  0 & \text{for $q_2$, $q_3$} \\ 2 & \text{for $q_1$, $q_4$}
  \end{cases}.
\end{split}
\end{equation*}  
\end{NB}%

\begin{NB}
  Polarization : $\dim \fM^\sigma = 6$ as $\bw_1 = 2$.
  We chose the polarization at $q_2$, $q_3$ \emph{partially}:
  fibers direction in $H^1(\mathcal C(1,2))\oplus H^1(\mathcal C(2,1))
  \xrightarrow[\sigma]{\cong} H^1(\mathcal C(-2,-1))\oplus H^1(\mathcal C(-1,-2))$ ($2$-dim).
  More precisely, $H^1(\mathcal C(1,2)) \cong H^1(\mathcal C(-2,-1))$ on $q_2$, and
  $H^1(\mathcal C(2,1))\cong H^1(\mathcal C(-1,-2))$ on $q_3$.
  Both are the $C$-direction. And also simultaneously the cotangent fiber direction.
  Not specified in
  $H^1(\mathcal C(1,-1))\oplus H^1(\mathcal C(-1,1))$ ($2$-dim),
  $H^1(\mathcal C(2,-2))\oplus H^1(\mathcal C(-2,2))$ ($2$-dim),
  $H^1(\mathcal C(1,-2))\oplus H^1(\mathcal C(-2,1))
  \xrightarrow[\sigma]{\cong} H^1(\mathcal C(2,-1))\oplus H^1(\mathcal C(-1,2))$ ($0$-dim).
  A change of the polarization in these direction gives conjugation by
  $\operatorname{diag}(1,1,1,-1)$, $\operatorname{diag}(-1,1,1,1)$, but it does not change the $R$-matrix.

  Since $q_2$ and $q_3$ are contained in the same component of
  $(\fM^\sigma)^{\Delta\CC^\times}$, the polarization in the remaining directions
  must be the same for $q_2$ and $q_3$: it is the polarization of
  $(\fM^\sigma)^{\Delta\CC^\times}$ in $\fM^\sigma$.
  On $(\fM^\sigma)^{\Delta\CC^\times}$, the decomposition
  $H^1(\mathcal C(1,-1))\oplus H^1(\mathcal C(2,-2))$ does not make sense, as
  the decomposition $x_1\oplus x_2$ is not well-defined. On the other hand,
  the decomposition $(x_1\oplus x_2) \oplus (x_{-1}\oplus x_{-2})$ is well-defined.
  Hence we can distinguish $H^1(\mathcal C(1,-1))\oplus H^1(\mathcal C(2,-2))$ 
  and $H^1(\mathcal C(-1,1))\oplus H^1(\mathcal C(-2,2))$. The polarization must
  be the either choice.
\end{NB}%

\subsubsection{Wall \texorpdfstring{$u_1= -u_2$}{u_1= - u_2}}\label{subsub:B}
Next consider the $R$-matrix for the chamber $\mathfrak C = \{ u_1 + u_2 > 0\}$ to $-\mathfrak C = \{ u_1 + u_2 < 0\}$.
We first take the fixed point set with respect to the \emph{anti-diagonal}
$\Delta'\CC^\times$ given by $u_1 + u_2 = 0$. In this case, $(\fM^\sigma)^{\Delta'\CC^\times}$ is a usual quiver variety of type $\mathrm{A}_{\ell-1}$ with
the framing dimension vector $(1,0,\dots,0,1)$. It realizes the representation $\CC^\ell\otimes(\CC^\ell)^*$ of
$\algsl_\ell$ from the corresponding tensor product quiver variety.
When the dimension vector $\bv$ is $(1,\dots, 1)$, it is the ALE space of type $\mathrm{A}_{\ell-1}$, and
either point or empty otherwise.

We focus on the $(-2)$-nd and $1$-st rows of the Young tableaux:
\begin{equation}\label{eq:h:19}
  \begin{ytableau}
    \none[-2] & \none & t_1 & t_2 & t_3 & t_4 \\
    \none[1] & \none & s 
  \end{ytableau}
  \qquad (\ell = 5)
\end{equation}
If $s$ does not appear in the $(-2)$-nd row, the dimension vector is $(1,\dots, 1)$. Hence
it belongs to a component with $\bv = (1,\dots, 1)$. There are $\ell$ such tableaux, given by
the choice $s = 1,2,\dots,\ell$. Other tableaux are contained in a single point component of
$(\fM^\sigma)^{\Delta'\CC^\times}$.
Therefore, we can concentrate on the tableaux in the first case.

Let us name the fixed point corresponding to the tableaux with $s = i$ by $p_i$ ($i=1,\dots,\ell$).
It is easy to compute weights of tangent spaces at $p_i$:
\begin{gather*}
  p_1 : (-u + \frac{\ell\hbar}2, u + \frac{(2-\ell)\hbar}2),\quad
  p_2 : (-u + \frac{(\ell-2)\hbar}2, u + \frac{(4-\ell)\hbar}2),\quad \dots,
  \\
  \dots,\quad
  p_\ell : (-u + \frac{(2-\ell)\hbar}2, u + \frac{\ell\hbar}2),
\end{gather*}
where $u = u_1 + u_2$.
There are $(\ell-1)$ $\proj^1$'s in the exceptional divisor, $\Sigma_1$, \dots, $\Sigma_{\ell-1}$:
$\Sigma_i$ contains $p_i$ and $p_{i+1}$. There is an additional noncompact divisor $F_\infty$ passing
through $p_\ell$. It is the attracting direction, and the weight of its normal bundle is
$-u + \frac{(2-\ell)\hbar}2$. 

The normal bundle of $(\fM^\sigma)^A$ in $(\fM^\sigma)^{\Delta'\CC^\times}$ is
$H^1(\mathcal C(1,-2))\oplus H^1(\mathcal C(-2,1)) \cong H^1(\mathcal C(2,-1))\oplus H^1(\mathcal C(-1,2))$,
which is $2$-dimensional if $s$ does not appear in the $(-2)$-nd row, and $0$-dimensional otherwise.
\begin{NB}
  Their weights are $-u_2-u_1 = -u$ and $u_1 + u_2 = u$ respectively.
\end{NB}%
We choose the polarization in the repelling direction.
\begin{NB}
  Namely $H^1(\mathcal C(1,-2))$.
\end{NB}%
%
Then
\begin{gather*}
  \Stab_{\mathfrak C}(p_\ell) = [F_\infty],\quad \Stab_{\mathfrak C}(p_{\ell-1}) = [F_\infty] + [\Sigma_{\ell-1}], \dots,
  \\
  \Stab_{\mathfrak C}(p_1) = [F_\infty] + [\Sigma_{\ell-1}] + \dots + [\Sigma_1].
\end{gather*}
We have the similar formula for the opposite chamber $-\mathfrak C = \{ u < 0\}$.

The normal bundle of $(\fM^\sigma)^{\Delta'\CC^\times}$ in $\fM^\sigma$ is
$H^1(\mathcal C(1,-1))\oplus H^1(\mathcal C(-1,1)) \oplus
H^1(\mathcal C(2,-2))\oplus H^2(\mathcal C(-2,2)) \oplus
H^1(\mathcal C(1,2))\oplus H^1(\mathcal C(2,1))$.
We do not fix the polarization at this stage. We assume that it is an induced one.
Therefore the choice must be consistent for points $p_1$, \dots, $p_\ell$.

The $(\ell\times\ell)$-entries of the $R$-matrix corresponding to has the following form:
\begin{equation}\label{eq:h:17}
  \frac{1}{u + \frac{\ell \hbar}2}
  \begin{pmatrix}
    u + \frac{(\ell-2)\hbar}2 & -\hbar & -\hbar & \dots \\
    -\hbar & u + \frac{(\ell-2)\hbar}2 & -\hbar & \dots \\
    -\hbar & -\hbar & \ddots & \\
    \vdots & \vdots & & \ddots 
  \end{pmatrix}
  = \id - \frac{\hbar}{u+\frac{\ell \hbar}2}
  \begin{pmatrix}
    1 & 1 & 1 & \dots \\
    1 & 1 & 1 & \dots \\
    1 & 1 & \ddots & \\
    \vdots & \vdots & & \ddots 
  \end{pmatrix}.
\end{equation}
It is different from the $R(u_1-u_2)$ above in $\ell=2$ by sign in off-diagonal entries.
It is because of the choice of the polarization is different.

When we write all the $(\ell^2\times\ell^2)$ entries of the $R$-matrix, rows and columns
are indexed by fixed points $(\fM_\zeta^\sigma)^A$.
Concretely we return back
to Young tableaux \eqref{eq:h:16}, though
\eqref{eq:h:19} is natural in view of geometry.
The above $(\ell\times\ell)$-entries are
those for $s=t=1,\dots, \ell$.
For example, if $\ell=2$, we have
\begin{equation*}
R^{\bullet\sigma}(u_1+u_2) = \begin{pmatrix} \frac{u_1+u_2}{u_1+u_2+\hbar} & 0 & 0 & \frac{-\hbar}{u_1+u_2+\hbar} \\ 0 & 1 & 0 & 0 \\ 0 & 0 & 1 & 0 \\ \frac{-\hbar}{u_1+u_2+\hbar} & 0 & 0 & \frac{u_1+u_2}{u_1+u_2+\hbar} \end{pmatrix}
\end{equation*}

Up to the shift $u\mapsto u + {\ell\hbar}/2$, this coincides with
$R^t(-u_1-u_2) = 1 + Q(u_1+u_2)^{-1}$ with $Q = \sum_{ij} e_{ij}\otimes e_{ij}$
with $\hbar=-1$,
which appears in the reflection equation for the twisted Yangian
in \cite[\S2.2]{MR2355506}.
Note that $R^t(u-v)^{-1}$ is the $R$-matrix for the vector representation
of $\mathrm{Y}(\gl_N)$ and its left dual representation, say by the first displayed
equation in the proof of \cite[Prop.~1.9.8]{MR2355506}.
The shift $\ell\hbar/2$ of the spectral parameter is explained by the
difference of our convention in \cref{rem:shift_of_R}.

\begin{NB}
  We chose the polarization at $q_1$ ($=p_1$), $q_4$ ($=p_2$) \emph{partially}:
  $C$-direction in
  $H^1(\mathcal C(1,-2))\oplus H^1(\mathcal C(-2,1))
  \xrightarrow[\sigma]{\cong} H^1(\mathcal C(2,-1))\oplus H^1(\mathcal C(-1,2))$ ($2$-dim).
  More precisely, $H^1(\mathcal C(1,-2)) \cong H^1(\mathcal C(2,-1))$ on both $q_1$ and $q_4$.
  Both are the $C$-direction. The cotangent fiber directions are
  $H^1(\mathcal C(1,-2))$ for $q_1$ and $H^1(\mathcal C(-2,1))$ for $q_4$.
  Not specified in
  $H^1(\mathcal C(1,2))\oplus H^1(\mathcal C(2,1))
  \xrightarrow[\sigma]{\cong} H^1(\mathcal C(-2,-1))\oplus H^1(\mathcal C(-2,-1))$ ($0$-dim), 
  $H^1(\mathcal C(1,-1))\oplus H^1(\mathcal C(-1,1))$ ($2$-dim),
  $H^1(\mathcal C(2,-2))\oplus H^1(\mathcal C(-2,2))$ ($2$-dim).
  Like in the wall $\{ u_1 = u_2\}$, the change of the polarization in these direction gives conjugation by
  $\operatorname{diag}(1,1,-1,1)$, $\operatorname{diag}(1,-1,1,1)$, but it does not change the $R$-matrix.

  Since $q_1$ and $q_4$ are contained in the same component of
  $(\fM^\sigma)^{\Delta'\CC^\times}$, the polarization in the remaining directions
  must be the same for $q_1$ and $q_4$: it is the polarization of
  $(\fM^\sigma)^{\Delta'\CC^\times}$ in $\fM^\sigma$.

  The decomposition $(x_1\oplus x_{-2}) \oplus (x_2\oplus x_{-1})$ is well-defined, but
  further decomposition $x_1 \oplus x_{-2}$, $x_2\oplus x_{-1}$ is not well-defined.
  Hence we cannot decompose $H^1(\mathcal C(1,-1))\oplus H^1(\mathcal C(-2,2))$,
  but 
  \begin{equation*}
    \left(H^1(\mathcal C(1,-1))\oplus H^1(\mathcal C(-2,2))\right) \oplus
    \left(H^1(\mathcal C(-1,1))\oplus H^1(\mathcal C(2,-2))\right)
  \end{equation*}
  is well-defined.
  The cotangent fiber direction is $H^1(\mathcal C(1,-1))\oplus H^1(\mathcal C(2,-2))$ on $q_1$,
  $H^1(\mathcal C(-1,1))\oplus H^1(\mathcal C(-2,2))$ on $q_4$.
  Thus, the polarization can be chosen either of summands of the above displayed
  equation, or cotangent fiber direction.
\end{NB}%

\subsubsection{Wall \texorpdfstring{$u_1=0$}{u_1=0}}

Finally consider the $K$-matrix, moving from the chamber $\{ u_1 > 0\}$ to $\{ u_1 < 0\}$.
The computation is the same for the $K$-matrix $\{ u_2 > 0\} \to \{ u_2 < 0\}$, just
by exchanging $u_1$ and $u_2$. 

We take the fixed point set with respect to the second $\CC^\times = \CC^\times_{u_2}$ given by $u_1 = 0$,
which is the union of $\ell$ copies of ALE spaces of type $\mathrm{A}_{\ell-1}$. 
We have $\ell$ copies from the fixed points with respect to $u_2$, which is indexed by
the entry $t$ in the $2$-nd row.
We can focus on the $(-1)$-st and $1$-st rows of the Young tableaux, but the $(-1)$-st row
is determined by the $1$-st row anyway. There are $\ell$ tableaux for $t=1,\dots, \ell$.
Since the geometry is the same as for $\{ u_1 + u_2 > 0\}\to \{u_1 + u_2 < 0\}$,
we do not need to repeat the computation. We replace $u = u_1 + u_2$ by $u = u_1 - (-u_1) = 2u_1$ in \eqref{eq:h:17}.
Note that the $K$-matrix has the size $\ell\times\ell$, hence there are no other entries.
We choose the polarization as above, i.e., in the repelling direction of the normal bundle
of $(\fM^\sigma)^A$ in $(\fM^\sigma)^{\CC^\times_{u_2}}$, which is
$H^1(\mathcal C(1,-1))\oplus H^1(\mathcal C(-1,1))$.

The analysis of the wall $u_2 = 0$ is the same.

We still need to consider a polarization of
the normal bundle of $(\fM^\sigma)^{\CC^\times_{u_2}}$ in $\fM^\sigma$, which is
$H^1(\mathcal C(2,-2))\oplus H^1(\mathcal C(-2,2)) \oplus H^1(\mathcal C(1,2))\oplus H^1(\mathcal C(2,1))
\oplus H^1(\mathcal C(-1,2))\oplus H^1(\mathcal C(2,-1))$.
The summands $H^1(\mathcal C(1,2))$, and $H^1(\mathcal C(2,1))$ are $1$-dimensional if
$s\neq t$, and $0$-dimensional if $s=t$. 
Their weights, restricted to $\Lie\CC^\times_{u_2}$, are $u_2$ and $-u_2$ respectively.
On the other hand, the summands
$H^1(\mathcal C(-1,2))$, $H^1(\mathcal C(2,-1))$ are $0$-dimensional if $s\neq t$,
and $1$-dimensional if $s = t$.
Their weights, restricted to $\Lie\CC^\times_{u_2}$, are $u_2$ and $-u_2$ respectively.
If the polarization of the normal bundle of $(\fM^\sigma)^{\CC^\times_{u_2}}$ in $\fM^\sigma$ 
would satisfy Assumption~\ref{assum:induced-polarization}, it must
be constant on connected components of $(\fM^\sigma)^{\CC^\times_{u_2}}$.
Note that the polarization of factors $H^1(\mathcal C(2,-2))\oplus H^1(\mathcal C(-2,2))$ are already
fixed, when we discussed the normal bundle of $(\fM^\sigma)^A$ in $(\fM^\sigma)^{\CC^\times_{u_1}}$.
By the second part of Assumption~\ref{assum:induced-polarization}, this part of the polarization
must be independent of $s$.
Hence the polarization in the remaining factor must be either $u_2$ or $-u_2$ on each component,
i.e., on each fixed $t$.
Since both $s\neq t$ and $s=t$ cases
appear in each component, we should choose the polarization as 
$H^1(\mathcal C(1,2))\oplus H^1(\mathcal C(-1,2))$ or
the opposite $H^1(\mathcal C(2,1))\oplus H^1(\mathcal C(2,-1))$ on each $t$.
\begin{NB}
If this is an induced polarization, the choice must be given by a polarization as
$A$-weights, restricted to $u_1 = 0$. That is, $u_2 - u_1$ for $s\neq t$,
$u_1 + u_2$ for $s=t$ for the choice $H^1(\mathcal C(1,2))\oplus H^1(\mathcal C(-1,2))$.
$u_1 - u_2$ for $s\neq t$, $-u_1 - u_2$ for $s=t$ for the opposite choice.
\end{NB}%

\begin{NB}
  We chose the polarization at $q_1$, $q_2$, $q_3$, $q_4$ \emph{partially}:
  Polarization should be the same for $q_1$, $q_2$ and $q_3$, $q_4$, as they differ only
  $(-2)$, $2$-nd rows.
  We choose $C$-direction in
  $H^1(\mathcal C(1,-1))\oplus H^1(\mathcal C(-1,1))$ ($2$-dim).
  More precisely, $H^1(\mathcal C(1,-1))$ for $q_1$, $q_2$, $q_3$, $q_4$.
  Not specified in
  $H^1(\mathcal C(2,-2))\oplus H^1(\mathcal C(-2,2))$ ($2$-dim),
  $H^1(\mathcal C(1,2))\oplus H^1(\mathcal C(2,1))
  \xrightarrow[\sigma]{\cong} H^1(\mathcal C(-2,-1))\oplus H^1(\mathcal C(-1,-2))$, 
  $H^1(\mathcal C(1,-2))\oplus H^1(\mathcal C(-2,1))
  \xrightarrow[\sigma]{\cong} H^1(\mathcal C(2,-1))\oplus H^1(\mathcal C(-1,2))$,
  where the latter two summands are $0$ and $2$-dimensional for $q_1$, $q_4$,
  $2$ and $0$-dimensional for $q_2$, $q_3$.

  We consider $K(u)\otimes 1$:
  \begin{equation*}
    \begin{pmatrix}
      \frac{2u_1}{2u_1 + \hbar} & 0 & \frac{-\hbar}{2u_1 + \hbar} & 0 \\
      0 & \frac{2u_1}{2u_1 + \hbar} & 0 & \frac{-\hbar}{2u_1 + \hbar} \\
      \frac{-\hbar}{2u_1 + \hbar} & 0 & \frac{2u_1}{2u_1 + \hbar} & 0 \\
      0 & \frac{-\hbar}{2u_1 + \hbar} & 0 & \frac{2u_1}{2u_1 + \hbar}
    \end{pmatrix}
  \end{equation*}

  We \emph{cannot} change the polarization on the part $H^1(\mathcal C(1,2))\oplus H^1(\mathcal C(2,1))$
  for $q_2$ and $q_3$ freely, as the conjugation by
  $\operatorname{diag}(1,-1,1,1)$ or $\operatorname{diag}(1,1,-1,1)$,
  or even $\operatorname{diag}(1,-1,-1,1)$ changes
  $K(u)\otimes 1$.
  The same applies to $H^1(\mathcal C(1,-2))\oplus H^1(\mathcal C(-2,1))$ for 
  $q_1$ and $q_4$.

  The decomposition $(x_1\oplus x_{-1}) \oplus x_2\oplus x_{-2}$ is well-defined,
  but further decomposition $x_1\oplus x_{-1}$ is not well-defined.
  Hence we cannot decompose $H^1(\mathcal C(1,-2))\oplus H^1(\mathcal C(-1,-2))$,
  but
  \begin{equation*}
    \begin{split}
      &\phantom{{}={}}
      \left(H^1(\mathcal C(1,-2))\oplus H^1(\mathcal C(-1,-2))\right) \oplus
      \left(H^1(\mathcal C(-2,1))\oplus H^1(\mathcal C(-2,-1))\right)
      \\
      &\cong
      \left(H^1(\mathcal C(2,-1))\oplus H^1(\mathcal C(2,1))\right) \oplus
      \left(H^1(\mathcal C(-1,2))\oplus H^1(\mathcal C(1,2))\right)
    \end{split}
  \end{equation*}
  and
  \begin{equation*}
    H^1(\mathcal C(2,-2)) \oplus H^1(\mathcal C(-2,2))
  \end{equation*}
  are well-defined.

  On the other hand, on the wall $u_2 = 0$, we chose the polarization
  at $q_1$, $q_2$, $q_3$, $q_4$ \emph{partially}:
  Polarization should be the same for $q_1$, $q_3$ and $q_2$, $q_4$, as they differ only
  $(-1)$, $1$-st rows.
  We choose $C$-direction in
  $H^1(\mathcal C(2,-2))\oplus H^1(\mathcal C(-2,2))$ ($2$-dim).
  More precisely, $H^1(\mathcal C(2,-2))$ for $q_1$, $q_2$, $q_3$, $q_4$.
  Not specified in
  $H^1(\mathcal C(1,-1))\oplus H^1(\mathcal C(-1,1))$ ($2$-dim),
  $H^1(\mathcal C(1,2))\oplus H^1(\mathcal C(2,1))
  \xrightarrow[\sigma]{\cong} H^1(\mathcal C(-2,-1))\oplus H^1(\mathcal C(-2,-1))$, 
  $H^1(\mathcal C(1,-2))\oplus H^1(\mathcal C(-2,1))
  \xrightarrow[\sigma]{\cong} H^1(\mathcal C(2,-1))\oplus H^1(\mathcal C(-1,2))$.

  We consider $1\otimes K(u)$:
  \begin{equation*}
    \begin{pmatrix}
      \frac{2u_1}{2u_1 + \hbar} & \frac{-\hbar}{2u_1 + \hbar} & 0 & 0\\
      \frac{-\hbar}{2u_1 + \hbar} & \frac{2u_1}{2u_1 + \hbar} & 0 & 0\\
      0 & 0 & \frac{2u_1}{2u_1 + \hbar} & \frac{-\hbar}{2u_1 + \hbar} \\
      0 & 0 & \frac{-\hbar}{2u_1 + \hbar} & \frac{2u_1}{2u_1 + \hbar}
    \end{pmatrix}
  \end{equation*}

  The decomposition $x_1\oplus x_{-1} \oplus (x_2\oplus x_{-2})$ is well-defined.
  Then
  \begin{equation*}
    \begin{split}
      &\phantom{{}={}}
      \left(H^1(\mathcal C(1,2))\oplus H^1(\mathcal C(1,-2))\right) \oplus
      \left(H^1(\mathcal C(2,1))\oplus H^1(\mathcal C(-2,1))\right)
      \\
      &\cong
      \left(H^1(\mathcal C(-2,-1))\oplus H^1(\mathcal C(2,-1))\right) \oplus
      \left(H^1(\mathcal C(-1,-2))\oplus H^1(\mathcal C(-1,2))\right)
    \end{split}
  \end{equation*}
  and
  \begin{equation*}
    H^1(\mathcal C(1,-1)) \oplus H^1(\mathcal C(-1,1))
  \end{equation*}
  are well-defined.
\end{NB}%

For the wall $\{ u_2 = 0\}$, an induced polarization of the normal bundle
of $(\fM^\sigma)^{\CC^\times_{u_1}}$ in $\fM^\sigma$ would be either
$H^1(\mathcal C(1,2))\oplus H^1(\mathcal C(1,-2))
\cong H^1(\mathcal C(1,2))\oplus H^1(\mathcal C(2,-1))$ or
$H^1(\mathcal C(2,1))\oplus H^1(\mathcal C(-2,1))
\cong H^1(\mathcal C(2,1))\oplus H^1(\mathcal C(-1,2))$ on each component.

These conditions are not compatible with our choice of the polarization
for the wall $\{ u_1 = u_2\}$ above, as we do not have distinction
between cases $s < t$, $s > t$.
Even if we would change the polarization for the wall $\{u_1 = u_2\}$,
we could not choose induced polarizations simultaneously for
$\{ u_1 = 0\}$ and $\{ u_2 = 0\}$ walls.
Indeed, suppose we
  choose $H^1(\mathcal C(2,1))\oplus H^1(\mathcal C(2,-1))$ for the component 
  of $(\fM^\sigma)^{\CC^\times_{u_2}}$ given by $t=1$.
  In particular, we need to choose $H^1(\mathcal C(2,-1))$ at $s=1$, and
  $H^1(\mathcal C(2,1))$ at $s\neq 1$. 
  Both have weight $-u_2$ on the wall $\{ u_1 = 0\}$, but $-u_1-u_2$
  and $u_1 - u_2$ on the whole Lie algebra of $A$.
  Then considering the component of $(\fM^\sigma)^{\CC^\times_{u_1}}$,
  i.e., moving $t$ fixing $s$,
  we must choose
  $H^1(\mathcal C(2,-1))\oplus H^1(\mathcal C(1,2))$ for $s=1$, and
  $H^1(\mathcal C(2,1))\oplus H^1(\mathcal C(-1,2))$ for $s\neq 1$.
  In particular, we need to choose
  $H^1(\mathcal C(1,2))$ at $s=1$, $t\neq 1$, and
  $H^1(\mathcal C(2,1))$ at $s\neq 1$, $t\neq s$.
  Now we move $s$ from $s=1$, $t\neq 1$ fixing $t$. We have to choose
  $H^1(\mathcal C(1,2)) \oplus H^1(\mathcal C(-1,2))$ for all the components $t\neq 1$.
  Thus the choice cannot be consistent on $H^1(\mathcal C(1,2))$ vs $H^1(\mathcal C(2,1))$  
  for $s\neq 1$, $t\neq 1$, $s\neq t$.

There is an exceptional case $\ell = 2$, as there is no pair $s$, $t$ with $\neq 1$, $s\neq t$. Indeed, one define
an induced polarization by the above argument. We can also see a well-defined polarization since
$\fM^\sigma$ is a cotangent bundle when $\ell=2$.
Indeed, for $\ell=2$, the diagram automorphism $\sigma'$ is trivial, and $\fM^\sigma$ is an example
discussed in \cref{sec:ex2}. We will prove that the equivariant cohomology $H^*_{\TT_\bw}(\fM^\sigma)$ is 
a representation of Molev-Ragoucy reflection equation algebra $\mathcal B(2,1)$.
The reflection equation algebra is a quotient of the extended reflection equation algebra
$\widetilde{\mathcal B}(2,1)$ \cite[\S2.16, Example~6]{MR2355506}, and $\widetilde{\mathcal B}(2,1)$ 
is isomorphic to the extended twisted Yangian $\mathrm{X}(\mathfrak{o}_2)$ in the sense of
\cite[Def.~2.13.1]{MR2355506} by \cite[\S2.16, Example~6]{MR2355506}. Indeed, one of the defining 
relations of $\mathrm{Y}(\mathfrak{o}_2)$, \cite[(2.6)]{MR2355506}, is \emph{not} satisfied
in our case, while the unitarity, which gives the quotient $\widetilde{\mathcal B}(2,1)\to\mathcal B(2,1)$,
is satisfied.


For $\ell > 2$, we observe that $\fM^\sigma$ is an ALE space of type $\mathrm{A}_{\ell-1}$ for $\bw_1 = 1$.
Therefore $H_2(\fL^\sigma)$ is $(\ell-1)$-dimensional. We have no non-trivial $(\ell-1)$-dimensional
representation of $\mathfrak{o}_\ell$. This is an obstruction to choose an induced polarization,
and well-defined $\sY(\mathfrak{o}_\ell)$ representation on $H^*_{\TT_\bw}(\fM^\sigma(\bw))$.

\subsection{\texorpdfstring{$\SO$}{SO}-instanton moduli spaces}

  The same argument can be applied to type $(+)$ case, where we
  put an \emph{orthogonal} form on $W_1\oplus W_{\ell-1}$. The $\sigma$-quiver variety
  $\fM_\zeta^\sigma$ is the moduli space of framed
  $\SO(2\bw_1)$-instantons on an ALE space of type $\mathrm{A}_{\ell-1}$ in
  that case. The action of $\sigma$ on $H^1(\mathcal C(k,-k))$ is
  $-1$, as we multiply the isomorphism $W_{\ell-1}\cong W_1^*$ by $-1$.
  Hence there is no contribution from $H^1(\mathcal C(k,-k))$, and the
  dimension of $\fM_\zeta^\sigma$ is $\bw_1(\bw_1-1)$. The torus
  $A_{\bw_1}$ fixed points are parametrized by $T$ with
  \eqref{eq:27}. We have $\ell^{\bw_1}$ fixed points in total.
  
  The charge $l_{\SO}(T)$ is given \emph{only} by the second term of
  $l_{\grpSp}(T)$ in \eqref{eq:34} because we do not have contribution from
  $H^1(\mathcal C(k,-k))$. Namely,
\begin{equation}
   l_{\SO}(T) = \frac12 \# \{ (k,l,a) \mid k\neq -l \text{ and \eqref{eq:28} is satisfied}\}.
\end{equation}
  By this difference, we have more $T$ with $l_{\SO}(T) =
  0$. Therefore $\fM_\zeta^\sigma$ is \emph{not} connected. Let us
  study several examples:

  \begin{Example}\label{ex:SO}
    (1) Suppose $\bw_1 = 1$, there are $\ell$ isolated points. If
    $\zeta_\CC$ is generic, they are points
    $(x,y,z) = (0,0,-\zeta_{a,\CC})$ ($a=1/2,\dots,\ell-1/2$) in
    $xy = (z+\zeta_{1/2,\CC})\dots (z+\zeta_{\ell-1/2,\CC})$ in
    \cref{ex:ALE}: the involution is $(x,y,z)\mapsto
  (-x,-y,z)$. Therefore we have isolated fixed points
  $(x,y,z) = (0,0,-\zeta_{a,\CC})$ ($a=\frac12,\dots,\ell-\frac12$).
  Like in the explanation in \cref{ex:ALE}, $\fM_\zeta^\sigma$ in this case is
  the moduli space of $\SO(2)$-instantons with fixed Chern class and boundary condition at $\infty$.
  Since $\SO(2)\cong \U(1)$, it is the moduli space of line bundles. But there is at most
  one line bundle with given Chern class and boundary condition at $\infty$.
  Therefore $\fM_\zeta^\sigma$ is a finite set of points. This is a geometric explanation of
  the above description.

\begin{NB}
  For the symmetric case, we still have decomposition as above. If $L$
  is a line bundle $E = L\oplus L^*$ is an $\SO(2)$-bundle so that the
  symmetric bilinear form is
  $(l_1\oplus l^*_1, l_2\oplus l^*_2) = \langle l_1, l^*_2\rangle +
  \langle l_2, l^*_1\rangle$. For the computation for the tangent
  space, we change $V_\ell \cong V_0^*$ by sign. Therefore $\sigma$
  acts on $\mathcal C(k,-k)$ by $-1$, instead of $1$. In particular,
  the $\sigma$-fixed part of the tangent space for the
  $\bw_1 = 1$ case is $0$.
\end{NB}%

(2) Suppose $\bw_1 = 2$. Let us check that there is one component,
which is an ALE space of type $\mathrm{A}_{\ell-1}$, and 
$\frac12 \ell(\ell-1)$ other components, which are ALE spaces of type
$\mathrm{A}_1$:
Since the charge $l_{\SO}(T)$ of $T$ in this case is only the second term in
$l_{\grpSp}(T)$ of \eqref{eq:34}, tableaux $T$ with $l_{\SO}(T) = 0$ are
\begin{equation}\label{eq:29}
  \Yvcentermath1
  \young(2345,2345,1,1)\ ,\quad
  \young(\ \ \ \ ,\ \ \ \ ,s,t)\ 
  \quad (s > t)
\end{equation}
where the $(-2)$-nd and $(-1)$-st rows consist of $1$ to $5$ excluding
$t$ and $s$ respectively by the rule \eqref{eq:27}.

Other fixed points are
\begin{equation*}
  \Yvcentermath1
  \young(\ \ \ \ ,\ \ \ \ ,s,s)
  \quad (s\neq 1),\quad
  \young(\ \ \ \ ,\ \ \ \ ,t,s)
  \quad (s > t)
\end{equation*}
which have index $1$.

The first type of a fixed point is connected to the first one in
\eqref{eq:29}. This can be checked as follows. We consider the sum of
data for the $1$-st and $(-2)$-nd rows only, as a quiver variety with
dimension vectors ${\bw} = (1,0,0,1)$,
${\bv} = (0,0,0,0)$. This quiver variety is an ALE space of
type $\mathrm{A}_{\ell-1}$. The fixed points are classified by the same data,
just taking the $1$-st and $(-2)$-nd rows only.
We add the image of data under $\sigma$ to see that those fixed points
are connected.

In the same manner, the second type of a fixed point sits in the same
connected component of the second one in \eqref{eq:29}. We take the
sum of the $1$-st and $2$-nd rows in this case. The component is an ALE
space of type $\mathrm{A}_1$.

In fact, these description give the whole $\fM_\zeta^\sigma$, not only
fixed points, as we know $\dim \fM_\zeta^\sigma = 2$, and we cover all
fixed points. In particular, for the second type of components,
$C_{\ell-\frac12}\dots C_{\frac12}$ and
$D_{\frac12}\dots D_{\ell-\frac12}$ vanish. In fact, when we first
take the sum of $1$-st and $2$-nd, these vanish as there is no $V_\ell$
in that component. Since a point is the direct sum of this plus its
image under $\sigma$, the assertion follows.

On the other hand, for the first type of components,
$C_{\ell-\frac12}\dots C_{\frac12}$ and
$D_{\frac12}\dots D_{\ell-\frac12}$ do \emph{not} vanish generically.


\begin{NB}
  This is a moduli space of $\SO(4)$-bundle over an ALE space $X$ of
  type $\mathrm{A}_{\ell-1}$.
  Let us consider $w_2\in H^2(X,\ZZ/2)\cong (\ZZ/2)^{\oplus
    \ell-1}$. Fixed points are direct sum $L_1\oplus L_2$ of
  $\SO(2)$-bundles, and $L_1$, $L_2$ can be read off from the $1$-st
  and $2$-nd rows of the Young diagrams respectively. If the first
  entry is $a$, $c_1(L_1) = (0,\dots,0,-1,1,0,\dots,0)$, where $1$
  appears in the $a$-th entry. (In particular,
  $c_1(L_1) = (1,0,\dots,0)$ for $a=1$, $=(0,\dots,0,-1)$ for $a=\ell$.)
  Therefore the first fixed point in \eqref{eq:29} has $w_2 = 0$,
  while the second fixed point in \eqref{eq:29} has $w_2$ with nonzero
  entries at $(b-1)$, $b$, $(a-1)$, $a$ if $a-1 > b$, or at $(a-2)$,
  $a$ if $a-1 = b$. If $a=2$, $b=1$, $\ell=2$, we do not have nonzero
  entries. Hence it cannot be distinguished by $w_2$.
\end{NB}%

(3) Suppose $\ell = 2$. Consider $\zeta$ with $\zeta_\CC = 0$, but
$\zeta_\RR$ is a standard one. Then $\fM_\zeta$ is the cotangent
bundle of the Grassmannian of $\bw_1$-dimensional subspaces in
$\CC^{2\bw_1}$. The $\sigma$-quiver variety $\fM_\zeta^\sigma$ is the
cotangent bundle of the Grassmannian of maximal isotropic subspaces,
which is known to have \emph{two} connected components. We have two
Young tableaux with $l_{\SO}(T) = 0$: (a) $T(k,1) = 1$ for $k > 0$,
$=2$ for $k < 0$ and (b) $T(k,1) = 1$ for $k > 1$ or $k=-1$,
$=2$ for $k < -1$ or $k=1$.
More generally, connected components of fixed points can be
distinguished by the parity of the number of $1$ in $T(k,1)$ for
$k > 0$: if $T(k,1) = T(k',1)$ for $k > k' > 0$, we use the analysis
in the above example (2) to conclude that it is contained in the same
connected component of the fixed point for $T'$, obtained from $T$ by
replacing $T(k,1)$, $T(k',1)$ by $1$. Repeating this procedure, we arrive at
either $T$ of (a) or (b).
The number of $A_{\bw_1}$-fixed points in each component is
$2^{\bw_1-1}$: $T(k,1)$ for $k=1,\dots, \bw_1-1$ is either $1$ or
$2$. The final entry $T(\bw_1,1)$ is determined by the parity of the
number of $1$ in $T(k,1)$ in total.
\end{Example}

\begin{Remark}\label{rem:coulombD}
  Suppose $\bw_1$ is even.
  Naively $\bM^\sigma\tslash G^\sigma$ is \emph{expected} to be the Coulomb
  branch of the following type D quiver theory:
  
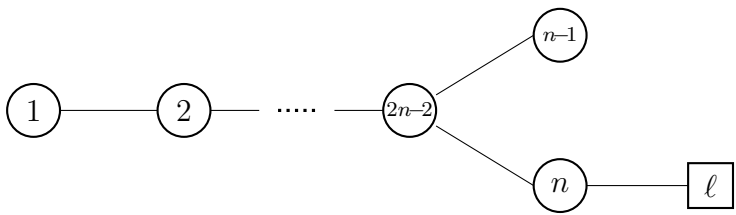
\begin{equation*}
    \begin{tikzpicture}[baseline=0pt,decoration={brace,amplitude=7},scale=.5,
    dot/.style = {draw, circle, thick, minimum size=#1,
              inner sep=0pt, outer sep=0pt},
    dot/.default = 6pt  
                    ] 
    \node[dot=20pt] at (0,0) (v1) {$1$};
    \node[dot=20pt] at (4,0) (v2) {$2$};
    \draw[-] (v1.east) -- (v2.west);
    \draw[-] (v2.east) -- (6,0);
    \draw[-,very thick,dotted] (6.5,0) -- (7.5,0);
    \node[dot=20pt] at (10,0) (v3) {$\scriptstyle 2n\!-\!2$};
    \draw[-] (8,0) -- (v3.west);
    \node[dot=20pt] at (14,2) (v4) {$\scriptstyle n\!-\!1$};
    \draw[-] (v4.west) -- ([yshift=1em]v3.east);
    \node[dot=20pt] at (14,-2) (v5) {$n$};
    \draw[-] ([yshift=-1em]v3.east) -- (v5.west);
    \node[draw,regular polygon,regular polygon sides=4,thick,
    inner sep=0pt, outer sep=0pt,minimum size=24pt]
     at (18,-2) (w5) {$\ell$};
    \draw[-] (v5) -- (w5);
    \end{tikzpicture}
  \end{equation*}
  where $n = \bw_1/2$. Indeed, de Campos Affonso studied $\bM^\sigma\tslash G^\sigma$ during
  preparing his master's thesis \cite{Henrique}, and found that it is \emph{not} irreducible.
  This is seen from the above computation of connected components of $\fM_\zeta^\sigma$.
  The Coulomb branch of the above quiver theory is given by the irreducible component
  corresponding to $\lsp{\circ}\fM_\zeta^\sigma$, where $C_{\ell-\frac12}\dots C_{\frac12}$
  and $D_{\frac12}\dots D_{\ell-\frac12}$ are invertible generically, as in (2).
  We identify the connected component of $\fM_\zeta^\sigma$ with the deformed Coulomb branch
  in a separate publication.
  Like in \cref{rem:coulombC}, we expect the cohomology of $\fM_\zeta^\sigma$ has commuting 
  actions of twisted Yangian and $\mathfrak{so}(\bw_1)$.

  As in \cref{rem:SdualC}, the categorical quotient $\bM^\sigma\tslash G^\sigma$ is
  \eqref{eq:h:23} 
  when $\ell$ is even,
  or \eqref{eq:h:22} 
  when $\ell$ is odd.
  In this case, there is a different delicate point from \cref{rem:SdualC},
  as $G^\sigma$ is disconnected when $\ell$ is even. See \cite{Examples_of_S-dual}.
\end{Remark}

\subsection{\texorpdfstring{$K$}{K}-matrix for type (+)}\label{subsec:K-matrix+}

Let us first suppose $\bw_1 = 2$ as in \cref{subsec:K-matrix}.

Computation of the $R$-matrix for type $(+)$ is the same as one in 
type $(-)$. It is because we are only concerned with the usual quiver
varieties, and the computation does not depend on the choice of type.

The computation of the $K$-matrix becomes trivial, as the fixed point satisfies
$(\fM^\sigma)^{\CC^\times_{u_2}}$ are isolated points in this case. Therefore the corresponding
$K$-matrix is the identity matrix.

A crucial difference is about the choice of a polarization. 
In type $(-)$, we cannot choose a polarization, which gives an \emph{induced} polarization on
the walls $\{ u_1 = 0\}$, $\{u_2=0\}$.
This issue disappears in type $(+)$.

Indeed, consider $(\fM^\sigma)^{\CC^\times_{u_2}}$. The entry $t$ in the
$2$-nd row is fixed on each component of $(\fM^\sigma)^{\CC^\times_{u_2}}$. In type $(-)$,
we cannot fix the entry $s$ in the $1$-st row, as each component is the type $\mathrm{A}_{\ell-1}$ ALE
space. However, in type $(+)$, $(\fM^\sigma)^{\CC^\times_{u_2}}$ consists
of isolated points, we can fix $s$ as well. In particular, any choice of a polarization
is automatically an induced one. We choose a polarization at walls
$\{ u_1 = u_2\}$, $\{ u_1 + u_2 = 0\}$ as in type $(-)$, and it works
for the remaining walls $\{ u_1 = 0\}$, $\{ u_2 = 0\}$.

This argument works for $\bw_1 \ge 2$ as well.
Let us denote the corresponding $\sigma$-quiver variety by $\fM^\sigma(\bw_1)$.
Here we fix a decomposition
$W = W^0 \oplus W^1\oplus W^{-1}\oplus W^2\oplus W^{-2}$ and consider
two dimensional torus $A$ acting on $W^1$, $W^{-1}$, $W^2$, $W^{-2}$ as
above, and trivially on $W^0$.
Connected components of $(\fM^\sigma(\bw))^A$ are divided according to
$s$, $t = 1,\dots,\ell$ according to summands for $W^1$, $W^2$ respectively as above. Together
with connected components of the remaining factor for $W^0$, i.e., $\fM^\sigma(\bw_1-2)$, we have 
$\pi_0(\fM^\sigma(\bw_1)^A) \cong \{(s,t) \mid s,t = 1,\dots,\ell\} \times \pi_0(\fM^\sigma(\bw_1-2))$.
In the fixed point set $(\fM^\sigma(\bw_1))^{\CC^\times_{u_2}}$ with respect to the smaller torus $\CC^\times_{u_2}$, we fix the entry $t$.
Thus $\pi_0((\fM^\sigma(\bw_1))^{\CC^\times_{u_2}}) \cong \{ t=1,\dots,\ell\} \times \pi_0(\fM^\sigma(\bw_1-1))$.
The normal bundle of a component of $(\fM^\sigma(\bw_1))^A$ has an additional summands
$H^1(\mathcal C(a,0))$, $H^1(\mathcal C(0,a))$ for $a=1,2$, but it is irrelevant to
the issue of choosing a well-defined induced polarization: the only issue was
choices between $u_1-u_2$ vs $-u_1-u_2$.
\begin{NB}
  Indeed, we can choose either of $H^1(\mathcal C(a,0))$, $H^1(\mathcal C(0,a))$, as they
  are dual to each other. Here we choose $a=1,2$ already. Hence we do not need to confuse with
  $H^1(\mathcal C(a,0))\cong H^1(\mathcal C(0,-a))$.
\end{NB}%
Then we only need to care summands $H^1(\mathcal C(a,b))$ for $a,b=\pm 1$, $\pm 2$.
But they already appeared in the case $\bw_1 = 2$: we choose
arbitrary polarization as above.

  %
  %

\begin{Theorem}\label{thm:O_twisted}
  The equivariant cohomology $H^*_{\TT_\bw}(\fM^\sigma(\bw_1))$ is a
  representation of Olshanski twisted Yangian $\mathrm{Y}(\mathfrak o_\ell)$.
\end{Theorem}

\begin{Remark}
It is natural to ask which $\sigma$-quiver varieties give rise representations
of the twisted Yangian $\mathrm{Y}(\algsp_{2\ell})$.
As in \cref{rem:coulombC}, one should look at Howe duality, 
involving $\mathrm{Y}(\algsp_{2\ell})$, or its variant, the coideal subalgebra $\bU'_q(\algsp_{2\ell})$.
The author could not find it, but \cite{MR3939581} provides one in which
the quantum group and the coideal subalgebra are exchanged.
The dual algebra is $\bU_q(\algsp_{2\bw_1})$, hence we should look at the
Coulomb branch of a quiver gauge theory of type $\mathrm{B}_{\bw_1}$.
This Coulomb branch was studied by de Campos Affonso during preparing his
master's thesis \cite{Henrique}. For a certain choice of dimension vectors,
it is a $\sigma$-quiver variety of type $\mathrm{A}_{2\ell-1}$
\begin{NB}
  as we want to realize twisted Yangian for a symmetric pair 
  $(\fg,\fg^\sigma) = (\mathfrak{sl}_{2\ell},\algsp_{2\ell})$
\end{NB}%
as in \eqref{eq:31}, but additional $1$-dimensional framing at the middle vertex.
Like in \cref{ex:no-solution}, there is no integral solution $\bv$ of $\mu = \bw - \bC\bv = 0$.
It means that we need to consider a singular $\sigma$-quiver variety.
\end{Remark}

\subsection{Type D and \texorpdfstring{$\mathrm{E}_6$}{E6}}

Consider moduli spaces of $\SO$-instantons on an ALE space of other types.
In the above argument for the case of type A, it was important that
$\fM^\sigma$ for $\bw_1 = 1$ is a finite set of points. It is because 
$\fM^\sigma$ is the moduli space of $\SO(2)$-instantons. See \cref{ex:SO}(1).
This explanation remains valid for other types. Concretely, we take
a vertex $i$ of the Dynkin diagram, which corresponds to a minuscule weight.
Concretely $i$ is the leftmost vertex of type D, or one in two tails of
type D, or the leftmost/rightmost vertex of type $\mathrm{E}_6$. We then choose
$\bw$ such that the corresponding weight $\lambda$ is $\varpi_i + \varpi_{\invast[i]}$.
We then choose $\bv$ so that the corresponding $\mu$ is $0$. We consider the
corresponding $\sigma$-quiver variety $\fM^\sigma(\bv,\bw)$ with $\sigma'=\id$.
In \cite[Table~1]{MR3900699} referred in \cref{rem:symmetric-pair}, it
corresponds to $(\fg,\fg^\sigma) = (\mathfrak{so}_{2\ell},\mathfrak{\so}_\ell\oplus
\mathfrak{so}_\ell)$ in Satake type DI, and $(E_6, \mathfrak{sp}_4)$ in Satake type EI.

We can choose polarizations which are induced on $\{u_1=0\}$, $\{ u_2 = 0\}$ walls as above,
hence get a representation of $\sX^{\mathrm{tw}}$ on $H^*_{\TT_\bw}(\fM^\sigma(\bv,\bw))$,
where $\bw$ corresponds to $\lambda = N(\varpi_i + \varpi_{\invast[i]})$ for $N=1,2,\dots$.
($N$ here corresponds to $\bw_1$ in the above subsections.)

In this case, the $K$-matrix is the identity matrix, as in \cref{subsec:K-matrix+}.
The $R$-matrix is the one for the $\ell$-fundamental representation of $\sY$ of
the underlying Dynkin diagram, i.e., of type D or $\mathrm{E}_6$.
As in \cref{subsec:K-matrix}, we need to compute it in terms of the fixed point
basis, but we leave it as an exercise for the reader.

After the computation, one should compare $\sX^{\mathrm{tw}}$ for
type D with $i = $ the leftmost vertex with the twisted Yangian
in \cite{MR3545488}.
Note that twisted Yangian in \cite{MR3545488} is classified into two types,
the first and second kinds, according to the $K$-matrix,
$\mathcal G(u)$ in their notation. For type D with $(\fg,\fg^\sigma) = (\mathfrak{so}_{2\ell},\mathfrak{\so}_\ell\oplus
\mathfrak{so}_\ell)$, it is of the first kind, i.e., $\mathcal G(u)$ is a constant matrix
with respect to $u$.
This is consistent with our $K$-matrix, which is the identity matrix.

\section{Examples -- partial flags}\label{sec:ex2}

Let us continue examples of involutions in type $\mathrm{A}_{\ell-1}$, but now
with $\sigma' \circ \invast = \id$. Therefore the diagram for
$\bM^\sigma/G^\sigma$ in \cref{subsec:cat_quotient} is an
orthosymplectic one. Let us further assume that framing vector spaces
appear only at the leftmost vertex. This class of examples studied in
\cite[\S6]{MR3900699}. We will use same techniques as in
\cref{sec:ex1} to compute the $K$-matrix explicitly, and show that the
twisted Yangian is the \emph{reflection equation algebra} in
\cite[\S2.16, Example~4]{MR2355506}.

\subsection{Partial flags}

Let us label linear maps as in \eqref{eq:31}, but without framing at
the rightmost vertex:
\begin{equation}\label{eq:h:20}
  \begin{tikzcd}[execute at end picture={
      \draw[dashed,thin]
      ([yshift=2ex]$(\tikzcdmatrixname-1-3.north)!0.5!
      (\tikzcdmatrixname-1-4.north)$)
      --
      ([yshift=-2ex]$(\tikzcdmatrixname-2-3.north)!0.5!
      (\tikzcdmatrixname-2-4.north)$);
      }]
      V_1 \arrow[r,yshift=.5ex,"C_{\frac32}"] 
      \arrow[d,xshift=.5ex,near end,"D_{\frac12}"]
      &
      V_2 \arrow[l,yshift=-.5ex,"D_{\frac32}"]
      \arrow[r,yshift=.5ex,"C_{\frac52}"]
      &
      \phantom{V}
      \arrow[r,no head,very thick,dotted]
      \arrow[l,yshift=-.5ex,"D_{\frac52}"]
      &
      \phantom{V}
      \arrow[r,yshift=.5ex,"C_{\ell-\frac52}"] &
      V_{\ell-2} \arrow[l,yshift=-.5ex,"D_{\ell-\frac52}"]
      \arrow[r,yshift=.5ex,"C_{\ell-\frac32}"] & V_{\ell-1}
      \arrow[l,yshift=-.5ex,"D_{\ell-\frac32}"]
      \\
      W
      \arrow[u,xshift=-.5ex, "C_{\frac12}"]
      &&\phantom{V}&\phantom{V}&&
      \phantom{W}
  \end{tikzcd}
\end{equation}
We drop the index $1$ from $W_1$ in this section.
Let us set $\bv_0 = \bw$, $\bv_{\ell} = 0$ for convention.
Our convention for parameters $\zeta$ is the same as in \cref{sec:ex1}.

Among the assumption $\sigma(\zeta) = \zeta$, $\sigma(\bv) = \bv$,
$\sigma(\bw) = \bw$, the last one is automatic. (See
\cref{def:involution}.) The first one means
$\zeta_{3/2} = \zeta_{\ell-3/2}$, $\zeta_{5/2} = \zeta_{\ell - 5/2}$,
\dots. Note that we can achieve this condition without violating the
genericity condition. However, the number of deformation parameters is
about half. The middle condition means
\begin{equation}\label{eq:h:21}
  \bv_0 - \bv_1 = \bv_{\ell-1}-\bv_{\ell},\quad
  \bv_1 - \bv_2 = \bv_{\ell-2} - \bv_{\ell-1}, \quad
  \dots
\end{equation}
In other words, $\bw = \bv_i + \bv_{\ell-i}$ for all $i=1,\dots,\ell-1$.

Depending on the type $(+)$, $(-)$, we assign an $\ve$-form on
$W$. Thus the involution $\sigma$ on $\fM_\zeta(\bv,\bw)$ is defined.

Let us concentrate on the case $\zeta_\CC = 0$, but
$\sqrt{-1}\zeta_{i,\RR} > 0$. Hence $\fM^\sigma_\zeta(\bv,\bw)$ is smooth.

The whole quiver variety $\fM_\zeta(\bv,\bw)$ was described in
\cite[\S7]{Na-quiver}. It is the cotangent bundle of a (partial) flag
variety of $\GL(W)$. It is the space of pairs
$(x\in\End(W), S_1\supset S_2\supset \dots \supset S_{\ell-1}\supset 0)$
satisfying conditions
\begin{equation*}
  x(S_i)\subset S_{i+1}
\end{equation*}
where $\dim S_i = \bv_i$. The $\GL(W)$-action is the natural one.
The involution was computed in \cite[\S6]{MR3900699}, and the
$\sigma$-quiver variety is the cotangent bundle of the partial flag
variety of classical type. More precisely, it is the space of pairs
$x\in\End(W)$ and flags $W\supset S_1\supset S_2\supset \dots \supset
S_{\ell-1}\supset 0$ as above, with further conditions
\begin{equation*}
  x^* = -x, \quad
  S_i^\perp = S_{\ell-i},
\end{equation*}
where $*$ and $\perp$ are defined by the chosen $\ve$-form on $W$.
The $G_\ve(W)$-action in \cref{rem:framing-symmetry} is identified
with the standard one on the cotangent bundle of the partial flag variety.

The categorical quotient $\bM^\sigma/G_\bv^\sigma$ in
\cref{def:cat_quot} is an orthosymplectic quiver variety, studied by
\cite{MR694606}. If the dominance condition
\begin{equation*}
  \bv_0 - \bv_1 \ge \bv_1 - \bv_2 \ge \dots \ge \bv_{\ell-2} - \bv_{\ell-1}
  \ge \bv_{\ell-1} - \bv_\ell
\end{equation*}
is satisfied, $\bM^\sigma/G_\bv^\sigma$ is isomorphic to the closure
of nilpotent $G_\ve(W)$-orbits in $\fg_\ve(\bw)$ corresponding to the
partition given by the above sequence of integers.
Unfortunately the dominance condition and \eqref{eq:h:20} are
compatible only when
$\bv_0 - \bv_1 = \dots = \bv_{\ell-2} - \bv_{\ell-1} = \bv_{\ell-1} - \bv_\ell$.
Indeed, the image of $\fM^\sigma_\zeta(\bv,\bw)$ under the projective
morphism $\pi$ is the nilpotent orbit corresponding to the partition,
which is obtained from
$\bv_0 - \bv_1$, $\bv_1 - \bv_2$, \dots $\bv_{\ell-2} - \bv_{\ell-1}$,
$\bv_{\ell-1}-\bv_\ell$ by sorting them in descending order.

\begin{Example}\label{ex:partial-flag}
  As in \cref{sec:ex1}, it is important to understand examples with
  small $\bw$. Let us study $\bw = 1$, $2$, or $3$.

  (1) Suppose $\bw = 1$. Only type $(+)$ is possible. By the condition
  \eqref{eq:h:21}, the only possibility is the case $\ell$ is odd and
  $\bw = \bv_1 = \dots = \bv_{(\ell-1)/2} = 1$,
  $0 = \bv_{(\ell+1)/2} = \dots = \bv_{\ell-1}$. Indeed,
  $\bw - \bv_1$, $\bv_1-\bv_2$, \dots, $\bv_{\ell-1}$ are $0$ or $1$,
  the total sum is $1$, and they appear in pairs by \eqref{eq:h:21}.
  Thus it is the only possibility.
  The corresponding $\fM^\sigma_\zeta(\bv,\bw)$ is a single point.
  Note also the whole quiver variety $\fM_\zeta(\bv,\bw)$ is a single point in this case.

  (2) Suppose $\bw=2$ and the type is $(+)$. We have a decomposition
  $W = W(1)\oplus W(-1)$ as before by the eigenvalue of
  $t\in G_\ve^0(W) = \SO(2) = \CC^\times$. There are only two
  $1$-dimensional isotropic subspaces, namely $W(1)$ and $W(-1)$. The
  linear map $x\in\End(W)$ is automatically $0$ as it is in
  $\fg_\ve(W)$ and nilpotent. Thus $\fM^\sigma_\zeta(\bv,\bw)$ is either
  empty, a single point or two points. It is a single point if and
  only if $\ell$ is odd and $\bw = \bv_1 = \dots = \bv_{(\ell-1)/2} = 2$,
  $0 = \bv_{(\ell+1)/2} = \dots = \bv_{\ell-1}$.
  It is two points if and only if
  $\bw = \bv_1 = \dots = \bv_{i-1} = 2$,
  $1 = \bv_i = \bv_{i+1} = \dots = \bv_{\ell-i}$,
  $0 = \bv_{\ell+1-i} = \dots = \bv_{\ell-1}$ for some
  $i\le \lfloor \ell/2\rfloor$.
The whole quiver variety $\fM_\zeta(\bv,\bw)$ is also a single point
in the case $\fM_\zeta^\sigma(\bv,\bw)$ is a single point,
and $T^*\proj^1$ in the case $\fM_\zeta^\sigma(\bv,\bw)$ is two points.

  (3) Suppose $\bw=2$ and the type is $(-)$. Contrary to the case (2)
  above, arbitrary $1$-dimensional subspace is lagrangian in the
  symplectic vector space $W$. The condition $x^* = -x$ simply says
  $x$ is trace-free, but it is automatically satisfied if $x$ is
  nilpotent. The possible dimension vectors giving nonempty
  $\fM^\sigma_\zeta(\bv,\bw)$ are the same as the case (2), but we get
  $T^*\proj^1$ in the case when $\fM^\sigma_\zeta(\bv,\bw)$ is a two point above.
  The whole quiver variety $\fM_\zeta(\bv,\bw)$ is the same as $\fM^\sigma_\zeta(\bv,\bw)$,
  in other words, $\sigma$ is a trivial involution.

  (4) Suppose $\bw = 3$ and the type is $(+)$. As in the case (1),
  $\ell$ is odd and $\bv_{(\ell-1)/2} - \bv_{(\ell+1)/2}$ is either
  $3$ or $1$.  In the former case,
  $\bw = \bv_1 = \dots = \bv_{(\ell-1)/2} = 3$,
  $0 = \bv_{(\ell+1)/2} = \dots = \bv_{\ell-1}$. The corresponding
  $\fM^\sigma_\zeta(\bv,\bw)$ is a point. The whole quiver variety
  $\fM_\zeta(\bv,\bw)$ is also a point in this case.
  
  In the latter case,
  $3 = \bw = \bv_1 = \dots = \bv_{i-1}$,
  $2 = \bv_i = \bv_{i+1} = \dots = \bv_{(\ell-1)/2}$,
  $1 = \bv_{(\ell+1)/2} = \dots = \bv_{\ell-i}$,
  $0 = \bv_{\ell+1-i} = \dots = \bv_{\ell-1}$ for some
  $i\le (\ell-1)/2$. We get $T^*\proj^1$ in this case.
  The whole quiver variety $\fM_\zeta(\bv,\bw)$ is the cotangent bundle
  of the full flag variety for $\CC^3$.
\end{Example}

\subsection{Torus action}

As in 
\cref{subsec:torus_action}, we decompose $W$ into $1$-dimensional spaces:
\begin{equation*}
  W = W(1)\oplus W(-1) \oplus W(2) \oplus W(-2) \oplus \dots.
\end{equation*}
The final part of $\dots$ is $W(\bw/2)\oplus W(-\bw/2)$ if $\bw$ is
even. If $\bw$ is odd, we continue until
$W((\bw-1)/2)\oplus W(-(\bw-1)/2)$ and add one more summand $W(0)$. We
then require $W(k)\cong W(-k)^*$ under the $\ve$-form, which remains to be true for $k=0$.

Let us take a maximal torus $A_\bw$ of $G_\ve(W)$ preserving the above
decomposition. Note that $A_\bw$ is $\bw/2$-dimensional
if $\bw$ is even, while $(\bw-1)/2$ dimensional if $\bw$ is odd. 

The torus fixed points are parametrized by Young tableaux as
before. The differences are followings:
\begin{itemize}
\item All rows have length $1$,
\item The cardinarity of the nodes with $i$ is $\bv_{i-1}-\bv_i$.
\item If the entry of $k$-th row is $s$, the entry of $(-k)$-th row is
  $\ell+1-s$. In particular, the only possible entry of the $0$-th row is
  $s = (\ell+1)/2$.
\end{itemize}

For example,
\begin{equation*}
  \vcenter{\hbox{\begin{ytableau}
    \none[-2] & \none & 5 \\
    \none[-1] & \none & 4\\
    \none[1] & \none & 2 \\
    \none[2] & \none & 1
  \end{ytableau}}}
  \qquad
  \vcenter{\hbox{\begin{ytableau}
    \none[-2] & \none & 5 \\
    \none[-1] & \none & 4\\
    \none[0] & \none & 3 \\
    \none[1] & \none & 2 \\
    \none[2] & \none & 1
  \end{ytableau}}}
\end{equation*}
for $\bw = 4$, $\ell = 5$ and $\bw=5$, $\ell = 5$ respectively.

Thanks to \cref{ex:partial-flag}, we can compute $\sigma$ on the complex
$H^1(\mathcal C(k,-k))$, and it is the identity map for $(-)$-type,
and $-\id$ for $(+)$-type. 
In case both entries of $k$ and $-k$-th rows are $(\ell+1)/2$,
$H^1(\mathcal C(k,-k))$ is $0$ anyway, hence it does not matter
which sign we take.

\subsection{\texorpdfstring{$K$}{K}-matrix}

Let us compute the $K$-matrix in this example.

Since $\fM(\bw)$ is the cotangent bundle of a partial flag variety,
we can choose a polarization as the cotangent directions.
It is automatically an induced polarization.

We consider the case $\bw=4$ or $5$ (type $(+)$ only). We take a maximal torus $A_\bw$
in $G_\ve(W)$, which is two dimensional torus in all cases. 
The Young tableaux are
\begin{equation*}
\begin{ytableau}
    \none[-2] & \none & s_2 \\
    \none[-1] & \none & s_1\\
    \none[0] & \none & \scriptstyle{\frac{\ell+1}2} \\
    \none[1] & \none & t_1 \\
    \none[2] & \none & t_2
  \end{ytableau}
\end{equation*}
with $s_1+t_1 = s_2 + t_2 = \ell+1$ when $\bw=5$. If $\bw=4$, we omit
the $0$-th row. In particular, the number of fixed points is always $\ell^2$.

The $R$-matrix for $u_1 = u_2$ is the same as the one in
\cref{subsec:K-matrix}\ref{subsub:A}, as we only need to consider
the same quiver varieties.
Let us decide to write the $K$-matrix as a matrix with rows and columns
indexed by
\ytableausetup{smalltableaux}
\begin{ytableau}
    t_1 \\ t_2
\end{ytableau}
as before.

The $R$-matrix for $u_1 = -u_2$ is also the same as the one in
\cref{subsec:K-matrix}\ref{subsub:A} with $u_1-u_2$ replaced by $u_1+u_2$,
as the geometry is the same.
There is no shift of spectral parameter unlike \cref{subsec:K-matrix}\ref{subsub:B}.
See \cref{rem:shift_of_R}.
However, there is one caveat: the natural index set is for
\begin{ytableau}
    s_1 \\ t_2
\end{ytableau}.
We want to write the $R$-matrix as the same matrix with $R(u_1+u_2)$,
$(1 + P\hbar (u_1+u_2)^{-1})(1+\hbar(u_1+u_2)^{-1})$.
Then we will write the $K$-matrix from a vector space with basis indexed
by
\begin{ytableau}
  t_1
\end{ytableau}
to \emph{another} vector space with basis indexed by
\begin{ytableau}
  s_1
\end{ytableau}.
\ytableausetup{nosmalltableaux}

In order to illustrate how $K$-matrix looks like, let us
consider the case $\bw=2$ of type $(+)$. Then all varieties involved
are points, hence entries of the $K$-matrix is either $0$ or $1$.
By the above discussion, the $K$-matrix is the anti-identity matrix:

\ytableausetup{smalltableaux}
\begin{table}[h]
\centering
\begin{blockarray}{lccc}
\diagbox[width=3.2em,height=2em]{
\begin{ytableau}
  s_1
\end{ytableau}
}{
\begin{ytableau}
  t_1
\end{ytableau}
} & $1$ & $\dots$ & $\ell$ \\
\begin{block}{c(ccc)}
  $1$ & $0$ & $\dots$ & $1$ \\
  $\vdots$ & $\vdots$ & $\iddots$ & $\vdots$ \\
  $\ell$ & $1$ & $\dots$ & $0$\\
\end{block}
\end{blockarray}
\ytableausetup{nosmalltableaux}
\end{table}
\begin{NB}
\begin{equation*}
\bordermatrix{
  \scalebox{-1}[1]{\nicefrac{\scalebox{-1}[1]{
  \framebox[1.2em]{\rule{0pt}{0.6ex}$t_1$}
  }}
  {\scalebox{-1}[1]{
  \framebox[1.2em]{\rule{0pt}{0.6ex}$s_1$}
  }}}%
            & 1     & \cdots & \ell     \cr
    1     & 0    & \ldots & 1       \cr
    \vdots  & \vdots  & \iddots & \vdots   \cr
    \ell     & 1     & \ldots & 0     \cr
}
\end{equation*}
\end{NB}%
Note that this matrix is conjugate to $\operatorname{diag}(1,\dots,1,-1,\dots,-1)$,
where $1$ appears $\lceil \frac{\ell}2 \rceil$ times and $-1$ appears
$\lfloor \frac{\ell}2 \rfloor$ times.
Thus the extended twisted Yangian, more precisely the version with
fixed $i = 1$ in \cref{rem:smaller_ext-Yang}, is the
Molev-Ragoucy reflection equation algebra $\mathcal{B}(\ell,\lfloor \frac{\ell}2\rfloor)$.
(See \cite[\S2.16, Example~4]{MR2355506} and \cite{MR1894013}.)
\begin{NB}
  Comment on the the unitary condition ?
\end{NB}%

If $\bw=2$ with type $(-)$, or $\bw=3$ with type $(+)$,
the $K$-matrix has nontrivial $(i,j)$ entries if $j=i$ or $j=\ell+1-i$
except $i=j=(\ell+1)/2$. Those nontrivial entries are the same as
the ones in \cref{subsec:K-matrix}\ref{subsub:A} with $u_1-u_2$ replaced by
$2u_1$.
The reflection equation algebra is alternatively defined by
the reflection equation and the unitary condition. The latter holds
thanks to \cref{subsec:unitarity}. Therefore the extended twisted Yangian,
the version with fixed $i=1$ in \cref{rem:smaller_ext-twist-Yang}(4), is again the reflection equation algebra.
Nevertheless the associated embedding $\sX^{\mathrm{tw}}\to \sX$ is different
from the one for $\bw=2$ with type $(+)$ above.

From these analysis, we get the following for general $\bw$.

\begin{Theorem}\label{thm:MR_refl}
  The equivariant cohomology $H^*_{\TT_\bw}(\fM^\sigma(\bw))$ is a representation
  of the Molev-Ragoucy reflection equation algebra $\mathcal{B}(\ell,\lfloor \frac{\ell}2\rfloor)$.
\end{Theorem}


\appendix
\section{ADHM description of instantons on ALE spaces for classical groups}
\label{sec:ADHM}

When a gauge group is a unitary group, instantons on ALE spaces have a
description in terms of quiver representations \cite{KN}. It is a
modification of ADHM description \cite{ADHM} of instantons on
$S^4$. (It followed more closely the presentation in \cite{MR1079726}.)
The latter has a version for $\SO/\grpSp$ gauge groups, as an
$\SO/\grpSp$ instanton can be considered as a unitary instanton $A$
together with an involutive isomorphism between the dual instanton and
the original instanton $A^*\cong A$. More precisely we take ADHM
description for both $A$ and $A^*$, and assign an isomorphism between
two descriptions. This was explained already in \cite{ADHM}, and has
been well-known in gauge theory context. See \cite{MR857374} for
instantons on $\overline{\CC P}^2$. It is discussed e.g.\ in
\cite{MR3508922} in the presentation in \cite{MR1079726}.

When \cite{KN} was written, description of $\SO/\grpSp$ instantons on
ALE spaces was \emph{not} known, as the description of the dual
instanton $A^*$ was not given. 
It was mentioned in the third paragraph of \cite[p.~266]{KN}.
More precisely, the description in
\cite{KN} involves the parameter $\zeta$ for the level of the
hyperk\"ahler moment map, and we are forced to change $\zeta$ to
$-\zeta$ when we describe $A^*$.
\begin{NB}
Also a diagram automorphism is
involved, which is easy to handle.
\end{NB}%
Thus we need an isomorphism from description for $\zeta$ to one for
$-\zeta$. Such an isomorphism was found later as the reflection functor $S_{w_0}$
corresponding to the longest element $w_0$ in the finite Weyl group
\cite[\S9]{Na-reflect}.

However, the ADHM description of $\SO/\grpSp$ instantons on ALE spaces was 
\emph{not} discussed in \cite[\S9]{Na-reflect}, as the motivation there was different.
Let us explain it in this appendix.
%
%
No new input other than \cite{KN,Na-reflect} is necessary, so we just
state the result without a proof. It is also good to look at
\cite[App.~A.4]{2015arXiv150303676N} where ADHM description of
$\SO/\grpSp$ instantons on $\RR^4/\Gamma$ is explained.
It can be considered as a degenerate case of the discussion below
where the reflection functor $S_{w_0}$ becomes the identity.

\subsection{ADHM description of dual instantons}\label{subsec:ADHM}

Let $(I,H)$ be the McKay quiver for $\Gamma$, namely $I$ is the set of
isomorphism classes of irreducible representations of $\Gamma$, and
$H$ is the set of arrows where we draw $a_{ij}$ arrows from $i$ to $j$
for $a_{ij} = \dim \Hom_\Gamma(\rho_i,\rho_j\otimes Q)$, where $Q$ is
the $2$-dimensional representation of $\Gamma$ given by the inclusion
$\Gamma\subset\SU(2)$. 
Contrary to the assumption in \cref{sec:inv-quiver}, our quiver is of affine type.
However, the finite type quiver was obtained by removing the $0$-th vertex corresponding to the trivial
representation $\rho_0$ in \cref{subsec:diagram}, hence the explanation
in \cref{sec:inv-quiver} is still valid.
We choose an orientation $\Omega$ of $H$,
which is a division $H = \Omega \sqcup\overline{\Omega}$, where
$\overline{h}$ is the arrow with the opposite direction to $h$. 
We choose it as in \cref{subsec:orientation}, where we choose a cyclic orientation,
instead of a linear orientation, in affine type A.


We take $\zeta = (\zeta_i) \in (\RR^3)^I$, the data for an ALE space
$X_\zeta$ asymptotic to $\RR^4/\Gamma$ at infinity. It sits in the
level $0$ hyperplane
$0 = \zeta\cdot\delta = \sum_i \zeta_i \dim \rho_i$, where $\delta$ is
the positive primitive imaginary root of the corresponding affine Lie
algebra.

Let us take an $\U(n)$ framed instanton $A$ on $X_\zeta$. We have the
corresponding ADHM description \cite{KN}. Namely the quiver variety
$\fM_\zeta^{\mathrm{reg}}(\bv,\bw)$ is the moduli space of framed
instantons on $X_\zeta$. Here the superscript ``reg'' means that we take
the open subset of the quiver variety consisting of free orbits of the
gauge group $\prod_{i\in I} \U(V_i)$.
Furthermore, we can replace the hyperK\"ahler quotient by the complex symplectic
quotient with the stability condition given by $\zeta_\RR$, hence the quiver
variety is one discussed in \cref{sec:inv-quiver}.
($B_h$ was denoted by $B_{i,j}$, and $a_i$, $b_i$ were denoted by
$i_k$, $j_k$ in \cite{KN}.)

The description of \cite{KN} uses the tautological bundle
$\scR$, which decomposes as
$\bigoplus_{i\in I} \scR_i\otimes\rho_i^*$.
Reflection functors in \cite{Na-reflect} are understood as
isomorphisms between different descriptions of instanton moduli spaces
for different choices of $\scR$. In particular, the reflection
functor $S_{w_0}$ for the longest element $w_0$ composed with the
diagram automorphism $\invast$
\begin{equation*}
  * \circ S_{w_0} = S_{w_0}\circ * \colon
  \fM_\zeta^{\mathrm{reg}}(\bv,\bw)\to
  \fM_{w_0\invast[\zeta]}^{\mathrm{reg}}(w_0\star \invast[\bv],\invast[\bw])
\end{equation*}
corresponds to the dual bundle
$\scR^* = \bigoplus_{i\in I}\scR_i^*\otimes\rho_i$ as
shown in \cite[9(iii)]{Na-reflect}.
%
%
%
%
Note that $S_{w_0}$ and $\invast$ do \emph{not} touch the vertex $0$
corresponding to the trivial representation, hence the component
$\bv_0$, $\bw_0$ are unchanged.

The dual instanton $A^*$ is given by the transpose $t$ in \cref{subsec:transpose}
%
%
%
\begin{equation*}
  t\colon \fM_\zeta^{\mathrm{reg}}(\bv,\bw)\to
  \fM_{-\zeta}^{\mathrm{reg}}(\bv,\bw).
\end{equation*}
Recall that $\fM_{-\zeta}^{\mathrm{reg}}(\bv,\bw)$ is associated with $W^*$.

Composing two isomorphisms, we get
\begin{equation*}
  S_{w_0}\circ\invast\circ t
  \colon \fM_\zeta^{\mathrm{reg}}(\bv,\bw)
  \to \fM_\zeta^{\mathrm{reg}}(w_0\star \invast[\bv],\invast[\bw]),
\end{equation*}
where we have used $-w_0\invast[\zeta] = \zeta$.

%

\subsection{$\SO/\grpSp$ instantons}

Let us take $\SO(n)$ or $\grpSp(n)$ framed instanton $A$ on
$X_\zeta$. We regard it as a $\U(n)$ framed instanton $A$ together
with an isomorphism $A\cong A^*$ compatible with the framing.
Since framed instantons have no nontrivial automorphisms, the
isomorphism $A\cong A^*$ is unique if it exists. Hence moduli spaces
of $\SO(n)/\grpSp(n)$ framed instantons are fixed point loci in
moduli spaces of $\U(n)$ instantons with respect to the involution
given by $A\mapsto A^*$. Here moduli spaces are constructed by a gauge
theoretic method as in \cite{MR1074476}.

Recall that the framing of $A$ is an approximate isomorphism of $A$
and a flat connection on $\RR^4/\Gamma$ at infinity. The flat
connection corresponds to a $\Gamma$-module
$\bigoplus_i W_i\otimes\rho_i$.
Its dual representation is
$\bigoplus W_i^*\otimes\rho_{\invast[i]} = \bigoplus W_{\invast[i]}^* \otimes\rho_i$.
We choose an isomorphism $W_i\cong W^*_{\invast[i]}$ as in \cref{subsec:form}.
Here we choose $\sigma' = \id$, and determine $\ve_i$ as in \cref{subsec:form},
according to a choice of type $(+)$ or $(-)$.
We choose $(+)$ if we want to describe $\SO$-instantons, and $(-)$ for $\grpSp$-instantons.
The rule for the choice of $\ve_i$ is determined so that the isomorphism
$\bigoplus W_i\otimes\rho_i\cong \bigoplus W_i^*\otimes\rho_{\invast[i]}$ is given by
an orthogonal form in type $(+)$ or a symplectic form in type $(-)$. 
Thus we fix an isomorphism of the flat connection and its dual in advance.



Being an $\SO(n)$ or $\grpSp(n)$ instanton, a $\U(n)$ instanton must have
vanishing first Chern class. Hence we have $w_0\star \invast[\bv] = \bv$ by \cite[\S9]{KN}.
%
in the last paragraph of \cref{subsec:ADHM}.
%


%

Now $\sigma = S_{w_0}\circ\invast\circ t$ is defined as an involution on
$\fM_\zeta^{\mathrm{reg}}(\bv,\bw)$.
\begin{Theorem}
  A moduli space of $\SO(n)$ or $\grpSp(n)$ framed instantons on
  $X_\zeta$ is isomorphic to the fixed point locus
  $\fM_\zeta^{\mathrm{reg}}(\bv,\bw)^{\sigma}$ as a hyperk\"ahler manifold.
\end{Theorem}

Dimension vectors $\bv$, $\bw$ are determined by the second Chern
class and the framing considered as $\U(n)$ instantons.

\subsection{Extension to partial compactification}

Recall that $\fM_\zeta^{\mathrm{reg}}(\bv,\bw)$ is defined as
the open subset consisting of free $\prod \U(V_i)$-orbits. We have a
larger space $\fM_\zeta(\bv,\bw)$ by dropping the freeness
condition. It is Uhlenbeck's partial compactification of the moduli
space of framed instantons on $X_\zeta$.

The isomorphism $\sigma$ extends to
$\fM_\zeta(\bv,\bw)$ as a homeomorphism. This is because the
reflection functor is defined on the larger space, and $t$, $*$
clearly extend. It is also clear in the gauge theoretic construction
of $\fM_\zeta(\bv,\bw)$.
If we use algebro-geometric construction of
$\fM_\zeta(\bv,\bw)$ via geometric invariant theory, we can
endow $\fM_\zeta(\bv,\bw)$ with a structure of a
quasiprojective variety. Then the extension of
$\sigma$ is an involution on a variety.

Note that $\fM_\zeta^{\mathrm{reg}}(\bv,\bw)$ has another
partial compactification as a moduli space of framed torsion free
sheaves on $X_\zeta$. The corresponding ADHM description is given as
follows. (See \cite{Na-ADHM} for detail.) We take an algebro-geometric
description of $\fM_\zeta(\bv,\bw)$. We
decompose the parameter $\zeta$ to complex and real parts $\zeta_\CC$,
$\zeta_\RR$, impose the complex moment map equation involving only
$\zeta_\CC$. We have the group action of $\prod \GL(V_i)$ on the
solution space. We take the quotient of the $\zeta_\RR$-semistable
locus by the S-equivalence relation. Thus
\begin{equation*}
  \fM_\zeta(\bv,\bw) \cong H^{\mathrm{ss}}_{(\zeta_\RR,\zeta_\CC)}/\!\sim,
\end{equation*}
where $H^{\mathrm{ss}}_{(\zeta_\RR,\zeta_\CC)}$ is the open subset of
$\zeta_\RR$-semistable points in the solution space of the complex
moment map equation. See \cite[Prop.~2.11]{Na-reflect}.

Recall that $\zeta_\RR$ lives on the level 0 hyperplane
$\zeta_\RR\cdot \delta = 0$. Then we take $\zeta_\RR'$ near
$\zeta_\RR$ with $\zeta_\RR'\cdot\delta < 0$. Then we define
$\fM_{(\zeta_\CC,\zeta_\RR')}(\bv,\bw)$ as the quotient of the
$\zeta_\RR'$-semistable (equivalently $\zeta_\RR'$-stable) locus by
the action of $\prod\GL(V_i)$.
Then $\zeta_\RR'$-semistability implies $\zeta_\RR$-semistability by
our choice, that is $H^{\mathrm{ss}}_{(\zeta_\RR',\zeta_\CC)}\subset
H^{\mathrm{ss}}_{(\zeta_\RR,\zeta_\CC)}$. It induces a morphism
\begin{equation*}
  \fM_{(\zeta_\CC,\zeta_\RR')}(\bv,\bw) \to
  \fM_\zeta(\bv,\bw), 
\end{equation*}
which is an isomorphism on
$\fM_\zeta^{\mathrm{reg}}(\bv,\bw)$.  Therefore
$\fM_{(\zeta_\CC,\zeta_\RR')}(\bv,\bw)$ is another partial
compactification of $\fM_\zeta^{\mathrm{reg}}(\bv,\bw)$.

Note that $-w_0 (\zeta_\RR')^* \neq \zeta_\RR'$ as they live in the
opposite side of the level 0 hyperplane $\zeta_\RR\cdot\delta = 0$.
Therefore $\sigma$ does not define an involution
on $\fM_{(\zeta_\CC,\zeta_\RR')}(\bv,\bw)$.

\bibliographystyle{myamsalpha}
\bibliography{nakajima,mybib,orthsymp,symmetric,MO}

\end{document}